\documentclass{amsart}
\usepackage{amssymb}
\usepackage{latexsym}
\usepackage{amsmath}
\usepackage{euscript} 
\usepackage{amsmath,amssymb,amsthm,graphicx,epstopdf,mathrsfs,url, color}

      \def\dC{{\mathbb C}}

   \def\dN{{\mathbb N}}   
      \def\dR{{\mathbb R}}

\def\cS{{\EuScript S}}

\def\bm\chi{\mbox{\boldmath$\chi$}}

\def\min{{\rm min\,}}
\def\max{{\rm max\,}}
\def\ker{{\rm ker\,}}
\def\ran{{\rm ran\,}}

\def\dom{{\rm dom\,}}

\def\dim{{\rm dim\,}}

\let\xker=\ker \def\ker{{\xker\,}}

\def\sgn{{\rm sgn\,}}

\def\sgn{{\text{\rm sgn\,}}}

\def\senki{{\lbrack\negthinspace [\bot ]\negthinspace\rbrack}}
\def\senki+{{\lbrack\negthinspace [+] \negthinspace\rbrack}}

\newtheorem{theorem}{Theorem}[section]
\newtheorem{proposition}[theorem]{Proposition}
\newtheorem{corollary}[theorem]{Corollary}
\newtheorem{lemma}[theorem]{Lemma}
\theoremstyle{definition}

\newtheorem{definition}[theorem]{Definition}
\newtheorem{remark}[theorem]{Remark}

\numberwithin{equation}{section}

\begin{document}

\title[Boundary triples for Dirac operators with singular interactions]{Boundary triples and Weyl functions for Dirac operators with singular interactions}

\author[J.~Behrndt]{Jussi~Behrndt}
\author[M.~Holzmann]{Markus Holzmann}
\author[C.~Stelzer]{Christian Stelzer}
\author[G.~Stenzel]{Georg Stenzel}

\address{Institut f\"ur Angewandte Mathematik\\
Technische Universit\"at Graz \\
Steyrergasse 30\\
8010 Graz \\
Austria}
\email{behrndt@tugraz.at, holzmann@math.tugraz.at, christian.stelzer@tugraz.at, gstenzel@math.tugraz.at}

\begin{abstract}
In this article we develop a systematic approach to treat Dirac operators $A_{\eta, \tau, \lambda}$ with singular electrostatic, Lorentz scalar, and anomalous magnetic interactions of strengths $\eta, \tau, \lambda \in \mathbb{R}$, respectively, supported on points in $\mathbb{R}$, curves in $\mathbb{R}^2$, and surfaces in $\mathbb{R}^3$ that is based on boundary triples and their associated Weyl functions. First, we discuss the one-dimensional case which also serves as a motivation for the multidimensional setting. Afterwards, in the two- and three-dimensional situation we construct quasi, generalized, and ordinary boundary triples and their Weyl functions, and provide a detailed characterization of the associated Sobolev spaces, trace theorems, and the mapping properties of integral operators which play an important role in the analysis of $A_{\eta, \tau, \lambda}$. We make a substantial step towards more rough interaction supports $\Sigma$ and consider general compact Lipschitz hypersurfaces. We derive conditions for the interaction strengths such that the operators $A_{\eta, \tau, \lambda}$ are self-adjoint, obtain a Krein-type resolvent formula, and characterize the essential and discrete spectrum. These conditions include purely Lorentz scalar and purely non-critical  anomalous magnetic interactions as well as the confinement case, the latter having an important application in the mathematical description of graphene. Using a certain ordinary boundary triple, we 
also show the self-adjointness of $A_{\eta, \tau, \lambda}$ for arbitrary critical combinations of the interaction strengths under the condition that $\Sigma$ is $C^{\infty}$-smooth and derive its spectral properties. In particular, in the critical case, a loss of Sobolev regularity in the operator domain  and a possible additional point of the essential spectrum are observed. 
\end{abstract}

\maketitle

\section{Introduction} \label{section_A_0}
Dirac operators with singular perturbations supported on points in $\mathbb{R}$, curves in $\mathbb{R}^2$, and
surfaces in $\mathbb{R}^3$ have received a lot of attention recently.  This class of operators appears as an idealized model in the description of relativistic spin $1/2$ particles propagating in strongly localized external potentials. The MIT bag model describing quark confinement is an example of this sort of physical system \cite{ALTR17, Bo68, CJJT74, CJJTW74,  J75}. Dirac operators with singular perturbations  are also important in the description of graphene and other innovative materials \cite{AB08, BFSB17_1, BFSB17_2, NGM04}.  For a rigorous understanding of these physical problems  a detailed analysis of the involved operators and their spectral properties is necessary.

\subsection{The free Dirac equation}
In 1928 Paul Dirac derived in \cite{D28} the free Dirac equation to find a quantum mechanical description of spin 1/2 particles with mass $m \geq 0$ propagating in $\dR^{q}$ with $q \in \{1,2,3\}$ taking relativistic effects into account. 
Choosing units such that the speed of light and Planck's constant are both equal to one yields the relativistic energy-momentum relationship  
 \begin{equation} \label{energy_momentum_rel} 
E^2 =  \sum\limits_{j=1}^q  p_j^2 + m^2.
\end{equation}
To associate it with a linear differential equation, which is of first order in time, Dirac made the ansatz
\begin{equation} \label{LinearisedDiracEq}
\bigg( E -  \sum\limits_{j=1}^{q} \alpha_j p_j - m \alpha_{0} \bigg)\bigg( E +  \sum\limits_{j=1}^{q} \alpha_j p_j + m \alpha_{0} \bigg) = 0.
\end{equation}
The equations~\eqref{energy_momentum_rel}  and~\eqref{LinearisedDiracEq} are only equivalent, if the coefficients $\alpha_j$ satisfy the anti-commutation relations 
\begin{equation} \label{anti_commutation}
\alpha_k \alpha_j + \alpha_j \alpha_k = 2 \delta_{kj} I_N \quad \text{for all} \quad k , j \in \{ 0, 1, \dots , q\},
\end{equation}
and hence one has to choose the coefficients matrix-valued, $\alpha_j  \in \dC^{N \times N}$, where $N=2$ for $q \in \{1,2\}$ and $N=4$ for $q=3$, and 
$I_N$ denotes the $N\times N$-identity matrix.
For $q \in\{1,2\}$ the matrices $\alpha_j$ can be chosen as the Pauli spin matrices
and for $q=3$ as the Dirac matrices, see \eqref{Dirac_matrices1} and \eqref{Dirac_matrices2}.
Applying the usual substitution rules $i \frac{\partial}{\partial t}$ and $- i \frac{\partial}{\partial x_j}$ for $E$ and $p_j$, respectively,  in one of the factors in  (\ref{LinearisedDiracEq}), say in the first one, one obtains the free Dirac equation
\begin{equation*}
i  \frac{\partial}{\partial t} \Psi =  \bigg(- i \sum\limits_{j=1}^{q} \alpha_j \frac{\partial}{\partial x_j} + m \alpha_{0} \bigg) \Psi,
\end{equation*}
which describes the propagation of a relativistic particle with spin $1/2$ such as an electron in $\dR^q$. 

Besides this original motivation another important application of the  Dirac equation was found recently
in \cite{AB08, NGM04} where it was demonstrated that the two-dimensional Dirac equation plays an important role in the mathematical description of graphene. This fact and the importance of graphene in many modern applications led to an increased interest in the Dirac equation in the last 15 years. In particular, depending on the way how the graphene is cut out of the planar sheet Dirac operators with different boundary conditions known as quantum dot boundary conditions are of interest and were studied in \cite{BHM20, BFSB17_1, BFSB17_2, R20}. Since graphene is a hexagonal lattice, it is of importance that the corresponding operators are investigated on rough domains with Lipschitz boundaries.

\subsection{The free Dirac operator}
As in the case of the Schrödinger equation, one defines the free Dirac operator as the right-hand side of the free Dirac equation by
\begin{equation} \label{def_A_0}
\begin{split}
A_0 f &= \big( - i (\alpha \cdot \nabla) + m \alpha_{0} \big) f, \qquad
\text{dom}\, A_0 = H^1(\dR^q ; \dC^N), 
\end{split}
\end{equation}  
where $H^k(\mathbb{R}^q; \mathbb{C}^N)$ denotes the ($L^2$-based) Sobolev space of $k$-times weakly differentiable functions and where we use for $x=(x_1, \dots, x_q) \in \mathbb{R}^q$ the notations 
\begin{equation*}
\alpha \cdot x := \sum_{j=1}^q \alpha_j x_j \quad \text{and} \quad \alpha \cdot \nabla := \sum_{j=1}^q \alpha_j \frac{\partial}{\partial x_j}.
\end{equation*}
With the help of the Fourier transform and \eqref{anti_commutation} it is not difficult to verify that $A_0$ is self-adjoint in $L^2(\dR^q ; \dC^N)$, that $A_0$ has purely absolutely continuous spectrum
\begin{equation}\label{spec}
\sigma(A_0) = \sigma_\textup{ac}(A_0) = (-\infty, -m] \cup [m , \infty),
\end{equation}
and that
\begin{equation} \label{equation_Dirac_Laplace}
(A_0 - z) ( A_0 + z) = \left( - \Delta + m^2 - z^2 \right) I_N,
\end{equation}
where $-\Delta$ is the free Laplace operator defined on $\dom (-\Delta)=H^2(\mathbb{R}^q; \mathbb{C})$; see \cite{T92} or \cite{W03} for details in dimension $q=3$.  This implies
\begin{equation} \label{resolvent_A_0}
(A_0-z)^{-1} = \big( -i (\alpha \cdot \nabla) + m \alpha_0 + z I_N \big) (-\Delta + m^2-z^2)^{-1} I_N
\end{equation}
for $z\in\mathbb (\mathbb{C}\setminus\mathbb R) \cup (-m,m)$.
Using the well known form of the resolvent of $-\Delta$, one finds that $(A_0-z)^{-1}$ is an integral operator in $L^2(\mathbb{R}^q; \mathbb{C}^N)$ with integral kernel $G_{z, q}(x-y)$ given by \eqref{def_G_lambda}. 
\subsection{Coupling of singular potentials}
To model the influence of external fields the free Dirac operator $A_0$ is coupled with potentials. In the following, we consider the physically relevant dimensions $q = 2,3$; motivated by this, similar mathematical expressions are studied also for $q=1$. For a  scalar potential $\Phi_s$ and a static electromagnetic potential  $( \Phi_{e}, A_m)$ taking also anomalous magnetic interactions with coupling constant $\mu$ into account,  the corresponding  Dirac operator is given by
\begin{equation} \label{ExternalFieldsNotDelta}
A_0 +  \Phi_{e} I_N +  \Phi_{s}  \alpha_0  - (\alpha \cdot A_m)  + i \mu (\alpha \cdot \nabla \Phi_{e} ) \alpha_0 + \mu V_q,
\end{equation}
where the term $V_q$ that contributes to the description of the anomalous magnetic moment depends on the space dimension $q$ and is given by $V_2 = (\nabla \times A_m) I_2$ for $q=2$ and $V_3 =  \alpha \cdot ( \nabla \times A_m ) \alpha_0 \alpha_1$ for $q=3$;
cf.  \cite{H90, T92}.

In the case of strongly localized potentials that have large values on a small neighbourhood of a set $\Sigma \subset \dR^q$ with measure zero and small values otherwise, it is often a useful simplification  to  replace the regular fields  by  $\delta$-potentials supported on $\Sigma$. To be more precise, assume in the following that $\Sigma$ is the boundary of a bounded sufficiently smooth domain $\Omega_+ \subset \mathbb{R}^q$, $q \in \{2,3\}$, with normal vector $\nu$ pointing outwards of $\Omega_+$ and set $\Omega_- := \mathbb{R}^q \setminus \overline{\Omega_+}$. Then, each point $x \in \mathbb{R}^q$ in a tubular neighborhood of $\Sigma$ (which are all points $x$ with $\textup{dist}(x, \Sigma) < b$) with a sufficiently small half-width $b>0$ has the form $x = x_\Sigma + t \nu(x_\Sigma)$ with unique $x_\Sigma \in \Sigma$ and $t \in (-b,b)$. Consider now static electromagnetic and Lorentz scalar potentials whose strengths depend only on the distance to $\Sigma$, that are  supported in the tubular neighborhood of $\Sigma$, and where the magnetic field is pointing in the normal direction, i.e.
\begin{equation*}
  \Phi_s(x) = f_s(t), \quad \Phi_e(x) = f_e(t), \quad \text{and }  A_m(x) = \nu(x_\Sigma) f_m(t), \quad x = x_\Sigma + t \nu(x_\Sigma),
\end{equation*}
with smooth scalar-valued functions $f_s, f_e$, and $f_m$ supported in $(-b,b)$. For such potentials 
one verifies $\alpha \cdot \nabla \Phi_{e} = ( \alpha \cdot \nu) f'_e$ and $V_q = 0$, and hence
the expression in~\eqref{ExternalFieldsNotDelta} simplifies to
\begin{equation*}
A_0 + f_e I_N + f_s \alpha_0 - (\alpha \cdot \nu) f_{m} + i \mu ( \alpha \cdot \nu) \alpha_0 f'_e.
\end{equation*}
For suitable choices of $f_e, f_s, f_m$, or $\mu f_e'$ one can replace the term $f_e I_N$ by $\eta I_N \delta_\Sigma$, $f_s \alpha_0$ by $\tau \alpha_0 \delta_\Sigma$,  $(\alpha \cdot \nu) f_{m}$ by $\omega (\alpha \cdot \nu) \delta_\Sigma$, or $i \mu ( \alpha \cdot \nu) \alpha_0 f'_e$ by $i \lambda ( \alpha \cdot \nu) \alpha_0 \big) \delta_{\Sigma}$ with  interaction strengths $\eta, \tau, \omega, \lambda \in \dR$. From a mathematical point of view it makes sense to study Dirac operators with all of the above singular potentials simultaneously, so we arrive at the expression
\begin{equation*} 
A_0 + \big( \eta I_N + \tau \alpha_0 - \omega (\alpha \cdot \nu) + i \lambda ( \alpha \cdot \nu) \alpha_0 \big) \delta_{\Sigma};
\end{equation*}
cf. \cite{AMV14, AMV15, BEHL19, Ben21, CLMT21}. The latter replacement was justified rigorously by approximation procedures in \cite{CLMT21, MP18a, MP16, S89, T20}.
We note that, without loss of generality, one can assume $\omega=0$, since this term can be gauged away by an appropriate unitary transformation, see \cite[Theorem~2.1]{CLMT21} or \cite[Theorems~1.2 and~1.3]{M17}. This finally leads to the 
main object of interest in the present paper: For interaction strengths $\eta, \tau, \lambda \in \dR$ we will consider Dirac operators $A_{\eta, \tau, \lambda}$ 
with singular perturbations formally given by 
\begin{equation}  \label{Diracpertubed}
A_{\eta, \tau, \lambda} = A_0 + \big( \eta I_N + \tau \alpha_0 + i \lambda ( \alpha \cdot \nu) \alpha_0 \big) \delta_{\Sigma};
\end{equation}
here we call $\eta I_N \delta_\Sigma$ the electrostatic, $\tau \alpha_0 \delta_\Sigma$ the Lorentz scalar, and $i \lambda ( \alpha \cdot \nu) \alpha_0 \delta_\Sigma$ the anomalous magnetic $\delta$-shell interaction, respectively.
%

The ``action'' of the $\delta$-potential on functions that are sufficiently smooth in $\Omega_\pm$ such that they admit Dirichlet traces $\gamma_D^{\pm} f_{\pm}$, $f_\pm := f \upharpoonright \Omega_\pm$, on $\Sigma$ in a suitable sense is defined as the anti-linear distribution
\begin{equation*}
\langle \delta_{ \Sigma } f , \varphi \rangle = \frac{1}{2} \int_{\Sigma} \big(\gamma_D^{+} f_{+} + \gamma_D^{-} f_{-}  \big) \overline{\varphi} d \sigma, \quad \varphi \in C_0^\infty(\mathbb{R}^q).
\end{equation*}
A formal integration by parts shows that the distribution $A_{\eta, \tau, \lambda} f$ is regular if and only if $(\alpha \cdot \nabla) f_\pm \in L^2(\Omega_\pm;\dC^N)$ and $f$ satisfies the jump condition
\begin{equation} \label{jump_condition}
  i (\alpha \cdot \nu) (\gamma_D^+ f_+ - \gamma_D^- f_-) + \frac{1}{2} \big( \eta I_N + \tau \alpha_0 + i \lambda ( \alpha \cdot \nu) \alpha_0 \big) (\gamma_D^+ f_+ + \gamma_D^- f_-) = 0;
\end{equation}
cf. \cite{AMV14, BEHL19}. In a similar way, one finds in space dimension $q=1$  with $\Sigma = \{ 0 \}$, $\Omega_+=(0,\infty)$, $\Omega_- = (-\infty, 0)$, and $\nu = -1$ that $A_{\eta, \tau, \lambda} f$ is regular if and only if $f_\pm \in H^1(\Omega_\pm;\dC^2)$ and $f$ fulfills
\begin{equation} \label{jump_condition1D}
  -i \alpha_1 (f(0+) - f(0-)) + \frac{1}{2} \big( \eta I_2 + \tau \alpha_0 - i \lambda \alpha_1 \alpha_0 \big) (f(0+) + f(0-)) = 0.
\end{equation}
Hence, we will define $A_{\eta, \tau, \lambda}$ on the set of functions that satisfy~\eqref{jump_condition} or \eqref{jump_condition1D}. To do this in a mathematically rigorous way, it is necessary to study Sobolev spaces that are suitable for the investigation of Dirac operators and to provide trace theorems on them. It is one goal in this paper to do this also in the case that $\Sigma$ is only a Lipschitz smooth curve in $\mathbb{R}^2$ or surface in $\mathbb{R}^3$.

The non-relativistic counterpart of~\eqref{Diracpertubed}, Schr\"odinger operators with singular $\delta$-potentials, were studied intensively and led to the discovery of several interesting spectral effects \cite{BEL14_JPA, E08, EI01, EK03, EY01}. While the replacement of regular potentials by $\delta$-potentials may simplify the spectral analysis of the Hamiltonian considerably, it leads to technical  difficulties in its mathematically rigorous definition due to the appearance of jump conditions in the operator domain in a similar vein as in~\eqref{jump_condition}. One natural way to introduce and study Schr\"odinger operators with singular interactions is via the corresponding quadratic form \cite{BEKS94,E08,EK15,H89,SV96}, but also extension theoretic approaches and in particular boundary triple techniques were successfully applied to investigate these models \cite{BEHL20,BLL13,BLLR18,MPS16}. Form methods are not suitable to study Dirac operators with singular interactions due to a lack of semi-boundedness. However, it seems natural to use boundary triple methods for this purpose.

\subsection{Boundary triples and Weyl functions}\label{subsub}

Assume that  $S$ is a densely defined closed symmetric operator in a Hilbert space  $\mathcal{H}$. Then an (ordinary) boundary triple $\{ \mathcal{G}, \Gamma_0 , \Gamma_1\}$ consists of a Hilbert space  $\mathcal{G}$ and two linear mappings $\Gamma_0, \Gamma_1 : \dom S^{\ast} \to \mathcal{G}$ such that 
$(\Gamma_0 , \Gamma_1)^\top : \dom S^{\ast} \to \mathcal{G} \times \mathcal{G}$ is surjective and the so-called ``abstract Green's identity'' 
\begin{equation}\label{absgreen} 
  ( S^{\ast} f,g)_{\mathcal H}-(f,S^{\ast} g)_{\mathcal H}=(\Gamma_1 f,\Gamma_0 g)_{\mathcal G}-(\Gamma_0 f,\Gamma_1 g)_{\mathcal G}
\end{equation}
is valid for all $f,g \in \dom S^{\ast}$.  To a boundary triple one associates a self-adjoint reference operator $A_0 := S^* \upharpoonright \ker \Gamma_0$ in $\mathcal H$ and the Weyl function $z\mapsto M(z)$ via 
$M(z)\Gamma_0 f_z=\Gamma_1 f_z$ for $f_z\in\ker(S^\ast -z)$, $z\in\rho(A_0)$.
We note that  the Weyl function is a Nevanlinna function whose values are bounded operators in $\mathcal{G}$.
Extensions of $S$ are now defined as restrictions of $S^\ast$ by
$$A_{\vartheta} = S^{\ast} \upharpoonright \ker( \Gamma_1 - \vartheta \Gamma_0 ),$$ 
where the parameter $\vartheta$  is a linear operator (or relation) in $\mathcal{G}$. Thus, the functions $f\in\dom A_\vartheta$ satisfy the abstract boundary condition $\Gamma_1 f = \vartheta \Gamma_0 f$. 
In the context of ordinary boundary triples it turns out that $A_{\vartheta}$  is self-adjoint in  $\mathcal{H}$  if and only if $\vartheta$  is self-adjoint in $\mathcal{G}$ and, moreover, 
one has an explicit Krein-type formula for the resolvent of $A_\vartheta$
and 
the spectral properties of $A_\vartheta$ are encoded in the parameter $\vartheta$ and the Weyl function $M$. Setting formally $B=-\vartheta^{-1}$ the parametrization of the  extensions of $S$ can also be given in the form 
$$A_{[B]} = S^{\ast} \upharpoonright \ker( \Gamma_0 + B \Gamma_1 ),$$ 
the other aspects 
of the theory translate in a straightforward manner; we refer to the monograph \cite{BHS20} and also Section~\ref{section_boundary_triples} for more details.

In applications to partial differential operators it is natural and 
desirable to identify $\mathcal{G}$ as an $L^2$-space on the boundary or interface
and to choose the boundary maps $\Gamma_0$ and $\Gamma_1$ as appropriate trace maps,
e.g. Dirichlet and Neumann traces in the case of second order elliptic operators,
which also implies that the associated Weyl function is the Dirichlet-to-Neumann map.
However, it turns out that the values of the trace maps on the maximal domain 
of a partial differential operator (which is typically $\dom S^*$) are often 
too singular and the classical Green's identity can not be extended onto the full maximal domain.
One way to overcome these problems is to generalize the notion of ordinary boundary
triples by defining $\Gamma_0$ and $\Gamma_1$  and requiring  \eqref{absgreen} only on a core of $S^*$.
This leads to the notion of quasi boundary triples proposed in \cite{BL07}, which 
is tailormade for extension problems involving partial differential operators. We refer also to
\cite{BL12,BLLR18,BM14} and \cite{DHM20,DHM22,DHMS06,DHMS09,DHMS12,DM95} 
for generalized boundary triples and other extensions of the notion 
of boundary triples and their Weyl functions. While there is no one to one correspondence of self-adjoint parameters $\vartheta$ and extensions $A_\vartheta$ for quasi and generalized boundary triples, there are other criteria that are useful to show the self-adjointness and spectral properties of $A_\vartheta$, see Section~\ref{section_boundary_triples} for details.

In the present situation of Dirac operators the adjoint operator $S^*$ is given as the unperturbed Dirac operator 
$- i (\alpha \cdot \nabla) + m \alpha_{0}$
on $\Omega_\pm$, defined for all $L^2$-functions $f=f_+\oplus f_-$ that satisfy $(\alpha \cdot \nabla)f_\pm \in L^2(\Omega_\pm;\dR^N)$.  The crucial jump condition \eqref{jump_condition} for the definition of the Dirac operator $A_{\eta,\tau,\lambda}$ can be fit in the
abstract approach of quasi boundary triples or generalized boundary triples 
by using the boundary mappings
\begin{equation} \label{def_Gammaint}
  \Gamma_0 f = i (\alpha \cdot \nu) (\gamma_D^+ f_+ -  \gamma_D^- f_-) \quad \text{and} \quad \Gamma_1 f = \frac{1}{2} (\gamma_D^+ f_+ + \gamma_D^- f_-)
\end{equation}
defined on the core $\dom S^* \cap H^s(\dR^q \setminus \Sigma;\dC^N)$ for $s \in [\tfrac{1}{2},1]$. Furthermore, the matrix $B= \eta I_N + \tau \alpha_0 + \lambda i (\alpha \cdot \nu) \alpha_0$ acts as the parameter
in the boundary space $L^2(\Sigma; \mathbb{C}^N)$. Hence, 
the condition \eqref{jump_condition} is satisfied by a function $f\in\dom S^*\cap H^s(\dR^q \setminus \Sigma;\dC^N)$ if and only if 
$f\in\ker (\Gamma_0 + B\Gamma_1)$. This leads to a rigorous definition of the operator 
in \eqref{Diracpertubed} as
\begin{equation}\label{Diracperturbed2}
\begin{split}
 A_{\eta, \tau, \lambda} f &=
               \big(- i (\alpha \cdot \nabla)f_+  + m \alpha_{0} f_+ \big) \oplus \big(
               - i (\alpha \cdot \nabla)f_-  + m \alpha_{0} f_-
                            \big),\\
                            \dom A_{\eta, \tau, \lambda}  & =  \ker (\Gamma_0  + B\Gamma_1 ).
\end{split}
\end{equation}
Depending on the Sobolev regularity of 
the core $\dom S^* \cap H^s(\dR^q \setminus \Sigma;\dC^N)$ the triple $\{ L^2(\Sigma; \mathbb{C}^N), \Gamma_0, \Gamma_1\}$ is a quasi boundary triple for $s\in (\tfrac{1}{2},1]$
and a generalized 
boundary triple in the case $s=\tfrac{1}{2}$. Note that the convenient choice of the parameter $s$ is related to the smoothness of $\Sigma$: while for smooth $\Sigma$ the value $s=1$ is most suitable, the smaller value $s=\frac{1}{2}$ is particularly useful for the case of Lipschitz smooth $\Sigma$, which is considered in this paper.
For completeness we mention that
in the one-dimensional case matters simplify and the one-dimensional analogue of the boundary maps in \eqref{def_Gammaint}, that is, 
\begin{equation} \label{boundary_mappings_intro}
  \Gamma_0 f = -i \alpha_1 (f(0+) -  f(0-)) \quad \text{and} \quad \Gamma_1 f = \frac{1}{2} (f(0+) +  f(0-)),
\end{equation}
leads to an ordinary boundary triple 
$\{ \mathbb C^2, \Gamma_0, \Gamma_1\}$.

\subsection{Dirac operators and extension theory}

Using boundary triples and closely related techniques, several advances were made in the recent years in the study of the operator $A_{\eta, \tau, \lambda}$ in~\eqref{Diracperturbed2}. It turned out for sufficiently smooth and compact interaction supports $\Sigma$ that the operator with interaction strengths $\eta,\tau,\lambda\in\mathbb R$ such that 
\begin{equation} \label{non_critical_intro}
  \left( \frac{\eta^2-\tau^2-\lambda^2}{4}-1 \right)^2 - \lambda^2 \neq 0,
\end{equation}
which is referred to as {\it non-critical case}, has properties as one may expect them from the model. The essential spectrum coincides with $\sigma(A_0)$ in~\eqref{spec}, there are at most finitely many discrete eigenvalues in the gap $(-m,m)$ of $\sigma(A_0)$, and functions in $\dom A_{\eta, \tau, \lambda}$ have $H^1$-smoothness in $\mathbb{R}^q \setminus \Sigma$. However, in the {\it critical case} there is a lack of Sobolev regularity in $\dom A_{\eta, \tau, \lambda}$ and there may be additional points or even intervals in the essential spectrum. Furthermore, it was observed that $A_{\eta, \tau, \lambda}$ decouples into the orthogonal sum of Dirac operators in $L^2(\Omega_\pm; \mathbb{C}^N)$ with certain 
boundary conditions on $\Sigma$ if $\eta^2 - \tau^2 - \lambda^2 = -4$; cf. \cite{AMV15, BHOP20, CLMT21, DES89} and Section~\ref{section_confinement}. This phenomenon is known as confinement and means that the $\delta$-potential is, in this case, impenetrable for the particle. Now, the above mentioned boundary conditions are exactly the quantum dot boundary conditions that appear in the analysis of graphene
and hence, properties of $A_{\eta, \tau, \lambda}$ are of interest in the study of graphene and similar materials.

Historically, $A_{\eta, \tau, \lambda}$ was first studied in the one-dimensional case in \cite{GS87} and then further investigated in \cite{AGHH05, BD94, BMP17, BMP18, CMP13, HT22, H99, PR14}. Using these results and a decomposition to spherical harmonics, $A_{\eta, \tau, \lambda}$ was studied in the three-dimensional setting for general interaction strengths and $\Sigma$ being the sphere in \cite{DES89}. Then, after 25 years  with only little progress in the investigation of $A_{\eta, \tau, \lambda}$ in higher dimensions, an extension theoretic approach that is closely related to boundary triple techniques was used in \cite{AMV14} to study the case of non-critical purely electrostatic interactions 
($\tau=\lambda=0$)
supported on the boundary of a bounded $C^2$-domain in dimension three; with the same approach the authors continued their work in \cite{AMV15, AMV16}  and studied several spectral properties, Lorentz scalar interactions, and the confinement case. 
In \cite{BEHL18} quasi boundary triple techniques were used for a qualitative spectral analysis of $A_{\eta, 0, 0}$ for interactions supported on smooth compact surfaces in $\mathbb{R}^3$ in the non-critical case; see also \cite{BEHL19, BHMP20, HOP18, R22a, R22b} for further spectral and scattering properties of $A_{\eta, \tau, \lambda}$ in dimension three for compact surfaces $\Sigma$ and non-critical interactions with $\lambda=0$. 
The case of critical electrostatic interactions was first handled in dimension three in \cite{BH20, OV16}, where an ordinary boundary triple was used in \cite{BH20} and a closely related idea in \cite{OV16}. In particular, it was shown in \cite{BH20} that $0$ may appear as an additional point in the essential spectrum and that there is a loss of Sobolev regularity in the operator domain in the critical case under additional geometrical assumptions.
In dimension two, the self-adjointness and spectral properties of $A_{\eta, \tau, \lambda}$ for $\lambda=0$ and general $C^\infty$-smooth closed curves $\Sigma$ were studied first in \cite{BHOP20} -- here both the critical and the non-critical case were considered and it was shown that for critical interactions $-m \frac{\tau}{\eta}$ belongs to the essential spectrum. In \cite{Ben21, CLMT21} the magnetic interactions were introduced and the self-adjointness and spectral properties were discussed also in the critical purely anomalous magnetic case ($\eta=\tau=0$). In \cite{Ben21} also the self-adjointness of $A_{\eta, \tau, 0}$ with critical combinations of electrostatic and Lorentz scalar interactions supported on $C^2$-smooth surfaces in $\mathbb{R}^3$ was shown, see also \cite{BP22} for a recent contribution to the study of the essential spectrum in this situation with smooth interaction supports.
The approximation of $A_{\eta, \tau, \lambda}$ by Dirac operators with scaled regular potentials under different assumptions is treated in \cite{BHT22b, CLMT21, MP18a, MP16, S89, T20}; surprisingly, a nontrivial rescaling of the coupling constants appears which was related to Klein's paradox. 
Unbounded interaction supports were investigated in \cite{BHT22a, BHT22b, Ben21, FL22, R21, R22a, R22b}. In this case, spectral transitions were observed in the sense that bands of continuous spectrum abruptly change to a single point in the essential spectrum, when the interaction strengths pass in certain configurations from non-critical to critical ones. Non-smooth interaction supports were treated in \cite{Ben22, FL22, PvdB20}. Eventually, we mention that with similar techniques boundary value problems for Dirac operators appearing in the analysis of graphene were directly studied in \cite{ALTR17, BHM20, BFSB17_1, BFSB17_2, CL20, EH22, H21, LTO18, S95} and similar effects as for $A_{\eta, \tau, \lambda}$ were observed.

\subsection{Content of this paper}

It is our main goal to develop a systematic extension theoretic approach to study the self-adjointness and the spectral properties of Dirac operators with singular interactions supported on a point in $\mathbb{R}$, compact and closed curves in $\mathbb{R}^2$ and surfaces in $\mathbb{R}^3$. In particular, we describe self-adjoint realizations of $A_{\eta,\tau, \lambda}$ also for Lipschitz smooth interaction supports and study $A_{\eta, \tau, \lambda}$ with critical interactions involving all three parameters $\eta, \tau, \lambda$ supported on smooth curves and surfaces $\Sigma$. 

As a starting point, we follow the approach in \cite{PR14} and consider the one-dimen\-sio\-nal case. It will be shown that the boundary maps in \eqref{boundary_mappings_intro}
give rise to an ordinary boundary triple $\{\mathbb C^2,\Gamma_0,\Gamma_1\}$
and the corresponding $\gamma$-field and Weyl function are computed; as 
a consequence we obtain the self-adjointness and spectral properties 
of $A_{\eta, \tau, \lambda}$.

The construction of boundary mappings and the investigation of the 
Dirac operators $A_{\eta, \tau, \lambda}$ in the two- and
three-dimensional setting is more involved and requires a careful analysis of the 
underlying function spaces, trace maps, various integral operators, and their mapping properties. 
Indeed, it is one goal in this article to investigate these objects 
also in the case that the interaction support is only a compact and closed Lipschitz smooth curve in $\mathbb{R}^2$ or surface in $\mathbb{R}^3$; cf. Sections~\ref{subsec_SobolevSpaces}--\ref{section_QBT}. 
More precisely, Sobolev spaces associated with the Dirac operator and 
the mapping properties of the Dirichlet trace in this context are necessary in our analysis \cite{BH20, BHOP20, BFSB17_1, OV16}.
In particular, it will turn out that for $s\in[0,1]$ functions $f \in H^s(\Omega_\pm; \mathbb{C}^N)$ with $(\alpha \cdot \nabla) f \in L^2(\Omega_\pm; \mathbb{C}^N)$ in the distributional sense admit Dirichlet traces in suitable Sobolev spaces on $\Sigma$ -- this is essential for the rigorous interpretation of the jump condition \eqref{jump_condition} and the definition of the boundary maps in \eqref{def_Gammaint}
and, thus, leads to a quasi or generalized boundary triple 
$\{ L^2(\Sigma; \mathbb{C}^N), \Gamma_0, \Gamma_1\}$
as already indicated 
in Section~\ref{subsub}, see Sections~\ref{subsec_SobolevSpaces} and \ref{section_QBT} for more details. In the further investigations,
a potential operator $\Phi_z$ and a strongly singular boundary integral operator $\mathcal{C}_z$ associated with the fundamental solution of the unperturbed Dirac equation play an important role, see also \cite{AMV14, AMV15, BH20, BHM20, BHOP20, Ben22, OV16}, where various properties of these objects were studied before in different situations. The map $\mathcal{C}_z$ is closely related to the Cauchy transform on $\Sigma$ in dimension two and the Riesz transform on $\Sigma$ in dimension three; it will
be identified as the Weyl function corresponding to $\{ L^2(\Sigma; \mathbb{C}^N), \Gamma_0, \Gamma_1\}$ 
and hence it will play a key role in the analysis 
of $A_{\eta, \tau, \lambda}$. More detailed properties of $\mathcal{C}_z$ that are needed for the study of $A_{\eta, \tau, \lambda}$ when $\Sigma$ is Lipschitz smooth are given in Section~\ref{section_Lipschitz} and when $\Sigma$ is $C^\infty$-smooth in Section~\ref{critical-section}.

Using the quasi boundary triples and generalized boundary triples from Section~\ref{section_QBT} 
we study the self-adjointness and the spectrum of $A_{\eta, \tau, \lambda}$ for non-critical interaction strengths in the case that $\Sigma$ is a Lipschitz smooth closed and compact curve in $\mathbb{R}^2$ or surface in $\mathbb{R}^3$. 
When we fix the smoothness index $s=\tfrac{1}{2}$ then
$\{ L^2(\Sigma; \mathbb{C}^N), \Gamma_0, \Gamma_1\}$ in~\eqref{def_Gammaint} becomes a generalized 
boundary triple and $A_{\eta, \tau, \lambda}$ can be rigorously defined as in \eqref{Diracperturbed2} 
acting as the unperturbed Dirac operator in $\Omega_\pm$, while its domain consists of those functions $f$ that have $H^{1/2}$-smoothness in $\Omega_\pm$, such that $(\alpha \cdot \nabla) f_\pm \in L^2(\Omega_\pm; \mathbb{C}^N)$ in the distributional sense, and that satisfy~\eqref{jump_condition}. We remark that one can not expect a better Sobolev regularity in the domain of $A_{\eta, \tau, \lambda}$, as, e.g., $\dom A_{\eta, \tau, \lambda} \subset H^1(\mathbb{R}^q \setminus \Sigma; \mathbb{C}^N)$, as it is known for special geometries that this does not hold; cf. \cite[Theorem~1.2~(ii)]{LTO18}.
In Section~\ref{section_Lipschitz} we derive  conditions on $\eta, \tau, \lambda$ under those we can describe self-adjoint realizations of $A_{\eta, \tau, \lambda}$, see Theorem~\ref{selfadjoint_theorem} for details. 
While we can not show the self-adjointness for all non-critical interaction strengths, the above mentioned conditions generalize known ones from \cite{Ben22} and include important cases as purely Lorentz scalar and non-critical purely anomalous magnetic interactions and the non-critical confinement case. For at least $C^1$-smooth $\Sigma$ our conditions reduce to all non-critical interaction strengths in~\eqref{non_critical_intro}. Moreover, we show under the above mentioned conditions that the essential spectrum of $A_{\eta, \tau, \lambda}$ is
\begin{equation*}
  \sigma_\textup{ess}(A_{\eta, \tau, \lambda})=(-\infty, -m] \cup[m, \infty)
\end{equation*}
and that there are at most finitely many discrete eigenvalues in  $(-m,m)$.
Eventually, we study in Section~\ref{section_confinement} the confinement case with non-critical and critical interaction strengths in more detail. In particular, we describe self-adjoint realizations and study the spectral properties for all combinations of the coupling constants in the confinement case for arbitrary bounded Lipschitz smooth curves in $\mathbb{R}^2$ and surfaces in $\mathbb{R}^3$. Since graphene is a hexagonal lattice, the associated differential operators have to be studied on non-smooth domains and thus, our results also contribute to the study of graphene and similar materials.

In order to handle critical interaction strengths as well, we construct in Section~\ref{section_OBT} in a similar way as in \cite{BH20, BHOP20, CLMT21} an ordinary boundary triple under the more general assumption of a Lipschitz smooth interaction support. The description of $A_{\eta, \tau, \lambda}$ via this ordinary boundary triple is not as natural as via the above mentioned quasi or generalized boundary triple and the corresponding parameter is an unbounded operator involving also the boundary integral operator $\mathcal{C}_z$ for $z=0$. To show the self-adjointness and study the spectral properties of $A_{\eta, \tau, \lambda}$ also for critical interaction strengths even a more detailed analysis of $\mathcal{C}_z$ is necessary. 
For this we have to assume additionally that $\Sigma$ is $C^\infty$-smooth. In this case we show in Theorem~\ref{theorem_self_adjoint_critical} for all interaction strengths the self-adjointness of $A_{\eta, \tau, \lambda}$. In particular, it turns out that in the critical case there is a loss of Sobolev regularity in the domain of $A_{\eta, \tau, \lambda}$: while in the non-critical case it is contained in $H^1(\mathbb{R}^q \setminus \Sigma; \mathbb{C}^N)$, which is for smooth $\Sigma$ better than for Lipschitz smooth interaction supports as discussed above, one has $\dom A_{\eta, \tau, \lambda} \not\subset H^s(\mathbb{R}^q \setminus \Sigma; \mathbb{C}^N)$ for any $s>0$ in the critical case. 
While the spectral properties of $A_{\eta, \tau, \lambda}$ are known in the non-critical case from the considerations for  less smooth $\Sigma$, we study them in the critical case in Theorem~\ref{theorem_essential_spectrum}. There we prove that also for critical interaction strengths involving all three parameters $\eta, \tau, \lambda$ the point $-m \frac{\tau}{\eta}$ may belong to the essential spectrum of $A_{\eta, \tau, \lambda}$. While the latter is always true in dimension two, an additional geometric assumption is required in dimension three in the present paper
(we also refer to \cite{BP22} for a very recent related result on additional intervals in the essential spectrum 
for the case $\lambda=0$).
This is another step forward in the analysis of $A_{\eta, \tau, \lambda}$ for general critical interaction strengths, which were only analyzed before when $\lambda=0$ or when $\eta=\tau=0$ and $\lambda^2=4$.

\subsection{Structure of the paper}

In Section  \ref{section_boundary_triples} we briefly recall the definition
of boundary triples and their Weyl functions from abstract extension theory 
and provide some related results that will be needed in the further course. 
As a warm up, we treat one-dimensional Dirac operators in Section \ref{Sec_OneDimDirac}.  In Section \ref{section_3d} we construct quasi and ordinary boundary triples and characterize the necessary function spaces and integral operators.  Then, in Section \ref{section_Lipschitz}, we study Dirac operators with constant electrostatic, Lorentz scalar, and anomalous magnetic $\delta$-interactions on Lipschitz hypersurfaces in the case of non-critical interaction strengths and discuss in Section \ref{section_confinement} the confinement case. Finally, in Section  \ref{critical-section}, we study critical interaction strengths in the case of smooth curves and surfaces as interaction supports.   \vspace{3mm}

 \subsection*{Notations:}  Unless stated otherwise, we will always assume that $\Omega_{+} \subseteq \dR^q$, $q \in \{2,3\}$, is a bounded  Lipschitz domain 
and $\Omega_{-} = \dR^q \setminus \overline{\Omega_{+}}$ is the corresponding exterior domain with boundary $\Sigma = \partial \Omega_{-} = \partial \Omega_{+}$.  The unit normal vector field on $\Sigma$ pointing outwards of $\Omega_+$ is denoted by $\nu$.

Let $q \in \{1, 2, 3 \}$.
For $s \in \dR$ the spaces $H^s(\dR^q; \dC^N)$, $H^s(\Omega_{\pm}; \dC^N)$, and $H^s(\Sigma; \dC^N)$ are the standard ($L^2$-based) Sobolev spaces of $\dC^N$-valued functions defined on $\dR^q$, $\Omega_{\pm}$, and $\Sigma$, respectively. Furthermore, we define the space $H^s_0(\Omega_{\pm}; \dC^N)$ as the closure of the space of infinitely times differentiable functions with compact support $C_0^{\infty}(\Omega_{\pm}; \dC^N)$ in $H^s(\Omega_{\pm}; \dC^N)$. We use the notation $C_0^\infty(\overline{\Omega_\pm}; \mathbb{C}^N) = \{ f \upharpoonright \Omega_\pm: f \in C_0^\infty(\mathbb{R}^q; \mathbb{C}^N) \}$. We denote the restrictions of functions $f: \mathbb{R}^q \rightarrow \mathbb{C}^N$ onto $\Omega_\pm$ by $f_\pm$; in this sense we write $H^s(\mathbb{R}^q \setminus \Sigma; \dC^N) = H^s(\Omega_+; \dC^N) \oplus H^s(\Omega_-; \dC^N)$ and identify $f \in H^s(\mathbb{R}^q \setminus \Sigma; \dC^N)$ with $f=f_+ \oplus f_-$, where $f_\pm \in H^s(\Omega_\pm; \dC^N)$. For the Dirichlet trace operator defined on $H^1(\mathbb{R}^q; \mathbb{C}^N)$ we write $\gamma_D: H^1(\mathbb{R}^q; \mathbb{C}^N) \rightarrow H^{1/2}(\Sigma; \mathbb{C}^N)$, while the trace map defined on suitable subspaces of $H^s(\Omega_\pm; \mathbb{C}^N)$, $s\in [0,1]$, 
is denoted by $\gamma_D^\pm$; cf. Lemma~\ref{lemma_trace_theorem} and Corollary~\ref{corollary_trace_extension}.


Let $\mathcal{H}$ and $\mathcal{G}$ be two Hilbert spaces. The domain, kernel, and range of a linear operator $T$ from $\mathcal{G}$ to $\mathcal{H}$ are denoted by $\text{dom}\, T$, $\text{ker}\, T$, and $\text{ran}\, T$, respectively. The resolvent set, the spectrum, the essential spectrum, the discrete spectrum, and the point spectrum of a self-adjoint operator $T$ are denoted by $\rho(T)$, $\sigma(T)$, $\sigma_{\text{ess}}(T)$, $\sigma_{\text{disc}}(T)$, and $\sigma_\textup{p}(T)$.  Furthermore, the restriction  to a subspace $\mathcal{U} \subset \text{dom} \, T$ is denoted by $T \upharpoonright \mathcal{U}$.  If we use for the anti-dual space of $\mathcal{H}$ the symbol $\mathcal{H}'$, then we write $\langle \cdot, \cdot \rangle_{\mathcal{H} \times \mathcal{H}'}$ for the sesquilinear duality product. Moreover, we define the anti-dual operator $A': \mathcal{G}' \rightarrow \mathcal{H}'$ of a bounded linear map $A: \mathcal{H} \rightarrow \mathcal{G}$ by the relation $\langle A \varphi, \psi \rangle_{\mathcal{G} \times \mathcal{G}'} = \langle \varphi, A' \psi \rangle_{\mathcal{H} \times \mathcal{H}'}$ for all $\varphi \in \mathcal{H}$ and $\psi \in \mathcal{G}'$.

For $q \in \{1,2\}$ the matrices $\alpha_j$ in \eqref{anti_commutation} will be chosen as the $2\times 2$ Pauli spin matrices
\begin{equation} \label{Dirac_matrices1}
\alpha_1 = \sigma_1 = \left( \begin{array}{cc}
0 & 1\\                                              
1 & 0 \\                                            
\end{array}\right), \quad \alpha_2 = \sigma_2 = \left( \begin{array}{cc}
0 & -i\\                                              
i & 0 \\                                            
\end{array}\right), \quad \alpha_0 = \sigma_3 = \left( \begin{array}{cc}
1 & 0\\                                              
0 & -1 \\                                            
\end{array}\right),
\end{equation}
and for  $q = 3$ as the $4\times 4$ Dirac matrices 
\begin{equation} \label{Dirac_matrices2}
\alpha_j = \left( \begin{array}{cc}
0 & \sigma_j \\                                              
\sigma_j & 0 \\                                            
\end{array}\right), \quad j \in \{1,2,3\}, \quad  \alpha_0 = \left( \begin{array}{cc}
I_2 & 0 \\                                              
0 & -I_2 \\                                            
\end{array}\right).
\end{equation}
Finally, we provide the integral kernel of the resolvent of the free Dirac operator $(A_0 - z)^{-1}$, which can be computed with the help of \eqref{resolvent_A_0} and the known form of the resolvent of the Laplace operator. 
For $q \in\{1,2,3\}$  the integral kernel $G_{z,q}$ is explicitly given by
\begin{equation} \label{def_G_lambda}
\begin{split}
G_{z, 1}(x) &=  \frac{i}{2}  e^{i k(z) |x| } \left( \begin{array}{cc}
\zeta(z) & \sgn (x)\\                                              
\sgn (x) & \zeta(z)^{-1} \\                                            
\end{array}\right),\\
G_{z, 2}(x) &= \frac{k(z)}{2 \pi} K_1 \big( - i k(z) | x |\big)  \frac{(\alpha \cdot x)}{ | x | }  + \frac{1}{2 \pi} K_0 \big(- i k(z) | x |\big) \big( z I_2 + m \alpha_3\big), \\
G_{z,3}(x) &= \left( z I_4 + m \alpha_0 + \left( 1 - i k(z) | x |\right) \frac{i (\alpha \cdot x)}{  | x |^2}  \right) \frac{1}{4 \pi | x |} e^{i k(z) |x|},
\end{split}
\end{equation}
respectively;
cf. \cite{AGHH05, BHOP20, T92, W03}. Here, we write $K_j$ for the modified Bessel functions of the second kind, set
\begin{equation} \label{def_k_zeta}
k(z) = \sqrt{z^2 - m^2} \quad \text{and } \zeta(z) = \frac{z + m }{ k(z)} = \frac{z + m }{ \sqrt{z^2 - m^2}},
\end{equation}
and choose $\sqrt{w}$ for $w \in \mathbb{C} \setminus [0, \infty)$ such that $\textup{Im} \sqrt{w} > 0$.

\subsection*{Acknowledgement.}
The authors are indebted to the referee for a careful reading of our manuscript and helpful suggestions to improve the text.
Moreover, we thank Dale Frymark and Mat\v{e}j Tu\v{s}ek for fruitful discussions and Konstantin Pankrashkin for informing us about the
manuscript \cite{BP22}.
Jussi Behrndt, Markus Holzmann, and Christian Stelzer gratefully acknowledge financial support by the Austrian Science Fund (FWF): P33568-N. This publication is based upon work from COST Action CA 18232 MAT-DYN-NET, supported by COST (European Cooperation in Science and Technology), www.cost.eu.

\section{Ordinary, generalized, and quasi boundary triples} \label{section_boundary_triples}

In this section we briefly recall basic definitions of ordinary, generalized, and quasi boundary triples, and some related techniques 
in extension and spectral theory of symmetric and self-adjoint operators in Hilbert spaces. 
The concepts will be presented such that they can be  applied directly to 
Dirac operators with singular interactions in the next sections. 
We refer the reader to \cite{BHS20,BL07,BL12,BGP08,DM91,DM95,GG91} for more details on boundary triple techniques. 

Throughout this section $\mathcal{H}$ denotes a complex Hilbert space with inner product~$(\cdot, \cdot)_\mathcal{H}$
and~$S$ is a densely defined closed symmetric operator with adjoint~$S^*$.

\begin{definition}\label{qbtdef}
  Let $T$ be a linear operator in $\mathcal{H}$ such that $\overline{T} = S^*$. A triple 
  $\{ \mathcal{G}, \Gamma_0, \Gamma_1 \}$ consisting of
  a Hilbert space $\mathcal{G}$
  and linear mappings $\Gamma_0, \Gamma_1: \dom T \rightarrow \mathcal{G}$
  is called a {\em quasi boundary triple} for $S^*$ if the following holds:
\begin{itemize}
 \item [{\rm (i)}] For all $f,g\in\dom T$ the abstract Green's identity
 \begin{equation*}
  (Tf,g)_{\mathcal H}-(f,Tg)_{\mathcal H}=(\Gamma_1 f,\Gamma_0 g)_{\mathcal G}-(\Gamma_0 f,\Gamma_1 g)_{\mathcal G}
 \end{equation*}
 is true.
 \item [{\rm (ii)}] The range of $\Gamma=(\Gamma_0,\Gamma_1)$ is dense in $\mathcal G\times\mathcal G$.
 \item [{\rm (iii)}] The restriction $A_0:=T\upharpoonright\ker\Gamma_0$ is a self-adjoint operator in $\mathcal H$.
\end{itemize}
If {\rm (i)} and {\rm (iii)} hold, and the mapping $\Gamma_0:\dom T\rightarrow\mathcal G$ is surjective, then  
$\{ \mathcal{G}, \Gamma_0, \Gamma_1 \}$ is called {\em generalized boundary triple}; if
{\rm (i)} and {\rm (iii)} hold, and the mapping
$\Gamma=(\Gamma_0,\Gamma_1):\dom T\rightarrow \mathcal G\times\mathcal G$
is surjective, then  
  $\{ \mathcal{G}, \Gamma_0, \Gamma_1 \}$ is called {\em ordinary boundary triple}.
\end{definition}

We remark that the above (non-standard) definition of generalized and ordinary boundary triples is equivalent to the 
usual one given in, e.g., \cite{BHS20, BGP08,DM91,DM95,GG91},
see~\cite[Corollaries~3.2 and~3.7]{BL07}. 
In particular, if $\{ \mathcal{G}, \Gamma_0, \Gamma_1 \}$ is an ordinary boundary triple,
then $T=S^*$. Moreover, each ordinary boundary triple is a generalized and quasi boundary triple, and each generalized boundary triple is a quasi boundary triple; cf. \cite[Corollary~3.7]{BL07}.
Note that a quasi boundary triple, generalized boundary triple, or ordinary boundary triple for $S^*$ exists if and only if 
the defect numbers $\dim\ker(S^*\pm i)$ coincide, i.e. if and only if $S$ admits self-adjoint extensions in $\mathcal H$.
Moreover, the operator $T$ in Definition~\ref{qbtdef} is in general 
not unique.

Next, we recall the definition of the $\gamma$-field and the Weyl function associated with the quasi 
boundary triple $\{ \mathcal{G}, \Gamma_0, \Gamma_1 \}$.
These mappings will allow us to describe spectral properties of self-adjoint extensions of $S$.
Let $A_0 = T \upharpoonright \ker \Gamma_0$. Then the direct sum decomposition
\begin{equation} \label{decomposition}
  \dom T = \dom A_0 \dot{+} \ker(T - z)=\ker\Gamma_0\dot{+} \ker(T - z),\qquad z\in\rho(A_0),
\end{equation}
holds.
The definition of the $\gamma$-field and Weyl function for quasi boundary triples
is in accordance with the one for ordinary and generalized boundary triples in \cite{DM91,DM95}.

\begin{definition} 
 Assume that $T$ is a linear operator in $\mathcal{H}$ satisfying $\overline{T} = S^*$ and let 
 $\{ \mathcal{G}, \Gamma_0, \Gamma_1 \}$ be a quasi boundary triple for $S^*$.
 Then the corresponding $\gamma$-field $\gamma$ and Weyl function $M$ are defined by
 \begin{equation*}
  \rho(A_0) \ni z\mapsto\gamma(z)=\bigl(\Gamma_0\upharpoonright\ker(T-z)\bigr)^{-1}
 \end{equation*}
and
 \begin{equation*}
  \rho(A_0) \ni z\mapsto M(z)=\Gamma_1 \bigl(\Gamma_0\upharpoonright\ker(T-z)\bigr)^{-1},
 \end{equation*}
respectively.
\end{definition}

From \eqref{decomposition} we get that the $\gamma$-field is well defined and that
$\ran \gamma(z) = \ker(T-z)$ holds for all $z \in \rho(A_0)$.
Moreover, $\dom \gamma(z) = \ran \Gamma_0$ is dense in $\mathcal{G}$ 
by Definition~\ref{qbtdef}.
With the help of the abstract Green's identity in Definition~\ref{qbtdef}~(i)
one verifies that
\begin{equation}\label{equation_gamma_star}
 \gamma(z)^*=\Gamma_1 (A_0-\overline{z})^{-1}, \qquad z \in \rho(A_0);
\end{equation}
this is a bounded and everywhere defined operator from $\mathcal H$ to $\mathcal G$. Therefore,
$\gamma(z)$ is a (in general not everywhere defined) bounded operator; cf. 
\cite[Proposition~2.6]{BL07} or \cite[Proposition~6.13]{BL12}. If $\{ \mathcal{G}, \Gamma_0, \Gamma_1 \}$
is a generalized or ordinary boundary triple, then $\gamma(z)$ is automatically bounded and everywhere defined.

Next, we state some useful properties of the Weyl function $M$ corresponding to the quasi boundary triple 
$\{\mathcal G,\Gamma_0,\Gamma_1\}$; see, e.g,~\cite[Proposition~2.6]{BL07} for proofs of these statements.
For any $z \in \rho(A_0)$ the operator $M(z)$ is densely defined in
$\mathcal G$ with $\dom M(z)=\ran\Gamma_0$ and
$\ran M(z)\subset \ran\Gamma_1$. Next, one has $\ran \Gamma_0 \subset \dom M(z)^*$ and
for all $z,\mu\in\rho(A_0)$ and $\varphi\in\ran\Gamma_0$ the relation
\begin{equation}\label{equation_diff_m}
 M(z)\varphi-M(\mu)^*\varphi=(z-\overline\mu)\gamma(\mu)^*\gamma(z)\varphi
\end{equation}
holds.
Therefore, we see that $M(z)\subset M(\overline{z})^*$ for any $z\in\rho(A_0)$ and hence
$M(z)$ is a closable, but in general unbounded linear operator in 
$\mathcal G$. If $\{ \mathcal{G}, \Gamma_0, \Gamma_1 \}$
is a generalized or ordinary boundary triple, then
$M(z)$ is bounded and everywhere defined.

In the main part of the paper we are going to use ordinary boundary triples, generalized boundary triples, quasi boundary triples, 
and their Weyl functions to define and study self-adjoint extensions of the underlying symmetry $S$.
Let again~$T$ be a linear operator in $\mathcal{H}$ such that $\overline{T} = S^*$,
let $\{ \mathcal{G}, \Gamma_0, \Gamma_1 \}$ be a quasi boundary triple for $S^*$, and
let $\vartheta$ be a linear operator (or relation) in $\mathcal{G}$. Then we define the extension $A_\vartheta$ of $S$ by
\begin{equation}\label{equation_def_extension}
 A_\vartheta=T\upharpoonright \ker(\Gamma_1 - \vartheta \Gamma_0),
\end{equation}
i.e. $f\in\dom T$ belongs to $\dom A_\vartheta$ if and only if $f$ satisfies 
$\Gamma_1 f = \vartheta \Gamma_0 f$. 
If $\vartheta$ is a symmetric operator in $\mathcal{G}$, then Green's identity implies 
\begin{equation}\label{abab}
 (A_\vartheta f,g)_{\mathcal H}-(f,A_\vartheta g)_{\mathcal H}=
 (\vartheta \Gamma_0 f, \Gamma_0 g)_{\mathcal G} - (\Gamma_0 f, \vartheta \Gamma_0 g)_{\mathcal G}=0
\end{equation}
for all $f, g \in \dom A_\vartheta$ and hence the extension $A_\vartheta$ is symmetric in $\mathcal H$. 


Of course, one is mostly interested in the self-adjointness of $A_\vartheta$. If $\{ \mathcal{G}, \Gamma_0, \Gamma_1 \}$ is an ordinary boundary triple, then $A_\vartheta$ is self-adjoint in $\mathcal H$ if and only if
$\vartheta$ is self-adjoint in $\mathcal G$; cf. Theorem~\ref{theorem_Krein_abstract} below.
However, if $\{ \mathcal{G}, \Gamma_0, \Gamma_1 \}$ is a generalized or a quasi boundary triple, then the self-adjointness of $\vartheta$ does, in general, not imply the self-adjointness of $A_\vartheta$, or vice versa. However, the following theorem, where we also state an abstract version of the Birman-Schwinger principle in item~(i) and a 
Krein-type resolvent formula for canonical extensions $A_\vartheta$ in assertion~(iii), will allow us to give conditions for the self-adjointness of $A_\vartheta$; 
for the proof we refer to \cite[Theorems~2.1.3 and~2.6.5]{BHS20}, \cite[Theorem~2.8]{BL07}, and \cite[Theorem~6.16]{BL12}.

\begin{theorem} \label{theorem_Krein_abstract}
Let $T$ be a linear operator in $\mathcal H$ satisfying $\overline{T} = S^*$, let $\{\mathcal G,\Gamma_0,\Gamma_1\}$
be a quasi boundary triple for $S^*$ with $A_0=T\upharpoonright\ker\Gamma_0$, and denote the associated $\gamma$-field and Weyl function 
by $\gamma$ and $M$, respectively.
Let $A_\vartheta$ be the extension of $S$ associated with an operator (or relation) $\vartheta$ in $\mathcal{G}$
as in \eqref{equation_def_extension}.
Then the following holds for all $z\in\rho(A_0)$:
\begin{itemize}
 \item [\rm (i)] $z \in \sigma_{\mathrm{p}}(A_\vartheta)$ if and only if 
 $0 \in \sigma_{\mathrm{p}}(\vartheta - M(z))$. Moreover, 
 \begin{equation*}
   \ker(A_\vartheta-z)=\bigl\{\gamma(z)\varphi:\varphi\in\ker(\vartheta - M(z))\bigr\}.
 \end{equation*}
 \item [{\rm (ii)}] If $z \notin \sigma_{\mathrm{p}}(A_\vartheta)$, then 
 $g\in\ran(A_\vartheta - z)$ if and only if $\gamma(\overline{z})^*g\in\ran(\vartheta - M(z))$.
 \item [{\rm (iii)}] If $z \notin \sigma_{\mathrm{p}}(A_\vartheta)$, then 
 \begin{equation*}
  (A_\vartheta - z)^{-1}g=(A_0-z)^{-1}g
      + \gamma(z)\bigl(\vartheta - M(z)\bigr)^{-1}\gamma(\overline{z})^*g
 \end{equation*}
 holds for all $g\in\ran(A_\vartheta-z)$.
\end{itemize}
If $\{\mathcal G,\Gamma_0,\Gamma_1\}$ is an ordinary boundary triple for $S^*$, then $A_\vartheta$ is self-adjoint in $\mathcal{H}$ if and only if $\vartheta$ is a self-adjoint operator (or relation) in $\mathcal{G}$ and in this case the following holds for all $z\in\rho(A_0)$:
\begin{itemize}
\item [\rm (iv)] $z \in \sigma(A_\vartheta)$ if and only if 
 $0 \in \sigma(\vartheta - M(z))$. 
 \item [\rm (v)] $z \in \sigma_{\mathrm{ess}}(A_\vartheta)$ if and only if 
 $0 \in \sigma_{\mathrm{ess}}(\vartheta - M(z))$. 
\end{itemize}
\end{theorem}

Assertion (ii) of the previous theorem shows  how the self-adjointness of an extension $A_\vartheta$
can be proven if $\{ \mathcal{G}, \Gamma_0, \Gamma_1 \}$ is a generalized or a quasi boundary triple. If $\vartheta$ is symmetric in $\mathcal G$, then $A_\vartheta$ is symmetric in $\mathcal H$ by \eqref{abab}, 
and hence $A_\vartheta$ is self-adjoint if, in addition, $\ran (A_\vartheta \mp i) = \mathcal{H}$. According to Theorem~\ref{theorem_Krein_abstract}~(ii)
the latter is the case, if $\ran \gamma(\mp i)^* \subset \ran(\vartheta - M(\pm i))$.

In Section~\ref{section_Lipschitz} we are going to use  a quasi boundary triple to study properties of Dirac operators with singular $\delta$-shell interactions. However, as indicated in~\eqref{Diracperturbed2} we will not introduce $A_{\eta, \tau, \lambda}$ in the form~\eqref{equation_def_extension}, but as
\begin{equation} \label{equation_def_A_B}
  A_{[B]}=T\upharpoonright \ker(\Gamma_0 + B \Gamma_1),
\end{equation}
where $\{ \mathcal{G}, \Gamma_0, \Gamma_1 \}$ is a quasi boundary triple for $\overline{T}=S^*$ and $B$ is a linear operator in $\mathcal{G}$. Formally, the definition in~\eqref{equation_def_A_B} corresponds to~\eqref{equation_def_extension} with $\vartheta = -B^{-1}$. 
With a similar calculation as in~\eqref{abab} one sees that $A_{[B]}$ is always symmetric in $\mathcal{H}$, if $B$ is symmetric in $\mathcal{G}$.
The following theorem is the counterpart of Theorem~\ref{theorem_Krein_abstract} for $A_{[B]}$; a proof follows, e.g., from \cite[Theorem~3.7 and Corollary~3.9]{BL12} and \cite[Corollary~2.6.3]{BHS20}. We remark that, in a similar way as above, item~(ii) of the following theorem allows us to show the self-adjointness of an extension $A_{[B]}$ defined by~\eqref{equation_def_A_B}.

\begin{theorem} \label{theorem_Krein_abstract_A_B}
Let $T$ be a linear operator in $\mathcal H$ satisfying $\overline{T} = S^*$, let $\{\mathcal G,\Gamma_0,\Gamma_1\}$
be a quasi boundary triple for $S^*$ with $A_0=T\upharpoonright\ker\Gamma_0$, and denote the associated $\gamma$-field and Weyl function 
by $\gamma$ and $M$, respectively.
Let $A_{[B]}$ be the extension of $S$ associated with an operator $B$ in $\mathcal{G}$
as in \eqref{equation_def_A_B}.
Then the following holds for all $z\in\rho(A_0)$:
\begin{itemize}
 \item [\rm (i)] $z \in \sigma_{\mathrm{p}}(A_{[B]})$ if and only if 
 $0 \in \sigma_{\mathrm{p}}(I + B M(z))$. Moreover, 
 \begin{equation*}
   \ker(A_{[B]}-z)=\bigl\{\gamma(z)\varphi:\varphi\in\ker(I + B M(z))\bigr\}.
 \end{equation*}
 \item [{\rm (ii)}] If $z \notin \sigma_{\mathrm{p}}(A_{[B]})$, then 
 $g\in\ran(A_{[B]} - z)$ if and only if $B \gamma(\overline{z})^*g\in\ran(I + B M(z))$.
 \item [{\rm (iii)}] If $z \notin \sigma_{\mathrm{p}}(A_{[B]})$, then 
 \begin{equation*}
  (A_{[B]} - z)^{-1}g=(A_0-z)^{-1}g
      - \gamma(z)\bigl(I + B M(z)\bigr)^{-1} B \gamma(\overline{z})^*g
 \end{equation*}
 holds for all $g\in\ran(A_{[B]}-z)$.
\end{itemize}
\end{theorem}

\section{One-dimensional Dirac operators with $\delta$-point interactions} \label{Sec_OneDimDirac}
In this section, we investigate one-dimensional Dirac operators with electrostatic, Lorentz scalar, and anomalous magnetic $\delta$-interactions supported on $\Sigma = \{0 \}$. The following results are well known, see for instance \cite{AGHH05, CMP13, GS87,HT22, H99, PR14}, but they are presented here for the sake of completeness. In particular, the used methods  and the obtained results  will serve as a motivation for the analysis of two- and three-dimensional Dirac operators in the following sections.  


Let $\Omega_+ = (0, \infty)$ and $\Omega_- := (-\infty, 0)$. Then, $\Sigma = \{ 0 \} = \partial \Omega_\pm$. Note that, in contrast to the higher dimensional setting described in the introduction, $\Omega_+$ is not bounded; however, the effects seen here are similar, as $\Sigma$ is compact. If we set $\nu = -1$ for the normal vector $\nu$ that is pointing outwards of $\Omega_+$, we obtain with $i \sigma_1 \sigma_3 = \sigma_2$ that the one-dimensional realization of~(\ref{Diracpertubed}) is
\begin{equation} \label{OneDimFormalExpression}
A_{\eta, \tau, \lambda} = A_0 + ( \eta I_2 + \tau\sigma_3  - \lambda \sigma_2 )\delta_{\Sigma},
\end{equation} 
where $A_0$ is the self-adjoint free Dirac operator in~\eqref{def_A_0}.

In what follows we will make use of the orthogonal decomposition $L^2(\dR;\dC^2) = L^2((0 , \infty); \dC^2) \oplus L^2((-\infty,0);\dC^2)$ and use for $f \in L^2(\mathbb{R}; \mathbb{C}^2)$ the notation $f_\pm = f \upharpoonright \Omega_\pm$.
Following the usual construction of self-adjoint realizations of the expression in~\eqref{OneDimFormalExpression} as in \cite{AGHH05}, we define first the symmetric operator
\begin{equation*}
\begin{split}
S f &:= \left( - i  \sigma_1 \frac{\mathrm{d}}{\mathrm{d} x} f_{+} + m \sigma_3 f_{+} \right) \oplus \left( - i   \sigma_1 \frac{\mathrm{d}}{\mathrm{d} x} f_{-} + m \sigma_3 f_{-} \right), \\
\dom S &:=  H^{1}_0 \left((0 , \infty);\dC^{2} \right) \oplus  H^{1}_0 ((- \infty, 0); \dC^2).
\end{split} 
\end{equation*}
It can be shown that the adjoint operator $S^{\ast}$ acts in the same way as $S$, but has the larger domain 
\begin{equation*}
\dom S^* =  H^{1} ((0,\infty);\dC^{2} ) \oplus  H^{1} ( (-\infty, 0) ;\dC^{2} ). 
\end{equation*}
In the next step, we want to introduce the expression in~\eqref{OneDimFormalExpression} as a self-adjoint extension $A_{\eta, \tau, \lambda}$ of $S$, i.e. by restricting $S^{\ast}$  to a suitable operator domain, which is characterized by imposing the coupling conditions~\eqref{jump_condition1D} on $\Sigma = \{ 0 \}$. Motivated by~\eqref{boundary_mappings_intro} we introduce the boundary mappings $\Gamma_0 , \Gamma_1 : \dom S^* \rightarrow \dC^2$ by 
\begin{equation} \label{def_Gamma_1D}
\Gamma_0 f = -i \sigma_1 \big(f(0+) - f(0-) \big)   \quad \text{ and } \quad \Gamma_1 f = \frac{1}{2} \big( f(0+) + f(0-) \big),
\end{equation}
and the matrix $P_{\eta, \tau, \lambda} = \eta I_2 + \tau \sigma_3 - \lambda \sigma_2$. With these notations we see that~\eqref{jump_condition1D} is equivalent to
$$\Gamma_0 f +  P_{\eta, \tau, \lambda} \Gamma_1 f = 0,\qquad f \in \dom S^*.$$
In the following proposition we investigate the maps $\Gamma_0 , \Gamma_1$. In order to formulate the result, recall that $k(z)$ and $\zeta(z)$, $z \in \rho(A_0) = \mathbb{C} \setminus ( (-\infty, -m] \cup [m, \infty))$, are defined in (\ref{def_k_zeta}).

\begin{proposition} \label{OneDimOBT}
	The triple  $\{ \dC^2 , \Gamma_0 , \Gamma_1 \}$ is an ordinary boundary triple for $S^{\ast}$. The value of the associated $\gamma$-field is given for $z \in \rho(A_0)$ by
	\begin{equation*}
\left[ \gamma(z) \left( \begin{array}{c}
\xi_1\\                                              
\xi_2 \\                                            
\end{array}\right) \right] (x)  = \frac{i}{2}  e^{i k(z) |x| } \left( \begin{array}{cc}
\zeta(z) & \sgn (x)\\                                              
\sgn (x) & \zeta(z)^{-1} \\                                            
\end{array}\right) \left( \begin{array}{c}
\xi_1\\                                              
\xi_2 \\                                            
\end{array}\right), \quad x \in \mathbb{R},
\end{equation*}
and the Weyl function is 
\begin{equation*}
  \rho(A_0) \ni z \mapsto M(z) :=  \frac{i}{2} \left( \begin{array}{cc}
\zeta(z) & 0 \\                                              
0 & \zeta(z)^{-1} \\                                            
\end{array}\right).
\end{equation*}
\end{proposition}

Note that the application of $\gamma(z)$ can be viewed as the matrix vector product of $G_{z,1}(x)$ in~\eqref{def_G_lambda} and the vector $(\xi_1,\xi_2)$, while the matrix representing $M(z)$ consists of the columns resulting from applying $\Gamma_1$ to the columns of $G_{z,1}(x)$.  Similar operators will play a role in the higher dimensional considerations in the next sections.

\begin{proof}[Proof of Proposition~\ref{OneDimOBT}]
	Integration by parts and a straightforward computation show that the abstract Green's identity in Definition \ref{qbtdef}~(i) is valid. To show that $(\Gamma_0, \Gamma_1): \dom S^* \rightarrow \mathbb{C}^4$ is surjective, consider for $(c_1,c_2, c_3, c_4) \in \dC^4$ the function
	\begin{equation*}
	f(x) = \frac{i}{2}  \begin{pmatrix}
	c_2 \\                                              
	c_1 \\                                            
	\end{pmatrix} \sgn (x)\, e^{-|x|} + \begin{pmatrix}
	c_3 \\                                              
	c_4 \\                                            
	\end{pmatrix} e^{-|x|},\quad x\in\dR.
	\end{equation*}
	Then $f \in \dom S^*$, $\Gamma_0 f = (c_1, c_2)$ and $\Gamma_1 f = (c_3,c_4)$ and hence,  $(\Gamma_0, \Gamma_1)$ is surjective. Finally, to show that Definition~\ref{qbtdef}~(iii) holds, notice that the restriction $A_0 = S^{\ast} \upharpoonright \ker \Gamma_0$ is the self-adjoint free Dirac operator defined in~\eqref{def_A_0}. Hence, it follows that the triple $\{ \mathbb{C}^2, \Gamma_0, \Gamma_1 \}$ is an ordinary boundary triple for $S^{\ast}$.
	
	It remains to show the claimed formulas for the $\gamma$-field and the Weyl function. Consider the functions
\begin{equation*}
f_1(x) = \frac{i}{2}  \left( \begin{array}{c}
\zeta(z) \\                                              
\sgn (x) \\                                            
\end{array}\right) e^{i k(z) |x|} \quad \text{and} \quad f_2(x) = \frac{i}{2}  \left( \begin{array}{c}
\sgn (x) \\                                              
\zeta(z)^{-1} \\                                            
\end{array}\right) e^{i k(z) |x|}.
\end{equation*}
These functions  form a basis of  $\text{ker}(S^{\ast} - z)$,  $z \in \rho(A_0)$, and $\Gamma_0 f_1 = (1,0)$ and $\Gamma_0 f_2 = (0,1)$. Hence, 
\begin{equation*}
\left[ \gamma(z) \left( \begin{array}{c}
\xi_1\\                                              
\xi_2 \\                                            
\end{array}\right) \right] (x)  = \xi_1 f_1 (x) + \xi_2 f_2(x) = \frac{i}{2}  e^{i k(z) |x| } \left( \begin{array}{cc}
\zeta(z) & \sgn (x)\\                                              
\sgn (x) & \zeta(z)^{-1} \\                                            
\end{array}\right) \left( \begin{array}{c}
\xi_1\\                                              
\xi_2 \\                                            
\end{array}\right),
\end{equation*}
which is the claimed expression for $\gamma(z)$. Moreover, we have
\begin{equation*}
M(z) \left( \begin{array}{c}
\xi_1\\                                              
\xi_2 \\                                            
\end{array}\right) = \Gamma_1 \gamma(z) \left( \begin{array}{c}
\xi_1\\                                              
\xi_2 \\                                            
\end{array}\right) = \frac{i}{2} \left( \begin{array}{cc}
\zeta(z) & 0 \\                                              
0 & \zeta(z)^{-1} \\                                            
\end{array}\right) \left( \begin{array}{c}
\xi_1\\                                              
\xi_2 \\                                            
\end{array}\right),
\end{equation*}
which yields the claimed formula for $M(z)$. This finishes the proof.
\end{proof}

Using the ordinary boundary triple from Proposition \ref{OneDimOBT} and the matrix $P_{\eta, \tau, \lambda} = \eta I_2 + \tau \sigma_3 - \lambda \sigma_2$, we can define now the operator 
\begin{equation} \label{def_A_eta_1D}
A_{\eta, \tau, \lambda} =  S^{\ast} \upharpoonright \text{ker}(\Gamma_0+ P_{\eta, \tau, \lambda} \Gamma_1),
\end{equation}
which is interpreted as the realization of the formal expression (\ref{OneDimFormalExpression}) in $L^2(\mathbb{R}; \mathbb{C}^2)$, cf. \eqref{Diracperturbed2}. In the following theorem, we show the self-adjointness of $A_{\eta, \tau, \lambda}$ and study its spectral properties.

\begin{theorem}
  For any $\eta,\tau,\lambda \in \dR$ the operator $A_{\eta,\tau,\lambda}$ defined in ~\eqref{def_A_eta_1D} is self-adjoint in $L^2(\mathbb{R}; \mathbb{C}^2)$ and for all $z \in \rho(A_0) \cap \rho(A_{\eta,\tau,\lambda})$ and  $f \in L^2(\dR;\dC^2)$ one has
	\begin{equation*}
	\begin{split}
	(&A_{\eta, \tau , \lambda} - z)^{-1}f =  (A_{0} - z)^{-1}f \\ 
	&~~~- \left( g_{z,1} , \overline{f} \right)_{L^2 (\dR;\dC^2)} \left( \begin{array}{c}
	\zeta(z) \\                                              
	\sgn(\cdot) \\                                            
	\end{array}\right) e^{i k(z) |\cdot|} - \left( g_{z,2}  , \overline{f} \right)_{L^2 (\dR;\dC^2)} \left( \begin{array}{c}
	\sgn(\cdot)  \\                                              
	\zeta(z)^{-1} \\                                            
	\end{array}\right) e^{i k(z) |\cdot|},
	\end{split}
	\end{equation*}
	where $g_{z,1}, g_{z,2}$ are given by
	\begin{equation*}
	g_{z,j}(x) = \frac{i}{2} \Big( (I_2+P_{\eta, \tau, \lambda} M(z))^{-1} P_{\eta, \tau, \lambda} G_{z,1}(-x) \Big)^{\top} e_j, \quad x \in \mathbb{R},~j \in \{1,2\},
	\end{equation*}
	with $e_1 = (1,0)$ and $e_2 = (0,1)$.
	Moreover, the following holds:
	\begin{itemize}
		\item[(i)] $\sigma_{\mathrm{ess}}(A_{\eta,\tau,\lambda}) =\sigma_{\mathrm{ess}}(A_0) =  \left(-\infty,-m\right]\cup \left[m,\infty\right)$.
		\item[(ii)] Set $d = \eta^2-\tau^2-\lambda^2$ and assume  $m > 0$. 
		\begin{itemize}
		  \item[(a)] If $d=4$, then $\sigma_\textup{disc}(A_{\eta, \tau, \lambda})= \{- \frac{m \tau}{\eta} \}$.
		  \item[(b)] If $d \neq 4$, then exactly those points 
		\begin{equation} \label{equation_z_pm}
		z_{\pm} = m \frac{-\eta \tau \pm \left| \frac{d}{4}-1\right| \sqrt{ \lambda^2 + \left( \frac{d}{4}+1\right)^2}}{\eta^2 + \left( \frac{d}{4} - 1 \right)^2}
		\end{equation}
		that obey $(d-4) ( m \tau + \eta z_{\pm} ) > 0$ belong to $\sigma_\textup{disc}(A_{\eta, \tau, \lambda})$.
		\end{itemize}  
	\end{itemize}
\end{theorem}

\begin{proof}
	First, since the matrix $P_{\eta, \tau, \lambda} = \eta I_2 + \tau \sigma_3 - \lambda \sigma_2$ is self-adjoint, Theorem~\ref{theorem_Krein_abstract_A_B} implies that the operator $A_{\eta, \tau, \lambda}$ defined by~\eqref{def_A_eta_1D} is self-adjoint in $L^2(\mathbb{R}; \mathbb{C}^2)$. Moreover, by Theorem \ref{theorem_Krein_abstract_A_B} (iii) we have for $z \in \rho(A_{\eta, \tau, \lambda}) \cap \rho(A_0)$ and $f \in L^2(\mathbb{R}; \mathbb{C}^2)$
	\begin{equation*}
	(A_{\eta, \tau, \lambda} - z)^{-1}f = (A_{0} - z)^{-1}f-\gamma(z) (I_2+  P_{\eta, \tau, \lambda} M(z) )^{-1} P_{\eta, \tau, \lambda} \gamma(\overline{z})^{\ast}f.
	\end{equation*}
	After a simple calculation using the expressions for $\gamma(z)$ and $M(z)$ in Proposition~\ref{OneDimOBT} one obtains the claimed resolvent formula. Statement~(i) follows from the fact that both $A_{\eta, \tau, \lambda}$ and $A_0$ are self-adjoint extensions of the operator $S$, which has the finite defect indices $(2,2)$. 
	
	It remains to show the claims about the discrete spectrum. For $m=0$ we have   $\sigma_{\rm ess}(A_{\eta, \tau , \lambda}) = \mathbb{R}$ and therefore, $\sigma_{\rm disc}(A_{\eta, \tau, \lambda}) = \emptyset$. Hence, we assume $m>0$ in the following. Note that~(i) and Theorem~\ref{theorem_Krein_abstract_A_B}~(i) imply that $z \in \sigma_{\text{disc}}(A_{\eta, \tau, \lambda})$ if and only if $0 \in \sigma( I_2 +  P_{\eta, \tau, \lambda} M(z))$. The latter is true if and only if $\det (I_2 +  P_{\eta, \tau, \lambda} M(z)) = 0$, which is equivalent to
	\begin{equation} \label{Eq_spec_Onedim}
	\frac{d}{4} - 1 = \frac{m \tau + z \eta}{\sqrt{m^2-z^2}}.
	\end{equation}
	Assume that $d = 4$. Then $\eta \neq 0$ and \eqref{Eq_spec_Onedim} implies that $\sigma_\textup{disc}(A_{\eta, \tau, \lambda})=\{ - \frac{m \tau }{\eta}\}$, which is the claim in item~(ii)~(a). In the case $d \neq 4$, observe that \eqref{Eq_spec_Onedim} can only be true if 
	\begin{equation*}
	0 < \sqrt{m^2-z^2} = \left( \frac{d}{4} - 1 \right)^{-1} ( m \tau + \eta z), 
	\end{equation*}
	which is the stated restriction for discrete eigenvalues. Squaring both sides in~\eqref{Eq_spec_Onedim} and rearranging the terms yield
	\begin{equation} \label{quadratic_eq_disc}
	\left( \eta^2 + \left( \frac{d}{4} -1 \right)^2 \right) z^2 + 2 m \tau \eta z + m^2 \left( \tau^2 - \left( \frac{d}{4} -1 \right)^2 \right) = 0.
	\end{equation}
	Observe that due to
	\begin{equation} \label{equation_discriminant}
	\eta^2 - \tau^2 + \left( \frac{d}{4} - 1\right)^2 = \lambda^2 + \left( \frac{d}{4} + 1\right)^2
	\end{equation}
	the discriminant of the quadratic equation \eqref{quadratic_eq_disc} is always non-negative and therefore the two real solutions of~\eqref{quadratic_eq_disc} are given by $z_\pm$ in item~(ii)~(b). Thus, all claims are shown.
\end{proof}

\begin{remark}
	For certain interaction strengths, the discrete spectrum can be specified more explicitly. In the following, we always assume that $d = \eta^2 - \tau^2 - \lambda^2 \neq 4$.
	\begin{enumerate}
		\item[(i)] If $\lambda = 0$, then the matrix $I_2 + P_{\eta, \tau, 0} M(z)$ is a diagonal matrix and hence, $\det(I_2 + P_{\eta, \tau, 0} M(z))=0$ if and only if one of its entries is zero. The resulting equations yield the solutions
		\begin{equation*}
		z_1 =  m \frac{4- (\eta + \tau)^2}{4+(\eta + \tau)^2} \quad \text{and} \quad z_2 =  m \frac{(\eta - \tau)^2-4}{ (\eta - \tau)^2+4}
		\end{equation*}
		and $z_1 \in \sigma_{\rm disc}(A_{\eta, \tau, \lambda})$,
		if $\eta + \tau < 0$, and $z_2 \in  \sigma_{\rm disc}(A_{\eta, \tau, \lambda})$ if $\eta - \tau > 0$.
		\item[(ii)] If $|\eta| = |\tau|$, then $A_{\eta, \tau, \lambda}$ has at most one discrete eigenvalue and the expression in~\eqref{equation_z_pm} simplifies to
		\begin{equation*}
		z = \begin{cases} m \frac{(\lambda^2+4)^2-16 \eta^2}{(\lambda^2+4)^2+16 \eta^2},& \text{ if } \eta = \tau < 0,\\ m \frac{16 \eta^2-(\lambda^2+4)^2}{ 16 \eta^2+(\lambda^2+4)^2},& \text{ if } \eta = -\tau > 0. \end{cases}
		\end{equation*}
		\item[(iii)] If $\eta = 0$, then the condition $0 < (d-4) ( m \tau + \eta z_{\pm} )= -m \tau (\tau^2+\lambda^2+4)$ is only fulfilled if		
		$\tau < 0$. In this case, the expression in~\eqref{equation_z_pm} simplifies to
		\begin{equation*}
		z_{\pm} = \mp \frac{4 m}{d-4} \sqrt{\lambda^2 + \left( \frac{d}{4} + 1\right)^2}.
		\end{equation*}
		\item[(iv)] If  $\tau = 0$ and $\eta \neq 0$, then the condition $0 < (d-4) ( m \tau + \eta z_{\pm} )= \eta z_{\pm} (d-4)$ is  fulfilled for exactly one of the solutions $z_\pm$ in~\eqref{equation_z_pm}, which can be simplified with~\eqref{equation_discriminant} to 
		\begin{equation*}
		z = \sgn(\eta) \frac{m}{4} \frac{d-4}{\sqrt{ \eta^2 + \left( \frac{d}{4} - 1 \right)^2}}.
		\end{equation*}
	\end{enumerate}
\end{remark}

\section{Boundary triples for two- and three-dimensional Dirac operators with singular interactions} \label{section_3d}

In this section we use similar boundary mappings as in Section~\ref{Sec_OneDimDirac} to construct boundary triples for Dirac operators with $\delta$-shell interactions in $\mathbb{R}^2$ and $\mathbb{R}^3$. However, by translating these natural boundary mappings in ~\eqref{def_Gamma_1D} directly to the higher dimensional setting 
one obtains a generalized or quasi boundary triple instead of an ordinary boundary triple; cf. Section~\ref{section_QBT}. In Section~\ref{section_OBT} we introduce an ordinary boundary triple that can be used to study Dirac operators with singular interactions, but with non-obvious and unbounded parameters. Before we can define the boundary triples, some preliminaries related to function spaces and trace theorems are needed. For smooth surfaces similar boundary triples and Sobolev spaces were used in \cite{BEHL18, BEHL19, BH20, HOP18} and \cite{BHOP20, BFSB17_1, OV16}, respectively, see also \cite{BHSS21} for the Lipschitz case. 

Let $q \in \{2,3\}$ be the space dimension and set $N := 2^{[(q+1)/2]}$, where $[\cdot]$ is the Gauss bracket. Consequently, we have $N=2$ for $q=2$ and $N=4$ for $q=3$. Let $\alpha_0, \dots, \alpha_q$ be the $q+1$ anti-commuting $\mathbb{C}^{N\times N}$-valued Dirac matrices introduced in \eqref{Dirac_matrices1}--\eqref{Dirac_matrices2}.
Throughout this section, let $\Omega_+ \subset \mathbb{R}^q$, $q \in \{2,3\}$, be a bounded Lipschitz domain, set $\Omega_- := \mathbb{R}^q \setminus \overline{\Omega_+}$ and $\Sigma := \partial \Omega_+ = \partial \Omega_-$. We denote by $\nu$ the unit normal vector field at $\Sigma$ that is pointing outwards of $\Omega_+$. In the following we will often denote the restriction of a function $f$ defined on $\mathbb{R}^q$ onto $\Omega_\pm$ by $f_\pm$ and we will view $f$ as a two component vector 
$f=f_+\oplus f_-$.

\subsection{Sobolev spaces for Dirac operators and related trace theorems}
\label{subsec_SobolevSpaces}

Define for $s \in [0,1]$ the space 
\begin{equation*}
  H^s_\alpha(\Omega_\pm; \mathbb{C}^N) := \bigl\{ f \in H^s(\Omega_\pm; \mathbb{C}^N): (\alpha \cdot \nabla) f \in L^2(\Omega_\pm; \mathbb{C}^N) \bigr\},
\end{equation*}
where the derivatives are understood in the distributional sense and $H^s(\Omega_\pm; \mathbb{C}^N)$ is the standard $L^2$-based Sobolev space of order $s$ of $\mathbb{C}^N$-valued functions,
and endow it with the norm 
\begin{equation*}
  \| f \|_{H^s_\alpha(\Omega_\pm; \mathbb{C}^N)}^2 := \| f \|_{H^s(\Omega_\pm; \mathbb{C}^N)}^2 + \| (\alpha \cdot \nabla) f \|_{L^2(\Omega_\pm; \mathbb{C}^N)}^2.
\end{equation*}
One can show with standard techniques that $H^s_\alpha(\Omega_\pm; \mathbb{C}^N)$ is a Hilbert space and that $C_0^\infty(\overline{\Omega_\pm}; \mathbb{C}^N)$ is dense in $H^s_\alpha(\Omega_\pm; \mathbb{C}^N)$; cf. \cite[Lemma~2.1]{BFSB17_1}, \cite[Lemma~3.2]{BH20}, or \cite[Proposition~2.12]{OV16} for similar arguments. Moreover, with the help of the Fourier transform one finds that $H^s_\alpha(\mathbb{R}^q; \mathbb{C}^N) = H^1(\mathbb{R}^q; \mathbb{C}^N)$ for any $s \in [0,1]$. We shall use the notation
\begin{equation*}
  H^s_\alpha(\mathbb{R}^q \setminus \Sigma; \mathbb{C}^N) = H^s_\alpha(\Omega_+; \mathbb{C}^N) \oplus H^s_\alpha(\Omega_-; \mathbb{C}^N).
\end{equation*}
Using \cite[Chapter~1, Theorem~14.3]{LM72} it is not difficult to see that the interpolation property
\begin{equation} \label{interpolation_H_alpha_Omega}
  \big[ H^1_\alpha(\Omega_\pm; \mathbb{C}^N), H^0_\alpha(\Omega_\pm; \mathbb{C}^N) \big]_{1-s} = H^s_\alpha(\Omega_\pm; \mathbb{C}^N)
\end{equation}
holds for any $s \in [0,1]$.
In the following lemma we state a trace theorem for $H^s_\alpha(\Omega_\pm; \mathbb{C}^N)$ for $s \geq \frac{1}{2}$.

\begin{lemma} \label{lemma_trace_theorem}
  For $s \in [\frac{1}{2}, 1]$ the map $C_0^\infty(\overline{\Omega_\pm}; \mathbb{C}^N) \ni f \mapsto f|_{\Sigma}$ admits a unique continuous extension $\gamma_D^\pm: H^s_\alpha(\Omega_\pm; \mathbb{C}^N) \rightarrow H^{s-1/2}(\Sigma; \mathbb{C}^N)$.
\end{lemma}
\begin{proof}
  For $s \in (\frac{1}{2}, 1]$ the claim follows from the classical trace theorem \cite[Theorem~3.38]{M00}, as $H^s_\alpha(\Omega_\pm; \mathbb{C}^N)$ is continuously embedded in $H^s(\Omega_\pm; \mathbb{C}^N)$. It remains to verify the claim for $s=\frac{1}{2}$. Consider for $s_1,s_2 \in \mathbb{R}$ the Hilbert space
  \begin{equation} \label{def_H_delta}
    H^{s_1,s_2}_\Delta(\Omega_\pm; \mathbb{C}^N) := \bigl\{ f \in H^{s_1}(\Omega_\pm; \mathbb{C}^N): \Delta f \in H^{s_2}(\Omega_\pm; \mathbb{C}^N) \bigr\}
  \end{equation}
  endowed with the norm 
  $$\| f \|^2_{H^{s_1,s_2}_\Delta(\Omega_\pm; \mathbb{C}^N)}:=\| f \|_{H^{s_1}(\Omega_\pm; \mathbb{C}^N)}^2 + \| \Delta f \|_{H^{s_2}(\Omega_\pm; \mathbb{C}^N)}^2.$$ 
  It follows from \cite[Lemma~3.1]{GM11} that there exists a continuous trace map from $H^{1/2,-1}_\Delta(\Omega_\pm; \mathbb{C}^N)$ to $L^2(\Sigma; \mathbb{C}^N)$.
  Since \eqref{anti_commutation} implies $(\alpha \cdot \nabla)^2 = \Delta I_N$ in the distributional sense, $H^{1/2}_\alpha(\Omega_\pm; \mathbb{C}^N)$ is continuously embedded in $H^{1/2,-1}_\Delta(\Omega_\pm; \mathbb{C}^N)$. This yields the claim also for $s=\frac{1}{2}$.
\end{proof}

Using Lemma~\ref{lemma_trace_theorem}, that $\pm \nu$ is the unit normal vector field pointing outwards of $\Omega_\pm$, and the fact that $C_0^\infty(\overline{\Omega_\pm}; \mathbb{C}^N)$ is dense in $H^s_\alpha(\Omega_\pm; \mathbb{C}^N)$ one can show for all $f, g \in H^s_\alpha(\Omega_\pm; \mathbb{C}^N)$, $s \in[ \frac{1}{2}, 1]$, the following integration by parts formula:
\begin{equation} \label{int_by_parts}
  \int_{\Omega_\pm} i (\alpha \cdot \nabla) f \cdot \overline{g} \,\textup{d}x = \pm \int_{\Sigma} i (\alpha \cdot \nu) \gamma_D^\pm f \cdot \overline{\gamma_D^\pm g} \,\textup{d}\sigma + \int_{\Omega_\pm} f \cdot \overline{i (\alpha \cdot \nabla) g} \,\textup{d}x.
\end{equation}

In the construction of boundary triples for Dirac operators with singular interactions some families of integral operators related to the fundamental solution $G_{z,q}$, $q=2,3$, given in~\eqref{def_G_lambda} are required. 
We introduce for $z \in\mathbb{C}\setminus ((-\infty, -m] \cup [m, \infty))$ the potential operator $\Phi_z: L^2(\Sigma; \mathbb{C}^N) \rightarrow L^2(\mathbb{R}^q; \mathbb{C}^N)$ by
\begin{equation} \label{def_Phi_lambda}
  \Phi_z \varphi(x) := \int_\Sigma G_{z,q}(x-y) \varphi(y) \textup{d}\sigma(y), 
  \quad \varphi \in L^2(\Sigma; \mathbb{C}^N),~x \in \mathbb{R}^q \setminus \Sigma,
\end{equation}
and the strongly singular boundary integral operator $\mathcal{C}_z: L^2(\Sigma; \mathbb{C}^N) \rightarrow L^2(\Sigma; \mathbb{C}^N)$ acting as
\begin{equation} \label{def_C_lambda}
  \mathcal{C}_z \varphi(x) := \lim_{\varepsilon \searrow 0} \int_{\Sigma \setminus B(x, \varepsilon)} 
  G_{z,q}(x-y) \varphi(y) \textup{d}\sigma(y), \quad
  \varphi \in L^2(\Sigma; \mathbb{C}^N),~x \in \Sigma,
\end{equation}
where $B(x,\varepsilon)$ is the ball of radius $\varepsilon$ centered at $x$.
Both operators $\Phi_z$ and $\mathcal{C}_z$ are well defined and bounded, see \cite[Lemma~2.1 and~Lemma~3.3]{AMV14} for $q=3$ and $z=0$; in the other cases this can be shown with the same arguments. Moreover, for $z \in (-m,m)$ the operator $\mathcal{C}_z$ is self-adjoint in $L^2(\Sigma; \mathbb{C}^N)$. The operator $\mathcal{C}_z$ also satisfies the formula
\begin{equation}\label{C_inv}
-4\big(\mathcal{C}_z (\alpha \cdot \nu)\big)^2 = I_N, \qquad z \in \mathbb{C} \setminus \big( (-\infty, -m] \cup [m, \infty) \big). 
\end{equation}
This result can be found in \cite[Lemma~2.2]{AMV15} for $z \in (- m, m)$ in the three-dimensional setting and can be shown in the same way in the other cases. Formula \eqref{C_inv} shows that $\mathcal{C}_z$ is bijective and that it has the bounded inverse 
\begin{equation} \label{formula_C_inv}
  (\mathcal{C}_z)^{-1} = -4(\alpha \cdot \nu) \mathcal{C}_z (\alpha \cdot \nu), \qquad z \in \mathbb{C} \setminus \big( (-\infty, -m] \cup [m, \infty) \big).
\end{equation}

In the next lemma we improve the mapping properties of $\Phi_z$.

\begin{lemma} \label{lemma_Phi_mapping_properties}
For any $z \in \rho(A_0)$ the operator $\Phi_z$ gives rise to a bounded map 
\begin{equation*}
  \Phi_z: L^2(\Sigma; \mathbb{C}^N) \rightarrow H^{1/2}_\alpha(\mathbb{R}^q \setminus \Sigma; \mathbb{C}^N).
\end{equation*}
\end{lemma}
\begin{proof}
  Let $\gamma_D: H^1(\mathbb{R}^q; \mathbb{C}^N) \rightarrow H^{1/2}(\Sigma; \mathbb{C}^N)$ be the Dirichlet trace operator. First, Fubini's theorem implies that
  \begin{equation} \label{equation_Phi_z_adjoint}
    ( \Phi_z \varphi, f )_{L^2(\mathbb{R}^q; \mathbb{C}^N)} = \big( \varphi, \gamma_D (A_0 - \overline{z})^{-1} f \big)_{L^2(\Sigma; \mathbb{C}^N)} 
  \end{equation}
  holds for all $\varphi \in L^2(\Sigma; \mathbb{C}^N)$ and $f \in L^2(\mathbb{R}^q; \mathbb{C}^N)$.
  In the following we denote by $-\Delta$ the free Laplace operator defined on $H^2(\mathbb{R}^q)$. Then, it is not difficult to see for $\mu \in \mathbb{C} \setminus [0,\infty)$ that $\gamma_D (-\Delta - \overline{\mu})^{-1}: L^2(\mathbb{R}^q) \rightarrow L^2(\Sigma)$ is bounded. Hence, we can define the single layer potential 
  \begin{equation*}
    \textup{SL}(\mu):= \big( \gamma_D (-\Delta-\overline{\mu})^{-1} \big)^*: L^2(\Sigma) \rightarrow L^2(\mathbb{R}^q).
  \end{equation*}
  It is known that the single layer potential gives rise to a bounded operator $\textup{SL}(\mu): L^2(\Sigma) \rightarrow H^{3/2,0}_\Delta(\mathbb{R}^q \setminus \Sigma)$, where $H^{3/2,0}_\Delta(\mathbb{R}^q \setminus \Sigma)$ is defined as in~\eqref{def_H_delta}; cf. \cite[Equation~(2.127)]{GM08}. Using~\eqref{equation_Dirac_Laplace}, we find with~\eqref{int_by_parts} for $\varphi \in L^2(\Sigma; \mathbb{C}^N)$ and any test function $f \in C_0^\infty(\Omega_\pm; \mathbb{C}^N) \subset C_0^\infty(\mathbb{R}^q; \mathbb{C}^N)$ that
  \begin{equation*}
    \begin{split}
      \big( (-i (\alpha \cdot \nabla) &+ m \alpha_0 + z I_N) (\textup{SL}(z^2-m^2) \varphi)_\pm, f \big)_{L^2(\Omega_\pm; \mathbb{C}^N)} \\
      &= \big( \textup{SL}(z^2-m^2) \varphi, (A_0 + \overline{z}) f \big)_{L^2(\mathbb{R}^q; \mathbb{C}^N)} \\
      &= \big( \varphi, \gamma_D (A_0 - \overline{z})^{-1} (A_0 + \overline{z})^{-1} (A_0 + \overline{z}) f \big)_{L^2(\Sigma; \mathbb{C}^N)} \\
      &= \big( \varphi, \gamma_D (A_0 - \overline{z})^{-1} f \big)_{L^2(\Sigma; \mathbb{C}^N)} 
      = \big( (\Phi_z \varphi)_\pm, f \big)_{L^2(\Omega_\pm; \mathbb{C}^N)}, \\
    \end{split}
  \end{equation*}
  where~\eqref{equation_Phi_z_adjoint} was used in the last step. Since this holds for all $f \in C_0^\infty(\Omega_\pm; \mathbb{C}^N)$, we conclude  that
  \begin{equation}
    \Phi_z \varphi = (-i \alpha \cdot \nabla + m \alpha_0 + z I_N) \textup{SL}(z^2-m^2) \varphi \quad \text{in } \mathbb{R}^q \setminus \Sigma
  \end{equation}
  holds. Therefore, the mapping properties of  $\textup{SL}(z^2-m^2)$ mentioned above and $(\alpha \cdot \nabla)^2=\Delta$ imply the claim of this lemma.
\end{proof}

We note that for $\varphi \in L^2(\Sigma; \mathbb{C}^N)$ the trace of $\Phi_z \varphi$, which is well defined by Lemma~\ref{lemma_trace_theorem} and Lemma~\ref{lemma_Phi_mapping_properties}, is given by
\begin{equation} \label{trace_Phi_lambda}
  \gamma_D^\pm (\Phi_z \varphi)_\pm = \mp \frac{i}{2} (\alpha \cdot \nu) \varphi + \mathcal{C}_z \varphi;
\end{equation}
this can be shown in the same way as in \cite[Lemma~3.3]{AMV14} or \cite[Proposition~3.4]{BHOP20}.

Finally, in order to further characterize the mapping properties of $\Phi_z$ and $\mathcal{C}_z$ in the next section, we need another type of Sobolev spaces. Let $\mathbb{U}: \Sigma \rightarrow \mathbb{C}^{N \times N}$ be a measurable function such that $\mathbb{U}(x)$ is a unitary matrix for any $x \in \Sigma$. Introduce for $t \in [0,\frac{1}{2}]$ the space 
\begin{equation} \label{def_H_U}
  H^t_\mathbb{U}(\Sigma; \mathbb{C}^N) := \big\{ \varphi \in L^2(\Sigma; \mathbb{C}^N): \mathbb{U} \varphi \in H^{t}(\Sigma; \mathbb{C}^N) \big\},
\end{equation}
where $H^t(\Sigma; \mathbb{C}^N)$ denotes the standard Sobolev space on $\Sigma$ of $\mathbb{C}^N$-valued functions. Moreover, we endow $H^t_\mathbb{U}(\Sigma; \mathbb{C}^N)$ with the natural scalar product
\begin{equation*}
(\varphi,\psi)_{H^t_\mathbb{U}(\Sigma; \mathbb{C}^N)} := (\mathbb{U} \varphi,\mathbb{U} \psi)_{H^t(\Sigma; \mathbb{C}^N)}, \quad  \varphi,\psi \in H^t_\mathbb{U}(\Sigma; \mathbb{C}^N).
\end{equation*} 
Note that $H^0_\mathbb{U}(\Sigma; \mathbb{C}^N) = L^2(\Sigma; \mathbb{C}^N)$.
By definition, for any $t \in [0,\frac{1}{2}]$ the multiplication by $\mathbb{U}$ is a unitary map from $H^t_\mathbb{U}(\Sigma; \mathbb{C}^N)$ to $H^t(\Sigma; \mathbb{C}^N)$, and the multiplication operator associated with the measurable function
$\mathbb{U}^*: \Sigma \rightarrow \mathbb{C}^{N \times N}$, $x\mapsto \mathbb{U}(x)^*$,
is a unitary map from $H^t(\Sigma; \mathbb{C}^N)$ to $H^t_\mathbb{U}(\Sigma; \mathbb{C}^N)$.

For $t \in [-\frac{1}{2},0]$ we also make use of the anti-dual spaces
\begin{equation} \label{def_H_U_dual}
  H^t_\mathbb{U}(\Sigma; \mathbb{C}^N) := \big( H^{-t}_{\mathbb{U}}(\Sigma; \mathbb{C}^N) \big)'.
\end{equation}
Since for any $t \in [0, \frac{1}{2}]$ the map $\mathbb{U}^*: H^t(\Sigma; \mathbb{C}^N) \rightarrow H^t_{\mathbb{U}}(\Sigma; \mathbb{C}^N)$ is unitary  with $(\mathbb{U}^*)^{-1}=\mathbb{U}$, its anti-dual map, that is denoted by $\mathbb{U}$ as well, also provides a unitary map from $H^{-t}_{\mathbb{U}}(\Sigma; \mathbb{C}^N)$ to $H^{-t}(\Sigma; \mathbb{C}^N)$; its action is, by definition, constituted by
\begin{equation*}
  \big\langle \mathbb{U} \varphi_1, \psi_1 \big\rangle_{H^{-t}(\Sigma; \mathbb{C}^N) \times H^t(\Sigma; \mathbb{C}^N)}
  = \big\langle  \varphi_1, \mathbb{U}^* \psi_1 \big\rangle_{H^{-t}_\mathbb{U}(\Sigma; \mathbb{C}^N) \times H^t_\mathbb{U}(\Sigma; \mathbb{C}^N)}
\end{equation*}
for $\varphi_1 \in H^{-t}_\mathbb{U}(\Sigma; \mathbb{C}^N)$ and $\psi_1 \in H^t(\Sigma; \mathbb{C}^N)$. In a similar way one finds that $\mathbb{U}^*: H^{-t}(\Sigma; \mathbb{C}^N) \rightarrow H^{-t}_{\mathbb{U}}(\Sigma; \mathbb{C}^N)$ is unitary and the action is given by
\begin{equation*}
  \big\langle \mathbb{U}^* \varphi_2,  \psi_2 \big\rangle_{H^{-t}_\mathbb{U}(\Sigma; \mathbb{C}^N) \times H^t_\mathbb{U}(\Sigma; \mathbb{C}^N)}
  = \big\langle  \varphi_2, \mathbb{U} \psi_2 \big\rangle_{H^{-t}(\Sigma; \mathbb{C}^N) \times H^t(\Sigma; \mathbb{C}^N)}
\end{equation*}
for $\varphi_2 \in H^{-t}(\Sigma; \mathbb{C}^N)$ and $\psi_2 \in H^t_\mathbb{U}(\Sigma; \mathbb{C}^N)$.
Again, by using \cite[Chapter~1, Theorem~14.3]{LM72} it is not difficult to verify the interpolation property
\begin{equation} \label{interpolation_H_alpha_Sigma1}
  \big[ H^{1/2}_{\mathbb{U}}(\Sigma; \mathbb{C}^N), H^0_{\mathbb{U}}(\Sigma; \mathbb{C}^N) \big]_{1-t} = H^{t/2}_{\mathbb{U}}(\Sigma; \mathbb{C}^N), \qquad t \in [0,1],
\end{equation}
and with duality one also gets
\begin{equation} \label{interpolation_H_alpha_Sigma2}
  \big[ H^{0}_{\mathbb{U}}(\Sigma; \mathbb{C}^N), H^{-1/2}_{\mathbb{U}}(\Sigma; \mathbb{C}^N) \big]_{1-t} = H^{(t-1)/2}_{\mathbb{U}}(\Sigma; \mathbb{C}^N), \qquad t \in [0,1].
\end{equation}
Mostly, the above construction is applied for $\mathbb{U} = \alpha \cdot \nu$, and then we will write
\begin{equation*}
  H_\alpha^t(\Sigma; \mathbb{C}^N) := H_{\alpha \cdot \nu}^t(\Sigma; \mathbb{C}^N).
\end{equation*}
Since $\mathbb{U} = \alpha \cdot \nu$ is self-adjoint, we get for any $t \in [-\frac{1}{2}, \frac{1}{2}]$ that the maps
\begin{equation} \label{mapping_properties_alpha}
  \alpha \cdot \nu: H^t_\alpha(\Sigma; \mathbb{C}^N) \rightarrow H^t(\Sigma;\mathbb{C}^N) \quad \text{and} \quad 
  \alpha \cdot \nu: H^t(\Sigma; \mathbb{C}^N) \rightarrow H^t_\alpha(\Sigma;\mathbb{C}^N)
\end{equation}
are unitary.
If $\Sigma$ is $C^{1,t+\varepsilon}$-smooth for some $\varepsilon>0$, then $H^t_\alpha(\Sigma; \mathbb{C}^N) = H^t(\Sigma; \mathbb{C}^N)$, cf. \cite[Lemma~A.2]{BHM20}.

\subsection{Quasi boundary triples and generalized boundary triples for Dirac operators with singular interactions} \label{section_QBT}

In this subsection we follow ideas from Section~\ref{Sec_OneDimDirac} and introduce a family of quasi boundary triples for Dirac operators; similar constructions 
can also be found in \cite{BEHL18, BH20}.

We introduce for $s \in [0, 1]$ the operators $T^{(s)}$ in $L^2(\mathbb{R}^q; \mathbb{C}^N)$ by
\begin{equation} \label{def_T} 
  \begin{split}
    T^{(s)} f &:= (-i (\alpha \cdot \nabla) + m \alpha_0) f_+ \oplus (-i (\alpha \cdot \nabla) + m \alpha_0) f_-, \\
    \dom T^{(s)} &:= H^s_\alpha(\Omega_+; \mathbb{C}^N) \oplus H^s_\alpha(\Omega_-; \mathbb{C}^N) = H^s_\alpha(\mathbb{R}^q \setminus \Sigma; \mathbb{C}^N),
  \end{split}
\end{equation}
and $S := T^{(s)} \upharpoonright H^1_0(\mathbb{R}^q \setminus \Sigma; \mathbb{C}^N)$, which is given more explicitly by
\begin{equation*} 
  S f = (-i (\alpha \cdot \nabla) + m \alpha_0) f, \qquad \dom S = H^1_0(\mathbb{R}^q \setminus \Sigma; \mathbb{C}^N). 
\end{equation*}
The operator $S$ is densely defined, closed, and symmetric. Using standard arguments and distributional derivatives one verifies that
\begin{equation} \label{equation_adjoint}
  S^* = T^{(0)} \qquad \text{and} \qquad (T^{(0)})^* = S;
\end{equation}
cf. \cite[Proposition~3.1]{BH20} for a similar argument.
For $s \in [\frac{1}{2}, 1]$ we introduce the mappings $\Gamma_0^{(s)}, \Gamma_1^{(s)}: \dom T^{(s)} \rightarrow L^2(\Sigma; \mathbb{C}^N)$ by
\begin{equation} \label{def_Gamma}
  \Gamma_0^{(s)} f := i (\alpha \cdot \nu) (\gamma_D^+ f_+ -  \gamma_D^- f_-) \quad \text{and} \quad \Gamma_1^{(s)} f := \frac{1}{2} (\gamma_D^+ f_+ + \gamma_D^- f_-).
\end{equation}
Observe that $\Gamma_0^{(s)}$ and $\Gamma_1^{(s)}$ are both well defined due to Lemma~\ref{lemma_trace_theorem}.

In the following theorem we show that the mappings $\Gamma_0^{(s)}$ and $\Gamma_1^{(s)}$ in \eqref{def_Gamma} give rise to a quasi boundary triple for $S^*$ and we compute the associated $\gamma$-field and Weyl function. Recall that $A_0$ is the free Dirac operator defined in~\eqref{def_A_0}, that $\Phi_z$ and $\mathcal{C}_z$ are the mappings introduced in~\eqref{def_Phi_lambda} and~\eqref{def_C_lambda}, respectively, and that $H^s_\alpha(\Sigma; \mathbb{C}^N)$ is the space defined in~\eqref{def_H_U} for $\mathbb{U} = \alpha \cdot \nu$.

\begin{theorem} \label{theorem_QBT}
  Let $s \in [\frac{1}{2},1]$. Then the following holds:   
  \begin{itemize}
    \item[\rm (i)] The triple $\big\{ L^2(\Sigma; \mathbb{C}^N), \Gamma_0^{(s)}, \Gamma_1^{(s)} \big\}$ is a quasi boundary triple for $S^*=\overline{T^{(s)}}$ with $T^{(s)} \upharpoonright \ker \Gamma_0^{(s)} = A_0$, and one has
  \begin{equation} \label{ran_Gamma}
    \ran \Gamma_0^{(s)} = H^{s-1/2}_\alpha(\Sigma; \mathbb{C}^N).
  \end{equation}
  In particular, $\big\{ L^2(\Sigma; \mathbb{C}^N), \Gamma_0^{(1/2)}, \Gamma_1^{(1/2)} \big\}$ is a generalized boundary triple.
    
    \item[(ii)] For $z \in \rho(A_0)= \mathbb{C} \setminus ( (-\infty, -m] \cup [m, \infty) )$ the values $\gamma^{(s)}(z)$ of the $\gamma$-field 
       are given by 
    \begin{equation*} 
      \begin{split}
        \gamma^{(s)}(z) = \Phi_z \upharpoonright H^{s-1/2}_\alpha(\Sigma; \mathbb{C}^N).
      \end{split}
    \end{equation*}
    Each $\gamma^{(s)}(z)$ is a densely defined bounded operator 
    from the Hilbert space $L^2(\Sigma; \mathbb{C}^N)$ to $L^2(\mathbb{R}^q; \mathbb{C}^N)$ and 
    an everywhere defined  bounded operator from $H^{s-1/2}_\alpha(\Sigma; \mathbb{C}^N)$ to 
    $H^s_\alpha(\mathbb{R}^q \setminus \Sigma; \mathbb{C}^N)$. In particular, $\gamma^{(1/2)}(z)$ is a bounded and everywhere defined operator from $L^2(\Sigma; \mathbb{C}^N)$ to $L^2(\mathbb{R}^q; \mathbb{C}^N)$. 
    Moreover, 
    \begin{equation*}
      \gamma^{(s)}(z)^*: L^2(\mathbb{R}^q; \mathbb{C}^N) \rightarrow L^2(\Sigma; \mathbb{C}^N)
    \end{equation*}
    is compact.
   
    \item[\rm (iii)] For $z \in \rho(A_0)= \mathbb{C} \setminus ( (-\infty, -m] \cup [m, \infty) )$ the values $M^{(s)}(z)$ of the Weyl function
    are given by
    \begin{equation*}
      \begin{split}
        M^{(s)}(z)=\mathcal{C}_z \upharpoonright H^{s-1/2}_\alpha(\Sigma; \mathbb{C}^N).
      \end{split}
    \end{equation*}
    Each $M^{(s)}(z)$ is a densely defined bounded operator in $L^2(\Sigma; \mathbb{C}^N)$ and a  bounded everywhere defined operator from $H^{s-1/2}_\alpha(\Sigma; \mathbb{C}^N)$ to $H^{s-1/2}(\Sigma; \mathbb{C}^N)$. In particular, $M^{(1/2)}(z)$ is a bounded and everywhere defined operator in $L^2(\Sigma; \mathbb{C}^N)$.
  \end{itemize}
\end{theorem}
\begin{proof}
  (i) Let $s \in [\frac{1}{2}, 1]$ be fixed. First, we show that $\big\{ L^2(\Sigma; \mathbb{C}^N), \Gamma_0^{(s)}, \Gamma_1^{(s)} \big\}$ is a quasi boundary triple. For this we note that $\overline{T^{(s)}}=T^{(0)}=S^*$ holds, as the set $C_0^\infty(\overline{\Omega_+}; \mathbb{C}^N) \oplus C_0^\infty(\overline{\Omega_-}; \mathbb{C}^N) \subset \dom T^{(s)}$ is dense in $H^0_\alpha(\mathbb{R}^q \setminus \Sigma; \mathbb{C}^N)$, while the norm in $H^0_\alpha(\mathbb{R}^q\setminus \Sigma; \mathbb{C}^N)$ and the graph norm induced by $T^{(0)}$ are equivalent; cf. the discussion at the beginning of Section~\ref{subsec_SobolevSpaces}. With~\eqref{equation_adjoint} this yields $\overline{T^{(s)}}=T^{(0)}=S^*$.

  Let us verify the abstract Green's identity in Definition~\ref{qbtdef}~(i). For this consider \[f = f_+ \oplus f_-,\,g = g_+ \oplus g_- \in \dom T^{(s)} = H^s_\alpha(\Omega_+; \mathbb{C}^N) \oplus H^s_\alpha(\Omega_-; \mathbb{C}^N).\]
  Then integration by parts~\eqref{int_by_parts} applied in $\Omega_\pm$ yields
  \begin{equation*}
    \begin{split}
      \big( (-i (\alpha \cdot \nabla) + m \alpha_0) f_\pm, g_\pm \big)_{L^2(\Omega_\pm; \mathbb{C}^N)}
     - \big( f_\pm, (-i (\alpha \cdot \nabla) + m \alpha_0) g_\pm \big)_{L^2(\Omega_\pm; \mathbb{C}^N)}~~& \\
     = \pm \big( -i (\alpha \cdot \nu) \gamma_D^\pm f_\pm, \gamma_D^\pm g_\pm \big)_{L^2\Sigma; \mathbb{C}^N)}&.
    \end{split}
  \end{equation*}
  By adding these two formulas for $\Omega_+$ and $\Omega_-$ and using
  \begin{equation*}
    \begin{split}
      \big( \Gamma_1^{(s)} f&, \Gamma_0^{(s)} g \big)_{L^2(\Sigma; \mathbb{C}^N)} - \big( \Gamma_0^{(s)} f, \Gamma_1^{(s)} g \big)_{L^2(\Sigma; \mathbb{C}^N)} \\
      &=\left( \frac{1}{2} (\gamma_D^+ f_+ + \gamma_D^- f_-), i (\alpha \cdot \nu ) (\gamma_D^+ g_+ - \gamma_D^- g_-) \right)_{L^2(\Sigma; \mathbb{C}^N)} \\
      &\qquad - \left( i (\alpha \cdot \nu ) (\gamma_D^+ f_+ - \gamma_D^- f_-), \frac{1}{2} (\gamma_D^+ g_+ + \gamma_D^- g_-) \right)_{L^2(\Sigma; \mathbb{C}^N)} \\
      &= \big( -i (\alpha \cdot \nu) \gamma_D^+ f_+, \gamma_D^+ g_+ \big)_{L^2(\Sigma; \mathbb{C}^N)} - \big( -i (\alpha \cdot \nu) \gamma_D^- f_-, \gamma_D^- g_- \big)_{L^2(\Sigma; \mathbb{C}^N)},
    \end{split}
  \end{equation*}
  one arrives at the abstract Green's identity.

  Next, we verify  
  \begin{equation}\label{123}
  A_0=T^{(s)} \upharpoonright \ker \Gamma_0^{(s)}.
  \end{equation}
  In fact, the inclusion $A_0 \subset T^{(s)} \upharpoonright \ker \Gamma_0^{(s)}$ in \eqref{123} is clear.
  On the other hand, the abstract Green's identity shown above implies that $T^{(s)} \upharpoonright \ker \Gamma_0^{(s)}$ is symmetric. As the free Dirac operator $A_0$ is self-adjoint we obtain \eqref{123} and hence 
  $T^{(s)} \upharpoonright \ker \Gamma_0^{(s)}$ is self-adjoint.
  
  It remains to prove that $\ran(\Gamma_0^{(s)}, \Gamma_1^{(s)})$ is dense in $L^2(\Sigma; \mathbb{C}^N) \times L^2(\Sigma; \mathbb{C}^N)$. For this, we prove 
  \begin{equation} \label{ran_Gamma_0}
    \ran(\Gamma_0^{(s)} \upharpoonright \ker \Gamma_1^{(s)}) = H^{1/2}_\alpha(\Sigma; \mathbb{C}^N)
  \end{equation}
  and 
  \begin{equation} \label{ran_Gamma_1}
    \ran(\Gamma_1^{(s)} \upharpoonright \ker \Gamma_0^{(s)}) = H^{1/2}(\Sigma; \mathbb{C}^N).
  \end{equation}
  To establish the inclusion ``$\subset$'' in~\eqref{ran_Gamma_0} we note that any $f \in \ker \Gamma_1^{(s)}$ satisfies $\gamma_D^+ f_+=- \gamma_D^- f_-$, which yields  that $f_+\oplus(-f_-) \in \ker \Gamma_0^{(s)} = \dom A_0 = H^1(\mathbb{R}^q; \mathbb{C}^N)$ and thus, $f \in H^1(\Omega_+; \mathbb{C}^N) \oplus H^1(\Omega_-; \mathbb{C}^N)$. Therefore, the claimed inclusion follows from the definition of $\Gamma_0^{(s)}$. For the converse inclusion ``$\supset$'' let $\varphi \in H^{1/2}_\alpha(\Sigma; \mathbb{C}^N)$. Choose $f_\pm \in H^1(\Omega_\pm; \mathbb{C}^N)$ such that $\gamma_D^\pm f_\pm = \mp \frac{i}{2} (\alpha \cdot \nu) \varphi$. Then $f \in \ker \Gamma_1^{(s)}$ and $\Gamma_0^{(s)} f = \varphi$, which shows the second inclusion in~\eqref{ran_Gamma_0}.
  To verify~\eqref{ran_Gamma_1} note that the inclusion ``$\subset$'' follows from $\ker \Gamma_0^{(s)} = H^1(\mathbb{R}^q; \mathbb{C}^N)$ and the definition of $\Gamma_1^{(s)}$. For the converse inclusion ``$\supset$''  let $\varphi \in H^{1/2}(\Sigma; \mathbb{C}^N)$. Choose $f \in H^1(\mathbb{R}^q; \mathbb{C}^N)$ such that $\gamma_D f = \varphi$. Then $f \in \ker \Gamma_0^{(s)}$ and $\Gamma_1^{(s)} f = \varphi$, which shows the other inclusion in~\eqref{ran_Gamma_1}.

  Hence, we have shown that $\big\{ L^2(\Sigma; \mathbb{C}^N), \Gamma_0^{(s)}, \Gamma_1^{(s)} \big\}$ is indeed a quasi boundary triple for all $s \in [\frac{1}{2},1]$. Thus, except for formula~\eqref{ran_Gamma} assertion~(i) has been shown. Equation~\eqref{ran_Gamma} will be proved together with items~(ii) and~(iii).

  Proof of (ii) and (iii): We verify that $\gamma^{(s)}(z)^*$ is compact for all $s$ and all $z \in \rho(A_0)$. For this purpose, recall that formula~\eqref{equation_gamma_star} implies $\gamma^{(s)}(z)^*=\Gamma_1^{(s)} (A_0-\overline{z})^{-1}$. Since $(A_0-\overline{z})^{-1}: L^2(\mathbb{R}^q; \mathbb{C}^N) \rightarrow H^1(\mathbb{R}^q; \mathbb{C}^N)$ is bounded, we see that $\gamma^{(s)}(z)^*$ is actually independent of $s$ and furthermore, that $\gamma^{(s)}(z)^*: L^2(\mathbb{R}^q; \mathbb{C}^N) \rightarrow H^{1/2}(\Sigma; \mathbb{C}^N)$ is also bounded. Since $H^{1/2}(\Sigma; \mathbb{C}^N)$ is compactly embedded in $L^2(\Sigma; \mathbb{C}^N)$, the claimed compactness of $\gamma^{(s)}(z)^*$ follows.
  
  The remaining assertions will be proved in three separate steps.

  \emph{Step 1.} Let $s = \frac{1}{2}$. 
   Consider for $\varphi \in L^2(\Sigma; \mathbb{C}^N)$ the function $f_z := \Phi_z \varphi$. Then  by Lemma~\ref{lemma_Phi_mapping_properties} we have $f_z \in H^{1/2}_\alpha(\mathbb{R}^q \setminus \Sigma; \mathbb{C}^N) = \dom T^{(1/2)}$ and by~\eqref{trace_Phi_lambda} we get $\Gamma_0^{(1/2)} f_z = \varphi$. Therefore, $\ran \Gamma_0^{(1/2)}=L^2(\Sigma; \mathbb{C}^N)$, which is~\eqref{ran_Gamma} for $s=\frac{1}{2}$. Moreover, as $G_{z,q}$ in~\eqref{def_G_lambda} is a fundamental solution for the Dirac equation the definition of $\Phi_z$ shows that
  \begin{equation*}
    \big( T^{(1/2)}-z\big) f_z = 0\quad \text{in } \mathbb{R}^q \setminus \Sigma.
  \end{equation*}
  Hence, $\gamma^{(1/2)}(z)=\Phi_z$. Finally, using the definition of $\Gamma_1^{(1/2)}$ and~\eqref{trace_Phi_lambda} it follows that $M^{(1/2)}(z)=\mathcal{C}_z$ and thus, $M^{(1/2)}(z)$ is bounded in $L^2(\Sigma; \mathbb{C}^N)$. This shows all claims for $s=\frac{1}{2}$.
  
  \emph{Step 2.} Let $s = 1$. First note that $\dom T^{(1)} = H^1(\mathbb{R}^q \setminus \Sigma; \mathbb{C}^N)$, the definition of $\Gamma_0^{(1)}$, and~\eqref{ran_Gamma_0} imply  $\ran \Gamma_0^{(1)}=H^{1/2}_\alpha(\Sigma; \mathbb{C}^N)$. As $\{ L^2(\Sigma; \mathbb{C}^N), \Gamma_0^{(1)}, \Gamma_1^{(1)}\}$ is a restriction of the triple for $s=\frac{1}{2}$ we deduce from the already shown results that 
  \begin{equation*}
    \gamma^{(1)}(z)=\gamma^{(1/2)}(z)\upharpoonright \ran \Gamma_0^{(1)} = \Phi_z \upharpoonright H^{1/2}_\alpha(\Sigma; \mathbb{C}^N)
  \end{equation*}
  and 
  \begin{equation*}
    M^{(1)}(z)=M^{(1/2)}(z)\upharpoonright \ran \Gamma_0^{(1)} = \mathcal{C}_z\upharpoonright H^{1/2}_\alpha(\Sigma; \mathbb{C}^N).
  \end{equation*}
  Using the closed graph theorem and the fact that $H^{1/2}_\alpha(\Sigma; \mathbb{C}^N)$ and $H^1(\mathbb{R}^q \setminus \Sigma; \mathbb{C}^N)$ are continuously embedded in $L^2(\Sigma; \mathbb{C}^N)$ and $L^2(\mathbb{R}^q; \mathbb{C}^N)$, respectively, one gets that 
  \begin{equation*}
    \gamma^{(1)}(z): \ran \Gamma_0^{(1)}=H^{1/2}_\alpha(\Sigma; \mathbb{C}^N) \rightarrow \dom T^{(1)} = H^1(\mathbb{R}^q \setminus \Sigma; \mathbb{C}^N)
  \end{equation*}
  is bounded as well. The mapping properties of the trace map yield that also
  \begin{equation*}
    M^{(1)}(z): \ran \Gamma_0^{(1)}=H^{1/2}_\alpha(\Sigma; \mathbb{C}^N) \rightarrow \ran \Gamma_1^{(1)} = H^{1/2}(\Sigma; \mathbb{C}^N)
  \end{equation*}
  is bounded. Hence, all claimed statements for $s=1$ have been shown.
  
  \emph{Step 3.} Let $s \in (\frac{1}{2},1)$. First we note that an interpolation argument, which is applicable due to \eqref{interpolation_H_alpha_Omega} and~\eqref{interpolation_H_alpha_Sigma1}, shows that  
  \[\Phi_z: H^{s-1/2}_\alpha(\Sigma; \mathbb{C}^N) \rightarrow H^s_\alpha(\mathbb{R}^q \setminus \Sigma; \mathbb{C}^N) = \dom T^{(s)}\] 
  is bounded. Together with~\eqref{trace_Phi_lambda} this implies that $\ran \Gamma_0^{(s)} = H^{s-1/2}_\alpha(\Sigma; \mathbb{C}^N)$, i.e.~\eqref{ran_Gamma} holds for $s \in(\frac{1}{2},1)$. Hence, we have $\gamma^{(s)}(z)=\Phi_z\upharpoonright H^{s-1/2}_\alpha(\Sigma; \mathbb{C}^N)$ and the trace theorem shows that 
  \[M^{(s)}(z)=\Gamma_1^{(s)} \gamma^{(s)}(z): H^{s-1/2}_\alpha(\Sigma; \mathbb{C}^N) \rightarrow H^{s-1/2}(\Sigma; \mathbb{C}^N)\]
   is bounded. Thus, all claims are proved.
\end{proof}

As a consequence of Theorem~\ref{theorem_QBT} we can construct extensions of $\Phi_z$ and $\mathcal{C}_z$ and show their mapping properties. The following result is a generalization of \cite[Proposition~4.4]{BH20} and \cite[Theorem~2.2 and Corollary~2.3]{OV16}. Recall that the multiplication of $\varphi \in H^{-1/2}(\Sigma; \mathbb{C}^N)$ by $\alpha \cdot \nu$ is understood in the sense of~\eqref{mapping_properties_alpha}.

\begin{proposition} \label{proposition_extensions_Phi_C}
  Let $z \in \rho(A_0)= \mathbb{C} \setminus ( (-\infty, -m] \cup[m,\infty) )$. Then the following holds:
  \begin{itemize}
    \item[(i)] The map $\Phi_z$ admits a unique continuous extension 
    $$\widetilde{\Phi}_z: H^{-1/2}(\Sigma; \mathbb{C}^N) \rightarrow H^0_\alpha(\mathbb{R}^q \setminus \Sigma; \mathbb{C}^N)$$ 
    and $\ran \widetilde{\Phi}_z = \ker (T^{(0)} - z)$ holds.
    \item[(ii)] The map $\mathcal{C}_z$ admits a unique continuous extension 
    $$\widetilde{\mathcal{C}}_z: H^{-1/2}(\Sigma; \mathbb{C}^N) \rightarrow H^{-1/2}_\alpha(\Sigma; \mathbb{C}^N)$$ 
    and 
    \begin{equation} \label{duality_C_z}
      \begin{split}
        \big\langle \widetilde{\mathcal{C}}_z \varphi, \psi \big\rangle_{H^{-1/2}_\alpha(\Sigma; \mathbb{C}^N) \times H^{1/2}_\alpha(\Sigma; \mathbb{C}^N)}
         = \big\langle  \varphi, \mathcal{C}_{\overline{z}} \psi \big\rangle_{H^{-1/2}(\Sigma; \mathbb{C}^N) \times H^{1/2}(\Sigma; \mathbb{C}^N)}
      \end{split}
    \end{equation}
    holds for all $\varphi \in H^{-1/2}(\Sigma; \mathbb{C}^N)$ and $\psi \in H^{1/2}_\alpha(\Sigma; \mathbb{C}^N)$.
    \item[(iii)] The map $L^2(\Sigma; \mathbb{C}^N) \ni \varphi \mapsto \gamma_D^\pm (\Phi_z \varphi)_\pm$ admits a unique continuous extension to a map from $H^{-1/2}(\Sigma; \mathbb{C}^N)$ to $H^{-1/2}_\alpha(\Sigma; \mathbb{C}^N)$ and
    $$\gamma_D^\pm (\widetilde{\Phi}_z \varphi)_\pm = \mp \frac{i}{2} (\alpha \cdot \nu) \varphi + \widetilde{\mathcal{C}}_z \varphi$$
    holds for all $\varphi\in H^{-1/2}(\Sigma; \mathbb{C}^N)$.
    \item[(iv)] For all $z_1,z_2 \in \rho(A_0)$ the operator $\widetilde{\mathcal{C}}_{z_1} - \widetilde{\mathcal{C}}_{z_2}: H^{-1/2}(\Sigma; \mathbb{C}^N) \rightarrow H^{1/2}(\Sigma; \mathbb{C}^N)$ is bounded.
  \end{itemize}
\end{proposition}
\begin{proof}
  (i) Since the map $\gamma^{(1)}(z)^* = \Gamma^{(1)}_1 (A_0 - \overline{z})^{-1} = \gamma_D (A_0-\overline{z})^{-1}$, cf.~\eqref{equation_gamma_star}, is bounded from $L^2(\mathbb{R}^q; \mathbb{C}^N)$ to $H^{1/2}(\Sigma; \mathbb{C}^N)$ with $\ker \gamma^{(1)}(z)^* = \ran(S-\overline{z})$, we can define its anti-dual
  \begin{equation*}
    \widetilde{\Phi}_z := (\gamma^{(1)}(z)^*)': H^{-1/2}(\Sigma; \mathbb{C}^N) \rightarrow L^2(\mathbb{R}^q; \mathbb{C}^N).
  \end{equation*}
  As $\ran \gamma^{(1)}(z)^* = H^{1/2}(\Sigma; \mathbb{C}^N)$ is closed, by the closed range theorem the same is true for $\ran \widetilde{\Phi}_z$  and this together with \eqref{equation_adjoint} implies that
  \begin{equation*}
    \ran \widetilde{\Phi}_z = \big( \ker \gamma^{(1)}(z)^* \big)^\bot = \big( \ran(S-\overline{z}) \big)^\bot = \ker (T^{(0)} - z).
  \end{equation*}
  Since the norms in $H^0_\alpha(\mathbb{R}^q \setminus \Sigma; \mathbb{C}^N)$ and $L^2(\mathbb{R}^q; \mathbb{C}^N)$ are equivalent on $\ker (T^{(0)} - z)$, the operator $\widetilde{\Phi}_z$ has the claimed mapping properties. Furthermore, we have for $\varphi \in L^2(\Sigma; \mathbb{C}^N)$ and $f \in L^2(\mathbb{R}^q; \mathbb{C}^N)$ that
  \begin{equation*} 
    \begin{split}
      \big( \widetilde{\Phi}_z \varphi, f \big)_{L^2(\mathbb{R}^q; \mathbb{C}^N)}
        &= \langle \varphi, \gamma^{(1)}(z)^* f \rangle_{H^{-1/2}(\Sigma; \mathbb{C}^N) \times H^{1/2}(\Sigma; \mathbb{C}^N)} \\
        &= ( \varphi, \Phi_z^* f )_{L^2(\Sigma; \mathbb{C}^N)} 
        =( \Phi_z \varphi, f )_{L^2(\mathbb{R}^q; \mathbb{C}^N)},
    \end{split}
  \end{equation*}
  which shows that $\widetilde{\Phi}_z$ is indeed an extension of $\Phi_z$. Eventually, as $L^2(\Sigma; \mathbb{C}^N)$ is dense in $H^{-1/2}(\Sigma; \mathbb{C}^N)$, there can only be one continuous extension $\widetilde{\Phi}_z$ of $\Phi_z$ defined on $H^{-1/2}(\Sigma; \mathbb{C}^N)$.
  
  (ii) Using the result from Theorem~\ref{theorem_QBT}~(iii) for $s=1$ we can define the anti-dual map
  \begin{equation*}
    \widetilde{\mathcal{C}}_z := (\mathcal{C}_{\overline{z}})': H^{-1/2}(\Sigma; \mathbb{C}^N) \rightarrow H^{-1/2}_\alpha(\Sigma; \mathbb{C}^N).
  \end{equation*}
  Then, by definition the relation~\eqref{duality_C_z} holds. It remains to show that $\widetilde{\mathcal{C}}_z$ is an extension of $\mathcal{C}_z$. For that we note that \eqref{equation_diff_m} and Theorem~\ref{theorem_QBT}~(iii) imply $\mathcal{C}_{\overline{z}}^* = \mathcal{C}_z$.  
  Using this and~\eqref{duality_C_z} we get for $\varphi \in L^2(\Sigma; \mathbb{C}^N)$ and $\psi \in H^{1/2}_\alpha(\Sigma; \mathbb{C}^N)$
  \begin{equation*} 
    \begin{split}
      \big\langle \widetilde{\mathcal{C}}_z \varphi, \psi \big\rangle_{H^{-1/2}_\alpha(\Sigma; \mathbb{C}^N) \times H^{1/2}_\alpha(\Sigma; \mathbb{C}^N)}
      &= \big\langle  \varphi, \mathcal{C}_{\overline{z}} \psi \big\rangle_{H^{-1/2}(\Sigma; \mathbb{C}^N) \times H^{1/2}(\Sigma; \mathbb{C}^N)} \\
       &= \big(\varphi, \mathcal{C}_{\overline{z}} \psi \big)_{L^{2}(\Sigma; \mathbb{C}^N)} \\
       & = \big( \mathcal{C}_z \varphi, \psi \big)_{L^{2}(\Sigma; \mathbb{C}^N)} \\
       &=  \big\langle \mathcal{C}_z \varphi, \psi \big\rangle_{H^{-1/2}_\alpha(\Sigma; \mathbb{C}^N) \times H^{1/2}_\alpha(\Sigma; \mathbb{C}^N)},
    \end{split}
  \end{equation*}
  which shows that $\widetilde{\mathcal{C}}_z \varphi = \mathcal{C}_z  \varphi$. This finishes the proof of item~(ii).

  (iii) By~\eqref{trace_Phi_lambda} we know that
  \begin{equation*}
    \gamma_D^\pm (\Phi_z \varphi )_\pm = \mp \frac{i}{2} (\alpha \cdot \nu) \varphi + \mathcal{C}_z \varphi
  \end{equation*}
  holds for all $\varphi \in L^2(\Sigma; \mathbb{C}^N)$. Since $L^2(\Sigma; \mathbb{C}^N)$ is dense in $H^{-1/2}(\Sigma; \mathbb{C}^N)$, item~(ii) and~\eqref{mapping_properties_alpha} show that the right hand side can be extended to a continuous operator $\varphi \ni H^{-1/2}(\Sigma; \mathbb{C}^N) \mapsto \gamma_D^\pm(\widetilde{\Phi}_z \varphi)_\pm \in H^{-1/2}_\alpha(\Sigma; \mathbb{C}^N)$. This shows the claim of assertion~(iii).
  
  (iv) Using~\eqref{equation_diff_m} for the quasi boundary triple $\{ L^2(\Sigma; \mathbb{C}^N), \Gamma_0^{(1/2)}, \Gamma_1^{(1/2)} \}$ we see 
  \begin{equation*}
    \mathcal{C}_{z_1} - \mathcal{C}_{z_2} = (z_1-z_2) \Phi_{\overline{z_2}}^* \Phi_{z_1}.
  \end{equation*}
  By continuity and assertions~(i) and~(ii) this can be extended to 
  \begin{equation*}
    \widetilde{\mathcal{C}}_{z_1} - \widetilde{\mathcal{C}}_{z_2} = (z_1-z_2) \Phi_{\overline{z_2}}^* \widetilde{\Phi}_{z_1}.
  \end{equation*}
  Taking the mapping properties of $\widetilde{\Phi}_{z_1}$ and $\Phi_{\overline{z_2}}^* = \gamma^{(1)}(\overline{z_2})^*$, as they are discussed in the proof of item~(i), into account one finds the claimed mapping properties of $\widetilde{\mathcal{C}}_{z_1} - \widetilde{\mathcal{C}}_{z_2}$.
\end{proof}

Using Theorem~\ref{theorem_QBT}, Proposition~\ref{proposition_extensions_Phi_C}, \eqref{def_T}, and a simple interpolation argument, which is applicable due to \eqref{interpolation_H_alpha_Omega}, \eqref{interpolation_H_alpha_Sigma1}, and \eqref{interpolation_H_alpha_Sigma2}, one obtains further mapping properties of $\Phi_z$ and $\mathcal{C}_z$. The next result is a generalization of \cite[Proposition~3.6 and Proposition~3.8]{BHM20}.

\begin{corollary} \label{corollary_extensions_Phi_C}
  Let $z \in \rho(A_0)=\mathbb{C} \setminus ( (-\infty, -m] \cup[m,\infty) )$ and let $s \in (0,\frac{1}{2})$. Then the following holds:
  \begin{itemize}
    \item[(i)] The map $\Phi_z$ admits a unique continuous extension 
    $$\widetilde{\Phi}^{(s)}_z: H^{s-1/2}(\Sigma; \mathbb{C}^N) \rightarrow H^{s}_\alpha(\mathbb{R}^q \setminus \Sigma; \mathbb{C}^N)$$
    and $\ran \widetilde{\Phi}^{(s)}_z= \ker (T^{(s)} - z)$ holds. 
    \item[(ii)] The map $\mathcal{C}_z$ admits a unique continuous extension 
    $$\widetilde{\mathcal{C}}^{(s)}_z: H^{s-1/2}(\Sigma; \mathbb{C}^N) \rightarrow H^{s-1/2}_\alpha(\Sigma; \mathbb{C}^N)$$ and 
    \begin{equation*}
      \begin{split}
        \big\langle \widetilde{\mathcal{C}}^{(s)}_z \varphi, \psi \big\rangle_{H^{s-1/2}_\alpha(\Sigma; \mathbb{C}^N) \times H^{-s+1/2}_\alpha(\Sigma; \mathbb{C}^N)}
         = \big\langle  \varphi, \mathcal{C}_{\overline{z}} \psi \big\rangle_{H^{s-1/2}(\Sigma; \mathbb{C}^N) \times H^{-s+1/2}(\Sigma; \mathbb{C}^N)}
      \end{split}
    \end{equation*}
    holds for all $\varphi \in H^{s-1/2}(\Sigma; \mathbb{C}^N)$ and $\psi \in H^{-s+1/2}_\alpha(\Sigma; \mathbb{C}^N)$.
  \end{itemize}
\end{corollary}

We note that~\eqref{formula_C_inv} can be extended to the extensions $\widetilde{\mathcal{C}}_z^{(s)}$: For any $s \in [0,\frac{1}{2}]$ the map $\widetilde{\mathcal{C}}^{(s)}_z: H^{s-1/2}(\Sigma; \mathbb{C}^N) \rightarrow H^{s-1/2}_\alpha(\Sigma; \mathbb{C}^N)$ is bijective and 
\begin{equation} \label{formula_C_inv_extension}
  (\widetilde{\mathcal{C}}^{(s)}_z)^{-1} = -4(\alpha \cdot \nu) \widetilde{\mathcal{C}}^{(s)}_z (\alpha \cdot \nu), \qquad z \in \mathbb{C} \setminus \big( (-\infty, -m] \cup [m, \infty) \big),
\end{equation}
holds, where the multiplication by $\alpha \cdot \nu$ is understood in the sense of~\eqref{mapping_properties_alpha}.

Eventually, we note that the result in Proposition~\ref{proposition_extensions_Phi_C} allows one to extend the Dirichlet trace operator to $H^s_\alpha(\Omega_\pm; \mathbb{C}^N)$ for $s \in [0,\frac{1}{2})$.

\begin{corollary} \label{corollary_trace_extension}
   Let $s \in [0,\frac{1}{2}]$. Then the Dirichlet trace map admits a unique continuous extension $\gamma_D^\pm: H^s_\alpha(\Omega_\pm; \mathbb{C}^N) \rightarrow H^{s-1/2}_\alpha(\Sigma; \mathbb{C}^N)$.
\end{corollary}
\begin{proof}
  We prove the assertion for $\Omega_+$ and $s=0$; the general statement for $\Omega_+$ follows from this, Lemma~\ref{lemma_trace_theorem}, and an interpolation argument (see \eqref{interpolation_H_alpha_Omega} and \eqref{interpolation_H_alpha_Sigma2}). For $\Omega_-$ the result can be shown in the same way.
  
  Consider the continuous embedding operator
  \begin{equation*}
    \iota_1: H^0_\alpha(\Omega_+; \mathbb{C}^N) \rightarrow H^0_\alpha(\mathbb{R}^q \setminus \Sigma; \mathbb{C}^N), \quad \iota_1 f = f \oplus 0,
  \end{equation*}
  let $T^{(0)}$ be the operator defined in~\eqref{def_T} for $s=0$ and assume here that $m>0$ (this is no restriction; the statement does not depend on $m$). Since $A_0 \subset T^{(0)}$ and $0 \in \rho(A_0)$, the direct sum decomposition
  \begin{equation*}
    H^0_\alpha(\mathbb{R}^q \setminus \Sigma; \mathbb{C}^N) = \dom T^{(0)} = \dom A_0 \dot{+} \ker T^{(0)}
  \end{equation*}
  holds and we see with Proposition~\ref{proposition_extensions_Phi_C}~(i) that for any $f \in H^0_\alpha(\Omega_+; \mathbb{C}^N)$ there exist unique elements $f_0 \in \dom A_0 = H^1(\mathbb{R}^q; \mathbb{C}^N) $ and $\varphi \in H^{-1/2}(\Sigma; \mathbb{C}^N)$ such that $\iota_1 f = f_0 + \widetilde{\Phi}_0 \varphi$. Hence, in view of Proposition~\ref{proposition_extensions_Phi_C}~(iii) the trace of $f$ can be defined as
  \begin{equation} \label{def_trace}
    \gamma_D^+ f := \gamma_D f_0 - \frac{i}{2} (\alpha \cdot \nu) \varphi + \widetilde{\mathcal{C}}_0 \varphi \in H^{-1/2}_\alpha(\Sigma; \mathbb{C}^N).
  \end{equation}
  Moreover, as the map
  \begin{equation*}
    \iota_2: H^0_\alpha(\Omega_+; \mathbb{C}^N) \rightarrow H^1(\mathbb{R}^q; \mathbb{C}^N) \times H^{-1/2}(\Sigma; \mathbb{C}^N), \quad \iota_2 f = \iota_2(f_0 + \widetilde{\Phi}_0 \varphi)_+ = (f_0, \varphi),
  \end{equation*}
  is closed, it is also continuous. Therefore, the continuity of $\iota_2$ and $\widetilde{\mathcal{C}}_0$ imply the continuity of the trace map in~\eqref{def_trace}.
\end{proof}

Finally, we mention that Corollary~\ref{corollary_trace_extension} allows us to extend the statement in~\eqref{int_by_parts} for $f \in H^1(\Omega_\pm; \mathbb{C}^N)$ and $g  \in H^0_\alpha(\Omega_\pm; \mathbb{C}^N)$ to
\begin{equation} \label{int_by_parts_extension}
  \begin{split}
    (i (\alpha \cdot \nabla) f, g)_{L^2(\Omega_\pm; \mathbb{C}^N)} 
    &= (f, i (\alpha \cdot \nabla) g)_{L^2(\Omega_\pm; \mathbb{C}^N)}\\
    &\quad\pm \langle i (\alpha \cdot \nu) \gamma_D^\pm f, \gamma_D^\pm g\rangle_{H^{1/2}_\alpha(\Sigma; \mathbb{C}^N) \times H^{-1/2}_\alpha(\Sigma; \mathbb{C}^N)}.
  \end{split}
\end{equation}

\subsection{An ordinary boundary triple for Dirac operators with singular interactions} \label{section_OBT}

In this subsection we introduce an ordinary boundary triple that is suitable for the investigation of Dirac operators with singular potentials in the so-called critical case; cf. Section~\ref{critical-section}. For that, the operator $A_\infty$ defined by
\begin{equation} \label{def_A_infty}
  \begin{split}
    A_\infty f &= (-i (\alpha \cdot \nabla) + m \alpha_0) f_+ \oplus (-i (\alpha \cdot \nabla) + m \alpha_0) f_-, \\
    \dom A_\infty &= \big\{ f \in H^1(\mathbb{R}^q \setminus \Sigma; \mathbb{C}^N): \gamma_D^+ f_+ = - \gamma_D^- f_-\big\},
  \end{split}
\end{equation}
will play an important role. The basic properties of $A_\infty$ are stated in the following lemma:

\begin{lemma} \label{lemma_A_infty}
  The operator $A_\infty$ is self-adjoint in $L^2(\mathbb{R}^q; \mathbb{C}^N)$ and its spectrum is given by $\sigma(A_\infty) = (-\infty,-m] \cup [m,\infty)$.
\end{lemma}
\begin{proof}
  Consider in $L^2(\mathbb{R}^q; \mathbb{C}^N)$ the operator $U$ given by
  \begin{equation*}
    U f = f_+ \oplus (-f_-), \qquad f = f_+ \oplus f_- \in L^2(\mathbb{R}^q; \mathbb{C}^N).
  \end{equation*}
  Then it is not difficult to see that $U$ is unitary and self-adjoint and $A_\infty = U A_0 U$. Hence, $A_\infty$ is self-adjoint and $\sigma(A_\infty)=\sigma(A_0) = (-\infty, -m] \cup[m, \infty)$.
\end{proof}

Now, we are prepared to introduce the desired ordinary boundary triple. Let $V: \Sigma \rightarrow \mathbb{C}^{N \times N}$ be a measurable function such that $V(x)$ is a unitary matrix for any $x \in \Sigma$ and let $H_{(\alpha \cdot \nu) V^*}^t(\Sigma; \mathbb{C}^N)$, $t \in [-\frac{1}{2}, \frac{1}{2}]$, be the space defined in~\eqref{def_H_U} and~\eqref{def_H_U_dual} for $\mathbb{U} =  (\alpha \cdot \nu) V^*$. We remark that depending on the type of the singular interaction different choices of $V$ are convenient and hence, we keep it in a general form here. In particular, when one is interested in combinations of electrostatic and Lorentz scalar interactions, then $V=I_N$ is often used, while the choice of $V$ in~\eqref{def_V} below is convenient when also anomalous magnetic interactions are considered.

By construction of $H_{(\alpha \cdot \nu) V^*}^t(\Sigma; \mathbb{C}^N)$ the map $V: H_{\alpha}^t(\Sigma; \mathbb{C}^N) \rightarrow H_{(\alpha \cdot \nu) V^*}^t(\Sigma; \mathbb{C}^N)$ is bijective.
Next, let $\Lambda$ be a self-adjoint operator in $L^2(\Sigma; \mathbb{C}^N)$ which can be viewed for any $s \in [-\frac{1}{2},0]$ as an isomorphism from $H^{s+1/2}_{(\alpha \cdot \nu) V^*}(\Sigma; \mathbb{C}^N)$ to $H^s_{(\alpha \cdot \nu) V^*}(\Sigma; \mathbb{C}^N)$.  Note that this choice  of $\Lambda$ implies that
\begin{equation} \label{equation_Lambda}
  \langle f, g \rangle_{H^{-1/2}_{(\alpha \cdot \nu)V^*}(\Sigma; \mathbb{C}^N) \times H^{1/2}_{(\alpha \cdot \nu)V^*}(\Sigma; \mathbb{C}^N)} = (\Lambda^{-1} f, \Lambda g)_{L^2(\Sigma; \mathbb{C}^N)}
\end{equation}
holds for all $f \in H^{-1/2}_{(\alpha \cdot \nu)V^*}(\Sigma; \mathbb{C}^N)$ and $g \in H^{1/2}_{(\alpha \cdot \nu)V^*}(\Sigma; \mathbb{C}^N)$.

In the statement of the next theorem, the Dirichlet trace operators $\gamma_D^\pm$ are understood in the sense of Corollary~\ref{corollary_trace_extension}. Recall that $T^{(0)}$ is defined in~\eqref{def_T} for $s=0$ and that $\widetilde{\Phi}_z$ and $\widetilde{\mathcal{C}}_z$ are the extensions of $\Phi_z$ and $\mathcal{C}_z$ from Proposition~\ref{proposition_extensions_Phi_C}, respectively. For simplicity, we assume that $m > 0$ in the theorem below; the case $m=0$ is commented in Remark~\ref{remark_m_0}.

\begin{theorem} \label{theorem_OBT}
  Let $m>0$. Define the maps $\Upsilon_0, \Upsilon_1: H^0_\alpha(\mathbb{R}^q \setminus \Sigma; \mathbb{C}^N) \rightarrow L^2(\Sigma; \mathbb{C}^N)$ by
  \begin{equation} \label{def_Upsilon}
    \begin{split}
      \Upsilon_0 f &:= \frac{1}{2} \Lambda^{-1} V (\gamma_D^+ f_+ + \gamma_D^- f_-),\\
      \Upsilon_1 f &:= -\Lambda \big( i V (\alpha \cdot \nu) (\gamma_D^+ f_+ - \gamma_D^- f_-) + 2 V (\alpha \cdot \nu) \widetilde{\mathcal{C}}_0 (\alpha \cdot \nu) (\gamma_D^+ f_+ + \gamma_D^- f_-) \big).
    \end{split}
  \end{equation}
  Then $\{ L^2(\Sigma; \mathbb{C}^N), \Upsilon_0, \Upsilon_1 \}$ is an ordinary boundary triple for $T^{(0)}$ with $T^{(0)} \upharpoonright \ker \Upsilon_0 = A_\infty$. Moreover, the $\gamma$-field $\beta$ and the Weyl function $\mathcal M$ associated with this triple defined on $\rho(A_\infty)=\mathbb{C} \setminus ((-\infty,-m] \cup [m,\infty))$ are given by
  \begin{equation*}
    \rho(A_\infty) \ni z \mapsto \beta(z) = -4 \widetilde{\Phi}_z (\alpha \cdot \nu) \widetilde{\mathcal{C}}_z (\alpha \cdot \nu) V^* \Lambda  
  \end{equation*}
  and
  \begin{equation*}
    \rho(A_\infty) \ni z \mapsto \mathcal{M}(z) = 4 \Lambda V (\alpha \cdot \nu) ( \widetilde{\mathcal{C}}_z - \widetilde{\mathcal{C}}_0 ) (\alpha \cdot \nu) V^* \Lambda.
  \end{equation*}
\end{theorem}

We remark that it is not obvious that all products in the definition of $\Upsilon_1$ are well defined. This will become more clear in the alternative representation of $\Upsilon_1$ in~\eqref{equation_Ups1} below. Moreover, we remark that the Weyl function $\mathcal{M}(z)$ is well defined, as $\widetilde{\mathcal{C}}_z - \widetilde{\mathcal{C}}_0: H^{-1/2}(\Sigma; \mathbb{C}^N) \rightarrow H^{1/2}(\Sigma; \mathbb{C}^N)$ is bounded by Proposition~\ref{proposition_extensions_Phi_C}~(iv).

\begin{proof}[Proof of Theorem~\ref{theorem_OBT}]
  First, let us prove the  abstract Green's identity in Definition~\ref{qbtdef}~(i). For that, we verify alternative representations of $\Upsilon_0$ and $\Upsilon_1$. Since $A_\infty$ is self-adjoint, $A_\infty \subset T^{(0)}$, and $0 \in \rho(A_\infty)$, the direct sum decomposition 
  \begin{equation*}
    \dom T^{(0)} = \dom A_\infty \dot{+} \ker T^{(0)}
  \end{equation*}
  holds. Taking this and Proposition~\ref{proposition_extensions_Phi_C} into account, there exist for each $f \in \dom T^{(0)}$ unique elements $f_\infty \in \dom A_\infty \subset H^1(\mathbb{R}^q \setminus \Sigma; \mathbb{C}^N)$ and $\xi \in H^{-1/2}(\Sigma; \mathbb{C}^N)$ such that $f = f_\infty + \widetilde{\Phi}_0 \xi$. Using Proposition~\ref{proposition_extensions_Phi_C}~(iii) and that $\gamma_D^+ f_{\infty,+} = - \gamma_D^- f_{\infty,-}$ holds for $f_\infty \in \dom A_\infty$ we see that
  \begin{equation} \label{equation_Ups0}
    \Upsilon_0 f = \Lambda^{-1} V \widetilde{\mathcal{C}}_0 \xi.
  \end{equation}  
  Similarly, employing Proposition~\ref{proposition_extensions_Phi_C} and $-4((\alpha \cdot \nu) \widetilde{\mathcal{C}}_0 )^2=I_N$, which holds by~\eqref{formula_C_inv_extension}, we get that
  \begin{equation} \label{equation_Ups1}
    \begin{split}
      \Upsilon_1 f &= -\Lambda \big( i V (\alpha \cdot \nu) (\gamma_D^+ (f_\infty+\widetilde{\Phi}_0 \xi)_+ - \gamma_D^- (f_\infty+\widetilde{\Phi}_0 \xi)_-) \\
      &\qquad \qquad + 2 V (\alpha \cdot \nu) \widetilde{\mathcal{C}}_0 (\alpha \cdot \nu) (\gamma_D^+ (f_\infty+\widetilde{\Phi}_0 \xi)_+ + \gamma_D^- (f_\infty+\widetilde{\Phi}_0 \xi)_-) \big) \\
      &= -\Lambda \big( i V (\alpha \cdot \nu) (\gamma_D^+ f_{\infty,+} - \gamma_D^- f_{\infty,-} - i (\alpha \cdot \nu) \xi)     + 4 V ((\alpha \cdot \nu) \widetilde{\mathcal{C}}_0 )^2 \xi) \big) \\
      &= -\Lambda i V (\alpha \cdot \nu) (\gamma_D^+ f_{\infty,+} - \gamma_D^- f_{\infty,-}).
    \end{split}
  \end{equation}
  Now we are prepared to prove the abstract Green's identity. Fix $f = f_\infty + \widetilde{\Phi}_0 \xi, g = g_\infty + \widetilde{\Phi}_0 \zeta \in \dom T^{(0)}$ with $f_\infty, g_\infty \in \dom A_\infty$ and $\xi, \zeta \in H^{-1/2}(\Sigma; \mathbb{C}^N)$. Using that $\widetilde{\Phi}_0 \xi, \widetilde{\Phi}_0 \zeta \in \ker T^{(0)}$, see Proposition~\ref{proposition_extensions_Phi_C}~(i), and that $A_\infty$ is self-adjoint we get 
  \begin{equation*}
    \begin{split}
      (T^{(0)} f&, g)_{L^2(\mathbb{R}^q; \mathbb{C}^N)} - (f, T^{(0)}  g)_{L^2(\mathbb{R}^q; \mathbb{C}^N)} \\
      &= (A_\infty f_\infty, \widetilde{\Phi}_0 \zeta)_{L^2(\mathbb{R}^q; \mathbb{C}^N)} - (\widetilde{\Phi}_0 \xi, A_\infty  g_\infty)_{L^2(\mathbb{R}^q; \mathbb{C}^N)} \\
      &= \left\langle -i (\alpha \cdot \nu) \gamma_D^+ f_{\infty,+}, -\frac{i}{2} (\alpha \cdot \nu) \zeta + \widetilde{\mathcal{C}}_0 \zeta \right\rangle_{H^{1/2}_\alpha(\Sigma; \mathbb{C}^N) \times H^{-1/2}_\alpha(\Sigma; \mathbb{C}^N)} \\
      &\qquad \qquad + \left\langle i (\alpha \cdot \nu) \gamma_D^- f_{\infty,-}, \frac{i}{2} (\alpha \cdot \nu) \zeta + \widetilde{\mathcal{C}}_0 \zeta \right\rangle_{H^{1/2}_\alpha(\Sigma; \mathbb{C}^N) \times H^{-1/2}_\alpha(\Sigma; \mathbb{C}^N)} \\
      &\qquad \qquad - \left\langle -\frac{i}{2} (\alpha \cdot \nu) \xi + \widetilde{\mathcal{C}}_0 \xi, -i (\alpha \cdot \nu) \gamma_D^+ g_{\infty,+} \right\rangle_{H^{-1/2}_\alpha(\Sigma; \mathbb{C}^N) \times H^{1/2}_\alpha(\Sigma; \mathbb{C}^N)} \\
      &\qquad \qquad - \left\langle \frac{i}{2} (\alpha \cdot \nu) \xi + \widetilde{\mathcal{C}}_0 \xi, i (\alpha \cdot \nu) \gamma_D^- g_{\infty,-} \right\rangle_{H^{-1/2}_\alpha(\Sigma; \mathbb{C}^N) \times H^{1/2}_\alpha(\Sigma; \mathbb{C}^N)},
    \end{split}
  \end{equation*}
  where~\eqref{int_by_parts_extension} in $\Omega_+$ and $\Omega_-$ and Proposition~\ref{proposition_extensions_Phi_C}~(iii) were applied  in the last step.  
  Since $f_\infty, g_\infty \in \dom A_\infty$ satisfy $\gamma_D^+ f_{\infty,+} = -\gamma_D^- f_{\infty,-}$ and $\gamma_D^+ g_{\infty,+} = -\gamma_D^- g_{\infty,-}$, we find now with~\eqref{equation_Ups0},~\eqref{equation_Ups1}, and~\eqref{equation_Lambda} that
  \begin{equation} \label{Green_OBT_Dirac}
    \begin{split}
      (T^{(0)} f&, g)_{L^2(\mathbb{R}^q; \mathbb{C}^N)} - (f, T^{(0)}  g)_{L^2(\mathbb{R}^q; \mathbb{C}^N)} \\
      &= \big\langle -i (\alpha \cdot \nu) (\gamma_D^+ f_{\infty,+} - \gamma_D^- f_{\infty,-}), \widetilde{\mathcal{C}}_0 \zeta \big\rangle_{H^{1/2}_\alpha(\Sigma; \mathbb{C}^N) \times H^{-1/2}_\alpha(\Sigma; \mathbb{C}^N)} \\
      &\qquad \qquad - \big\langle \widetilde{\mathcal{C}}_0 \xi, -i (\alpha \cdot \nu) (\gamma_D^+g_{\infty,+} - \gamma_D^- g_{\infty,-}) \big\rangle_{H^{-1/2}_\alpha(\Sigma; \mathbb{C}^N) \times H^{1/2}_\alpha(\Sigma; \mathbb{C}^N)} \\
      &= \big\langle -i V (\alpha \cdot \nu) (\gamma_D^+ f_{\infty,+} - \gamma_D^- f_{\infty,-}), V \widetilde{\mathcal{C}}_0 \zeta \big\rangle_{H^{1/2}_{(\alpha\cdot \nu)V^*}(\Sigma; \mathbb{C}^N) \times H^{-1/2}_{(\alpha\cdot \nu)V^*}(\Sigma; \mathbb{C}^N)} \\
      &\quad - \big\langle V\widetilde{\mathcal{C}}_0 \xi, -i V(\alpha \cdot \nu) (\gamma_D^+g_{\infty,+} - \gamma_D^- g_{\infty,-}) \big\rangle_{H^{-1/2}_{(\alpha\cdot \nu)V^*}(\Sigma; \mathbb{C}^N) \times H^{1/2}_{(\alpha\cdot \nu)V^*}(\Sigma; \mathbb{C}^N)} \\
      &= (\Upsilon_1 f, \Upsilon_0 g)_{L^2(\Sigma; \mathbb{C}^N)} - (\Upsilon_0 f, \Upsilon_1 g)_{L^2(\Sigma; \mathbb{C}^N)},
    \end{split}
  \end{equation}
  which is the abstract Green's identity.
  
  Next, it follows from Green's identity~\eqref{Green_OBT_Dirac} that $T^{(0)} \upharpoonright \ker \Upsilon_0$ is symmetric. Moreover, one directly sees that $A_\infty \subset T^{(0)} \upharpoonright \ker \Upsilon_0$. Since $A_\infty$ is self-adjoint by Lemma~\ref{lemma_A_infty}, we find $A_\infty = T^{(0)} \upharpoonright \ker \Upsilon_0$.
  
  Next, we show that $(\Upsilon_0, \Upsilon_1): \dom T^{(0)} \rightarrow L^2(\Sigma; \mathbb{C}^N) \times L^2(\Sigma; \mathbb{C}^N)$ is surjective. Denote by $E_\pm: H^{1/2}(\Sigma; \mathbb{C}^N) \rightarrow H^1(\Omega_\pm; \mathbb{C}^N)$ the bounded extension operator defined by $\gamma_D^\pm (E_\pm \varphi) = \varphi$. Consider for arbitrary $\xi, \zeta \in L^2(\Sigma; \mathbb{C}^N)$ the function
  \begin{equation*}
    \begin{split}
      f &= \frac{1}{2} \big((E_+ i (\alpha \cdot \nu)V^* \Lambda^{-1} \zeta) \oplus (-E_- i (\alpha \cdot \nu) V^* \Lambda^{-1} \zeta)\big) -4 \widetilde{\Phi}_0  (\alpha \cdot \nu) \widetilde{\mathcal{C}}_0 (\alpha \cdot \nu) V^* \Lambda \xi \\
      &\in \dom A_\infty + \ker T^{(0)} = \dom T^{(0)}.
    \end{split}
  \end{equation*}
  Then, with~\eqref{equation_Ups0} and~\eqref{equation_Ups1} it is not difficult to see that $\Upsilon_0 f = \xi$ and $\Upsilon_1 f = \zeta$. Hence, $(\Upsilon_0, \Upsilon_1)$ is surjective, which finishes the proof that $\{ L^2(\Sigma; \mathbb{C}^N), \Upsilon_0, \Upsilon_1 \}$ is an ordinary boundary triple for $T^{(0)}$.
  
  It remains to show the claimed formulas for the $\gamma$-field and the Weyl function. First, define for $z \in \rho(A_\infty)$ and $\varphi \in L^2(\Sigma; \mathbb{C}^N)$ the function
  \begin{equation*}
    f_z := -4 \widetilde{\Phi}_z (\alpha \cdot \nu) \widetilde{\mathcal{C}}_z (\alpha \cdot \nu) V^* \Lambda \varphi.
  \end{equation*}
  Then,  by Proposition~\ref{proposition_extensions_Phi_C}~(i) one has that $f_z \in \ker (T^{(0)} - z)$. Moreover, Proposition~\ref{proposition_extensions_Phi_C}~(iii) and~\eqref{formula_C_inv_extension} yield that
  \begin{equation*}
    \Upsilon_0 f_z = -4 \Lambda^{-1} V \widetilde{\mathcal{C}}_z (\alpha \cdot \nu) \widetilde{\mathcal{C}}_z (\alpha \cdot \nu) V^* \Lambda \varphi = \varphi,
  \end{equation*}
  which shows the claimed formula for the $\gamma$-field $\beta(z)$. Furthermore, with the help of Proposition~\ref{proposition_extensions_Phi_C}~(iii) and~\eqref{formula_C_inv_extension} the Weyl function can be computed and we get for $\varphi \in L^2(\Sigma; \mathbb{C}^N)$
  \begin{equation*}
    \begin{split}
      \mathcal{M}(z) \varphi &= \Upsilon_1 \beta(z) \varphi \\
      &= -\Lambda \big( i V (\alpha \cdot \nu) (\gamma_D^+ (\beta(z) \varphi)_+ - \gamma_D^- (\beta(z) \varphi)_-) \\
      &\qquad \qquad + 2 V (\alpha \cdot \nu) \widetilde{\mathcal{C}}_0 (\alpha \cdot \nu) (\gamma_D^+ (\beta(z) \varphi)_+ + \gamma_D^- (\beta(z) \varphi)_-) \big) \\      
      &= -\Lambda \big( -4 i V (\alpha \cdot \nu) (-i (\alpha \cdot \nu) (\alpha \cdot \nu) \widetilde{\mathcal{C}}_z (\alpha \cdot \nu) V^* \Lambda \varphi) \\
      &\qquad \qquad -16 V (\alpha \cdot \nu) \widetilde{\mathcal{C}}_0 (\alpha \cdot \nu) (\widetilde{\mathcal{C}}_z (\alpha\cdot \nu))^2  V^* \Lambda \varphi \big) \\  
      &= 4 \Lambda V (\alpha \cdot \nu) ( \widetilde{\mathcal{C}}_z - \widetilde{\mathcal{C}}_0 ) (\alpha \cdot \nu) V^* \Lambda \varphi.
    \end{split}
  \end{equation*}
  Thus, all claims are shown.
\end{proof}

We remark that the boundary triple from the previous theorem can be viewed as a regularized boundary triple of a quasi boundary triple in the sense of \cite{BM14}. Indeed, it is easy to see that $\{ L^2(\Sigma; \mathbb{C}^N), V \Gamma_1^{(1)}, -V \Gamma_0^{(1)} \}$ is a quasi boundary triple that satisfies the assumptions in \cite[Theorem~2.12]{BM14} and the resulting ordinary boundary triple is exactly $\{ L^2(\Sigma; \mathbb{C}^N), \Upsilon_0, \Upsilon_1 \}$ in Theorem~\ref{theorem_OBT}; we refer to \cite{BH20}, where a similar procedure was applied in the context of Dirac operators with singular interactions.

\begin{remark} \label{remark_m_0}
  If $m=0$, then one can show that for any fixed $\zeta \in \mathbb{C} \setminus \mathbb{R}$ the maps $\widehat{\Upsilon}_0, \widehat{\Upsilon}_1: \dom T^{(0)} \rightarrow L^2(\Sigma; \mathbb{C}^N)$ defined by
  \begin{equation*} 
    \begin{split}
      \widehat{\Upsilon}_0 f &:= \frac{1}{2} \Lambda^{-1} V (\gamma_D^+ f_+ + \gamma_D^- f_-),\\
      \widehat{\Upsilon}_1 f &:= -\Lambda \big( i V (\alpha \cdot \nu) (\gamma_D^+ f_+ - \gamma_D^- f_-) + V (\alpha \cdot \nu) \big(\widetilde{\mathcal{C}}_{\zeta} + \widetilde{\mathcal{C}}_{\overline{\zeta}}\big) (\alpha \cdot \nu) (\gamma_D^+ f_+ + \gamma_D^- f_-) \big),
    \end{split}
  \end{equation*}
  constitute an ordinary boundary triple $\{ L^2(\Sigma; \mathbb{C}^N), \widehat{\Upsilon}_0, \widehat{\Upsilon}_1 \}$ for 
  $T^{(0)}$ such that  $T^{(0)}\upharpoonright \ker \Upsilon_0 = A_\infty$. Note that $\rho(A_\infty)=\mathbb{C} \setminus \mathbb R$ holds in this case. The corresponding $\gamma$-field $\widehat{\beta}$ 
  and Weyl function $\widehat{\mathcal{M}}$ are given by
  \begin{equation*}
    \rho(A_\infty) \ni z \mapsto \widehat{\beta}(z) = -4 \widetilde{\Phi}_z (\alpha \cdot \nu) \widetilde{\mathcal{C}}_z (\alpha \cdot \nu) V^* \Lambda  
  \end{equation*}
  and
  \begin{equation*}
    \rho(A_\infty) \ni z \mapsto \widehat{\mathcal{M}}(z) = 4 \Lambda V (\alpha \cdot \nu) \left( \widetilde{\mathcal{C}}_z - \frac{1}{2} \big(\widetilde{\mathcal{C}}_\zeta+ \widetilde{\mathcal{C}}_{\overline{\zeta}} \big) \right) (\alpha \cdot \nu) V^* \Lambda.
  \end{equation*}
  To see the above claims, one can apply \cite[Proposition~2.11]{BHOP20} with $B = A_\infty$ and $\mathcal{T} f = -i \Lambda V (\alpha \cdot \nu) (\gamma_D^+ f_+ - \gamma_D^- f_-)$, $f \in \dom A_\infty$. Using that $A_\infty = T^{(1)} \upharpoonright \ker \Gamma_1^{(1)}$ with $\Gamma_1^{(1)}$ defined by~\eqref{def_Gamma} and the associated Krein-type resolvent formula from Theorem~\ref{theorem_Krein_abstract}, one can employ then a similar construction as in \cite[Proposition~3.5]{BHOP20}.
\end{remark}

\section{Dirac operators with singular interactions supported on Lipschitz smooth curves and surfaces}

In this section we study Dirac operators with singular interactions supported on boundaries of compact Lipschitz domains in $\mathbb{R}^2$ and $\mathbb{R}^3$. In Section~\ref{section_Lipschitz} we employ the generalized boundary triple $\big\{L^2(\Sigma;\dC^N),\Gamma_0^{(1/2)},\Gamma_1^{(1/2)} \big\}$ from Section~\ref{section_QBT} to derive sufficient conditions for the self-adjointness of these operators and to recover their spectral properties; cf. \cite{AMV14, AMV15, BEHL18, BEHL19, BHSS21, Ben21} for similar considerations in dimension $q=3$ under different assumptions on the interaction support $\Sigma$. We also mention that in \cite{Ben22}  various rough domains and interaction strengths are considered.

In Section~\ref{section_confinement} we investigate Dirac operators with singular interactions in  the confinement case, i.e. for certain combinations of interaction strengths which lead to a decoupling of these operators into two operators which act independently in $L^2(\Omega_\pm;\dC^N)$. 

Similarly as in Section \ref{section_3d}, we assume that $\Omega_+ \subset \dR^q$, $q \in \{2,3\}$, is a bounded Lipschitz domain with boundary $\Sigma$ and  $\Omega_- = \dR^q \setminus \overline{\Omega_+}$. Moreover, $\nu$ denotes the unit normal vector field on $\Sigma$ that is pointing outwards of $\Omega_+$ and $\alpha_j \in \dC^{N\times N}$, $j \in \{0,\dots,q\}$, are the Dirac matrices defined in~\eqref{Dirac_matrices1}--\eqref{Dirac_matrices2}.

For $\eta, \tau, \lambda \in \mathbb{R}$ we set 
\begin{equation} \label{def_P_eta_tau_lambda}
  P_{\eta, \tau, \lambda} := \eta I_N + \tau \alpha_0 + \lambda i (\alpha \cdot \nu) \alpha_0
\end{equation}
and define, in a similar way as in~\eqref{Diracperturbed2}, in $L^2(\mathbb{R}^q; \mathbb{C}^N)$ the symmetric operator $A_{\eta,\tau,\lambda}$ by
\begin{equation}\label{def_A_eta,tau,mu}
A_{\eta,\tau,\lambda} := T^{(1/2)} \upharpoonright \ker \big(\Gamma_0^{(1/2)} + P_{\eta,\tau,\lambda}\Gamma_1^{(1/2)} \big).
\end{equation}
The operator $A_{\eta,\tau,\lambda}$ is given more explicitly by
\begin{equation*}
\begin{split}
A_{\eta,\tau,\lambda} f &= \big(-i (\alpha \cdot \nabla) + m \alpha_0\big) f_+ \oplus \big(-i (\alpha \cdot \nabla) + m \alpha_0\big) f_-, \\
\dom A_{\eta,\tau,\lambda} &= \left\{ f \in H^{1/2}_\alpha(\mathbb{R}^q \setminus \Sigma; \mathbb{C}^N): i (\alpha \cdot \nu) (\gamma_D^+ f_+ - \gamma_D^- f_-) \right.\\
&\hspace{130 pt }\left.+ P_{\eta,\tau,\lambda}\frac{1}{2} (\gamma_D^+ f_+ + \gamma_D^- f_-) = 0 \right\}.
\end{split}
\end{equation*}
The symmetry of $A_{\eta, \tau, \lambda}$ follows in a similar way as in \eqref{abab} from the symmetry of $P_{\eta, \tau, \lambda}$ and the abstract Green's identity in Definition \ref{qbtdef} (i).

\subsection{Spectral analysis of Dirac operators with singular interactions on boundaries of compact Lipschitz domains} 
\label{section_Lipschitz}

This subsection is devoted to the analysis of the self-adjointness and spectral properties of the symmetric operator $A_{\eta,\tau,\lambda}$.
For this purpose we investigate 
$ I_N + P_{\eta,\tau,\lambda}M^{(1/2)}(z) = I_N + P_{\eta,\tau,\lambda}\mathcal{C}_z$; cf. Theorems \ref{theorem_Krein_abstract_A_B} and~\ref{theorem_QBT}. Thus, we start by proving preliminary results on $I_N + P_{\eta,\tau,\lambda}\mathcal{C}_z$, the main results on $A_{\eta, \tau, \lambda}$ are then proved in Theorem \ref{selfadjoint_theorem}. 

Recall that $A_0$ is the free Dirac operator defined in~\eqref{def_A_0} and that $\mathcal{C}_z$, $z \in \rho(A_0)$, is given by~\eqref{def_C_lambda}. We introduce the operator $\mathcal{D}_\alpha: L^2(\Sigma;\dC^N) \to L^2(\Sigma;\dC^N)$ acting as
\begin{equation} \label{def_D_alpha}
	\begin{split}
	\mathcal{D}_\alpha \varphi(x)&:= \lim_{\varepsilon \searrow 0} \int_{\Sigma \setminus B(x,\varepsilon)} \frac{i\left(\alpha \cdot (x-y)\right)}{2^{q-1}\pi |x-y|^{q}}f(y) \, d\sigma(y), \quad \varphi \in L^2(\Sigma; \mathbb{C}^N),~x \in \Sigma.
	\end{split}
\end{equation}
Then, the following holds:

\begin{lemma}\label{lemma_K_alpha}
	The operator $\mathcal{D}_\alpha$ defined by~\eqref{def_D_alpha} 	
	is well defined and bounded, and for all $z \in \rho(A_0)$ the following holds:
	\begin{itemize}
		\item[(i)] $\mathcal{D}_\alpha - \mathcal{C}_z$ is compact in $L^2(\Sigma;\dC^N)$.
		\item[(ii)] $\alpha_0 \mathcal{D}_\alpha + \mathcal{D}_\alpha\alpha_0 = 0$. In particular $\alpha_0 \mathcal{C}_z + \mathcal{C}_z \alpha_0$ is compact in $L^2(\Sigma;\dC^N)$.
		\item[(iii)]$ \mathcal{D}_\alpha$ is self-adjoint.
		\item[(iv)] $\sigma_{\mathrm{ess}}(|\mathcal{D}_\alpha|) \subset  \left[\frac{1}{4\|\mathcal{C}_z\|_{L^2(\Sigma;\dC^N)\to L^2(\Sigma;\dC^N)}},\|\mathcal{C}_z\|_{L^2(\Sigma;\dC^N)\to L^2(\Sigma;\dC^N)} \right]$.
	\end{itemize}
\end{lemma}
\begin{proof}
	First, we prove that $R_z := \mathcal{D}_\alpha - \mathcal{C}_z$ is a compact operator in $L^2(\Sigma;\dC^N)$. We prove this assertion for $q=3$, the proof for $q=2$ follows the same lines using the representation of $G_{z,2}$ in \cite[Lemma~3.2]{BHOP20}.
	
	By \eqref{def_G_lambda} and~\eqref{def_D_alpha} the integral kernel of $R_z$ is given for $x,y \in \Sigma$, $x \neq y$, by
	\begin{equation*}
	\begin{split}
	r_z(x-y) =&  \frac{i \left(\alpha \cdot (x-y)\right)}{4 \pi  | x-y |^3} \left(1-e^{i k(z) |x-y|}\right) \\
	&-\left( z I_4 + m \alpha_0 + k(z)  \frac{ \left(\alpha \cdot (x-y)\right)}{  | x-y |}  \right) \frac{1}{4 \pi | x-y |} e^{i k(z) |x-y|}.
	\end{split}    
	\end{equation*}
	It is not difficult to see that $|r_z(x-y)| = \mathcal{O}(|x-y|^{-1})$. Hence, one can prove with the same arguments as in \cite[(3.11)~Proposition]{F95} that $R_z$ is a compact operator in $L^2(\Sigma; \mathbb{C}^N)$. Since $\mathcal{C}_z$ is bounded in $L^2(\Sigma; \mathbb{C}^N)$, see \eqref{def_C_lambda} and the text below, we obtain that $\mathcal{D}_\alpha$ is well defined and bounded and that assertion~(i) is true.
	
	Assertion~(ii) is a direct consequence of the anti-commutation relations of the Dirac matrices in \eqref{anti_commutation} and item~(i). Next, let us show the self-adjointness of $\mathcal{D}_\alpha$. Since the relation $(r_z(x-y))^* = r_{\overline{z}}(y-x)$ holds for $x \neq y \in \Sigma$, one has $(R_z)^* = R_{\overline{z}}$. Hence, by \eqref{equation_diff_m} and Theorem~\ref{theorem_QBT}~(iii) we get
	\[(\mathcal{D}_\alpha)^* = \left(\mathcal{C}_{z} + R_{{z}}\right)^* = \mathcal{C}_{\overline{z}} + R_{{\overline{z}}} = \mathcal{D}_\alpha.\] 
	
	Finally, we prove assertion (iv). Since $R_z$ and $R_{\overline{z}}$ are compact in $L^2(\Sigma;\dC^N)$, the inclusion  
	\begin{equation*}
	\begin{split}
	\sigma_{\mathrm{ess}}(\mathcal{D}_\alpha^2) &= \sigma_{\mathrm{ess}}(\mathcal{C}_z\mathcal{C}_{\overline{z}})  = \sigma_{\mathrm{ess}}(\mathcal{C}_z\mathcal{C}_{z}^*) \\
	&\subset  \sigma(\mathcal{C}_z\mathcal{C}_{z}^*) \subset  \left[\|\mathcal{C}_z^{-1}\|_{L^2(\Sigma;\dC^N) \rightarrow L^2(\Sigma;\dC^N)}^{-2},\|\mathcal{C}_z\|_{L^2(\Sigma;\dC^N) \rightarrow L^2(\Sigma;\dC^N)}^2 \right]
	\end{split}
	\end{equation*}
	holds. By~\eqref{formula_C_inv} we have
	\begin{equation*}
	\begin{split}
	  \|\mathcal{C}_z^{-1}\|_{L^2(\Sigma;\dC^N)\rightarrow L^2(\Sigma;\dC^N)} &= 4\|(\alpha \cdot \nu)\mathcal{C}_z (\alpha \cdot \nu) \|_{L^2(\Sigma;\dC^N)\rightarrow L^2(\Sigma;\dC^N)} \\
	  &= 4\|\mathcal{C}_z  \|_{L^2(\Sigma;\dC^N)\rightarrow L^2(\Sigma;\dC^N)}.
    \end{split}
    \end{equation*}
	Taking the definition of the modulus of a self-adjoint operator into account, one concludes the claim of item~(iv) from the last two displayed formulas.
\end{proof}

In the next lemma we show another useful connection of the operators $\mathcal{C}_z$ and $\mathcal{D}_\alpha$. In the following we denote by $\chi_\cS$ the characteristic function of a measurable set $\cS \subset \dR$.

\begin{lemma}\label{selfadjoint_lemma_2}
	Let $\omega_{\max} := \max \sigma_{\mathrm{ess}}(|\mathcal{D}_\alpha|) $, $\omega_{\min} := \min \sigma_{\mathrm{ess}}(|\mathcal{D}_\alpha|)$, and define the function $f_\alpha: \dR \to \dR$ by
	\begin{equation*}
	\begin{split}
	f_\alpha(\omega) &:= \omega_\min \chi_{[0,\omega_\min)}(\omega)  + \omega\chi_{[\omega_\min,\omega_\max]}(\omega) + \omega_\max \chi_{(\omega_\max,\infty)}(\omega).
	\end{split}
	\end{equation*} 
	Then, for any $z \in \rho(A_0)$ there exists a self-adjoint unitary operator $U$ which is independent of $z$ and commutes with $f_\alpha(|\mathcal{D}_\alpha|)$  and a compact operator $\mathcal{K}_z$ such that $\mathcal{C}_z =  Uf_\alpha(|\mathcal{D}_\alpha|)+ \mathcal{K}_z$.
\end{lemma}
\begin{remark} \label{remark_K_alpha}
	The definition of $f_\alpha(|\mathcal{D}_\alpha|)$ and assertion~(iv) of Lemma \ref{lemma_K_alpha} imply $\omega_\min > 0$ and for all $ g \in L^2(\Sigma;\dC^N)$ that
	\[ \omega_\min \|g\|_{L^2(\Sigma;\dC^N)}^2 \leq  \left(f_\alpha(|\mathcal{D}_\alpha|)g,g\right)_{L^2(\Sigma;\dC^N)} \leq \omega_\max \|g\|_{L^2(\Sigma;\dC^N)}^2  \] 
	holds. In particular, $\sigma (f_\alpha(|\mathcal{D}_\alpha|)) \subset [\omega_\min, \omega_\max]$.
\end{remark}
\begin{proof}[Proof of Lemma~\ref{selfadjoint_lemma_2}]
  Define the functions $f_{\alpha,1}(\omega) := (\omega-\omega_\min) \chi_{[0,\omega_\min)}(\omega)$ and $f_{\alpha,2}(\omega) := (\omega-\omega_\max) \chi_{(\omega_\max,\infty)}(\omega)$. 
	Then, for any $\omega \geq 0$ the equality
	\begin{equation*}
	  \omega = f_\alpha(\omega) + f_{\alpha,1}(\omega) + f_{\alpha,2}(\omega)
	\end{equation*}
	holds, which implies
	\begin{equation*}
	|\mathcal{D}_\alpha| = f_\alpha(|\mathcal{D}_\alpha|) +f_{\alpha,1}(|\mathcal{D}_\alpha|) +f_{\alpha,2}(|\mathcal{D}_\alpha|).
	\end{equation*}
	Next, we show that $f_{\alpha,1}(|\mathcal{D}_\alpha|)$ is compact. We denote the elements of the set $\sigma_{\mathrm{disc}}(|\mathcal{D}_{\alpha}|)$ below $\omega_{\min}$ by $0\leq \omega_1<\omega_2<\dots$ . This sequence is  either finite or it converges to $\omega_{\min}$; if it is finite, we extend it by $\omega_{\min}$ to an infinite sequence. Define for $n \in \dN$ the functions
	\begin{equation*}
	f_{\alpha,1}^{(n)}(\omega) := (\omega- \omega_{\min}) \chi_{[0,\omega_n)}(\omega).
	\end{equation*}
	 Then, the operators $f_{\alpha,1}^{(n)}(|\mathcal{D}_\alpha|)$ are finite rank operators. Moreover, one has
	\begin{equation*}
	\|f_{\alpha,1}(|\mathcal{D}_\alpha|)-f_{\alpha,1}^{(n)}(|\mathcal{D}_\alpha|)\|_{L^2(\Sigma;\dC^N) \to L^2(\Sigma;\dC^N)} \leq \omega_{\min} - \omega_{n}.
	\end{equation*}
	Hence, $f_{\alpha,1}^{(n)}(|\mathcal{D}_\alpha|)$ converges to $f_{\alpha,1}(|\mathcal{D}_\alpha|)$ with respect to the operator norm and  therefore $f_{\alpha,1}(|\mathcal{D}_\alpha|)$ is compact. In the same way one gets that $f_{\alpha,2}(|\mathcal{D}_\alpha|)$ is compact. Furthermore, if we set $U := \sgn(\mathcal{D}_\alpha)$, with the convention $\sgn(0)=1$, we obtain
	\begin{equation*}
	\mathcal{D}_\alpha = U |\mathcal{D}_\alpha| = Uf_\alpha(|\mathcal{D}_\alpha|) + U(f_{\alpha,1}(|\mathcal{D}_\alpha|) + f_{\alpha,2}(|\mathcal{D}_\alpha|)). 
	\end{equation*}
	Taking the considerations from above into account, we see with Lemma \ref{lemma_K_alpha}~(i) that $\mathcal{K}_z := U(f_{\alpha,1}(|\mathcal{D}_\alpha|) + f_{\alpha,2}(|\mathcal{D}_\alpha|)) - 	\mathcal{D}_\alpha + 	\mathcal{C}_z$ is compact, which implies then the claimed result.
\end{proof}

The next lemma is devoted to the injectivity of the maps $I_N+P_{\eta,\tau,\lambda}\mathcal{C}_z $ and $I_N+\mathcal{C}_z P_{\eta,\tau,\lambda}$, where $P_{\eta,\tau,\lambda}$ is the matrix-valued function defined in~\eqref{def_P_eta_tau_lambda}:

\begin{lemma}\label{selfadjoint_lemma_1}
	Let $\eta,\tau,\lambda \in \dR$ and $z \in \dC\setminus\dR$. Then, $I_N+P_{\eta,\tau,\lambda}\mathcal{C}_z $ as well as $I_N+\mathcal{C}_z P_{\eta,\tau,\lambda}$ is injective in $L^2(\Sigma;\dC^N)$.
\end{lemma}
\begin{proof}
	Throughout the proof, let $z \in \mathbb{C} \setminus \mathbb{R}$. First, we show that $I_N+P_{\eta,\tau,\lambda}\mathcal{C}_z$ is injective. Recall that $\mathcal{C}_z = M^{(1/2)}(z)$, where $M^{(1/2)}$ is the Weyl function associated with the generalized boundary triple $\{ L^2(\Sigma; \mathbb{C}^N), \Gamma_0^{(1/2)}, \Gamma_1^{(1/2)} \}$; cf. Theorem~\ref{theorem_QBT}. Since the operator $A_{\eta,\tau,\lambda}$ in~\eqref{def_A_eta,tau,mu} is symmetric, we have $z \notin \sigma_{\mathrm{p}}(A_{\eta,\tau,\lambda})$. By Theorem~\ref{theorem_Krein_abstract_A_B}~(i) this implies $0 \notin \sigma_{\mathrm{p}}(I_N + P_{\eta,\tau,\lambda}M^{(1/2)}(z)) = \sigma_{\mathrm{p}}(I_N + P_{\eta,\tau,\lambda}\mathcal{C}_z)$, i.e. $I_N+P_{\eta,\tau,\lambda}\mathcal{C}_z$ is injective.

	Next, to show the injectivity of $I_N+\mathcal{C}_z P_{\eta,\tau,\lambda}$ we use the formula for the inverse of $\mathcal{C}_z$ in \eqref{formula_C_inv} and rewrite $I_N+\mathcal{C}_z P_{\eta,\tau,\lambda}$ in the  form
	\begin{equation*}	
	\begin{split}
	I_N + \mathcal{C}_z P_{\eta,\tau,\lambda}  &= \mathcal{C}_z\left( -4(\alpha \cdot \nu)\mathcal{C}_z (\alpha \cdot \nu) + P_{\eta,\tau,\lambda} \right) \\
	&= -4 \mathcal{C}_z (\alpha \cdot \nu) ( \mathcal{C}_z - P_{\frac{\eta}{4},-\frac{\tau}{4},-\frac{\lambda}{4}}) (\alpha \cdot \nu).
	\end{split}
	\end{equation*}
	Since $T^{(1/2)}\upharpoonright \ker (\Gamma_1^{(1/2)} - P_{\frac{\eta}{4},-\frac{\tau}{4},-\frac{\lambda}{4}}\Gamma_0^{(1/2)})$ is symmetric by~\eqref{abab}, we conclude from Theorem~\ref{theorem_Krein_abstract}~(i) that $\mathcal{C}_z - P_{\frac{\eta}{4},-\frac{\tau}{4},-\frac{\lambda}{4}}$ is injective. Since $\mathcal{C}_z$ and $\alpha \cdot \nu$ are bijective, this and the last displayed formula imply that $I_N+\mathcal{C}_z P_{\eta,\tau,\lambda}$ is injective as well.
\end{proof}

In the following proposition we state conditions under which $I_N+P_{\eta,\tau,\lambda}\mathcal{C}_z $, $z \in \mathbb{C} \setminus \mathbb{R}$, is surjective as well. Recall that $\omega_{\min} := \min \sigma_{\mathrm{ess}}(|\mathcal{D}_\alpha|)$ and $\omega_{\max} := \max \sigma_{\mathrm{ess}}(|\mathcal{D}_\alpha|) $.

\begin{proposition}\label{selfadjoint_prop_1}
	Let $\eta,\tau,\lambda \in \dR$ and let $d =\eta^2 -\tau^2 -\lambda^2$ such that one of the following conditions holds:
	\begin{itemize}
	\item[(i)] $1- d \omega^2 \neq 0$ for all $\omega \in [\omega_{\min},\omega_{\max}]$ and $\max_{\omega \in [\omega_{\min},\omega_{\max}]}\big|\frac{2\lambda \omega_{\max}}{{1- d \omega^2}}\big| <1$.
	\item[(ii)] $\lambda \neq 0$ and $\max_{\omega \in [\omega_{\min},\omega_{\max}]}\big|\frac{1- d \omega^2}{\lambda \omega_{\min}(1+4\omega_{\min}^2)}\big| <1$.
	\item[(iii)] $\lambda^2 \neq 4$ and $\frac{4|d+4|(1+\omega_\max|\lambda|)\omega_\max^2}{|4-\lambda^2|(1+4\omega_{\min}^2)} <1 $. 
	\item[(iv)] $\eta = \tau = 0$ and $\lambda^2 \neq 4$.
	\item[(v)] $|\eta| + |\tau| + |\lambda| < \frac{1}{\omega_{\max}}$.
	\item[(vi)] $|\eta| + |\tau| + |\lambda| \neq 0$ and $\frac{|d|}{|\eta| + |\tau| + |\lambda|} > 4 \omega_{\max}$.
	\end{itemize}
	Then, for any $z \in \dC \setminus \dR$ the operator $I_N + P_{\eta,\tau,\lambda}\mathcal{C}_z$ is bijective in $L^2(\Sigma;\dC^N)$.
\end{proposition}

We remark that in items (i) and (ii) in the previous proposition the interval $[\omega_{\min},\omega_{\max}]$ can be replaced by the (in general smaller) set $\sigma_{\textup{ess}}(|\mathcal{D}_\alpha|)$, cf. \eqref{equation_norms}, \eqref{equation_norms2}, and Lemma~\ref{selfadjoint_lemma_2}.

\begin{proof}[Proof of Proposition~\ref{selfadjoint_prop_1}]
	By Lemma \ref{selfadjoint_lemma_1} the operator $I_N + P_{\eta,\tau,\lambda}\mathcal{C}_z$  is injective. To show its surjectivity we study  in all cases a product of the form
	\begin{equation}\label{def_F}
	F:=L (I_N + P_{\eta,\tau,\lambda} \mathcal{C}_z)
	\end{equation}
	with a suitable injective and bounded operator $L$ in $L^2(\Sigma; \mathbb{C}^N)$. We show that $F$ is a Fredholm operator with index zero, which shows that $F$ is bijective. This implies that $L$ is surjective as well, which yields its bijectivity. Therefore, also 
	\begin{equation*}
	  I_N + P_{\eta,\tau,\lambda} \mathcal{C}_z = L^{-1} F
	\end{equation*}
	is surjective, which implies the claim. Hence, it suffices to show in all cases that for a suitable injective and bounded operator $L$ the map $F$ is a Fredholm operator with index zero.

	(i) Set $L = (I_N - \mathcal{C}_z P_{\eta,-\tau,-\lambda})$ and note, that $L$ is injective by Lemma~\ref{selfadjoint_lemma_1}. Using that~\eqref{anti_commutation} implies that $P_{\eta, -\tau, -\lambda} P_{\eta, \tau, \lambda} = d I_N$, we find that
	\begin{equation} \label{equation_F_i1}
	\begin{split}
	F&= I_N - d\mathcal{C}_z^2 +i \lambda (\mathcal{C}_z (\alpha \cdot \nu) \alpha_0 +(\alpha \cdot \nu) \alpha_0 \mathcal{C}_z  ) + \tau (\alpha_0 \mathcal{C}_z + \mathcal{C}_z \alpha_0) \\
	&= I_N - d\mathcal{C}_z^2 +i \lambda \alpha_0 (\mathcal{C}_z(\alpha \cdot \nu) -(\alpha \cdot \nu)\mathcal{C}_z  )+ K_1, 
	\end{split}
	\end{equation}
	where $K_1 :=( \alpha_0 \mathcal{C}_z +   \mathcal{C}_z\alpha_0)(\tau I_N - i \lambda (\alpha \cdot \nu))$ is compact by Lemma~\ref{lemma_K_alpha}~(ii). Using that $\mathcal{C}_z =  Uf_\alpha(|\mathcal{D}_\alpha|)+ \mathcal{K}_z$ holds by Lemma~\ref{selfadjoint_lemma_2}, we get 
	\begin{equation*}
	  \mathcal{C}_z^2 = f_\alpha(|\mathcal{D}_\alpha|)^2 + K_2
	\end{equation*}
	with the compact operator $K_2 := U f_\alpha (|\mathcal{D}_\alpha|) \mathcal{K}_z + \mathcal{K}_z U f_\alpha (|\mathcal{D}_\alpha|) + \mathcal{K}_z^2$ and
	\begin{equation*}
	  \alpha_0 (\mathcal{C}_z(\alpha \cdot \nu) -(\alpha \cdot \nu)\mathcal{C}_z  ) = \alpha_0 ( Uf_\alpha(|\mathcal{D}_\alpha|)(\alpha \cdot \nu) -(\alpha \cdot \nu)Uf_\alpha(|\mathcal{D}_\alpha|)  ) + K_3
	\end{equation*}
	with the compact operator $K_3 := \alpha_0 (\mathcal{K}_z(\alpha \cdot \nu) -(\alpha \cdot \nu)\mathcal{K}_z)$. Define the bounded operators $A_1:= I_N- df_\alpha(|\mathcal{D}_\alpha|)^2$ and $B_1:=i \lambda \alpha_0 ( Uf_\alpha(|\mathcal{D}_\alpha|)(\alpha \cdot \nu) -(\alpha \cdot \nu)Uf_\alpha(|\mathcal{D}_\alpha|)  )$ and the compact operator $K_4 := K_1 - d K_2 + i \lambda K_3$. Then,~\eqref{equation_F_i1} simplifies to
	\begin{equation} \label{equation_F_i}
	\begin{split}
	F&=I_N- df_\alpha(|\mathcal{D}_\alpha|)^2+ i \lambda \alpha_0 ( Uf_\alpha(|\mathcal{D}_\alpha|)(\alpha \cdot \nu) -(\alpha \cdot \nu)Uf_\alpha(|\mathcal{D}_\alpha|)  )+ K_4\\
	&= A_1 + B_1 + K_4.
	\end{split}
	\end{equation}
	It follows from Lemma~\ref{selfadjoint_lemma_2}, Remark~\ref{remark_K_alpha}, and the assumption in this point  that $A_1$ is invertible and
	\begin{equation} \label{equation_norms}
	\begin{split}
	\|A_1^{-1}\|_{L^2(\Sigma;\dC^N)\to L^2(\Sigma;\dC^N)}  &\leq \max_{\omega \in [\omega_{\min},\omega_{\max}]}\frac{1}{{|1- d \omega^2|}}, \\
	\|B_1\|_{L^2(\Sigma;\dC^N)\to L^2(\Sigma;\dC^N)} &\leq 2|\lambda| \omega_{\max}.
	\end{split}
	\end{equation}
	Hence, the assumption in~(i) implies that $\|A_1^{-1}B_1\|_{L^2(\Sigma;\dC^N)\to L^2(\Sigma;\dC^N)} <1$ and thus, $A_1+B_1$ is bijective in $L^2(\Sigma;\dC^N)$, which yields that $F$ is a Fredholm operator with index zero.

	(ii) We follow a similar strategy as in the proof under assumption~(i), consider the injective operator $L = (I_N - \mathcal{C}_z P_{\eta,-\tau,-\lambda})$, and use that Lemmas~\ref{selfadjoint_lemma_2} and~\ref{lemma_K_alpha}~(ii) and formula~\eqref{C_inv} imply
	\begin{equation} \label{equation_F_ii}
	\begin{split}
	F&= I_N- d\mathcal{C}_z^2 +i \lambda \alpha_0  \mathcal{C}_z(\alpha \cdot \nu) (I_N +4\mathcal{C}_z^2  )+ K_1\\
	&= I_N- df_\alpha(|\mathcal{D}_\alpha|)^2 +i \lambda \alpha_0 Uf_\alpha(|\mathcal{D}_\alpha|)(\alpha \cdot \nu) (I_N +4f_\alpha(|\mathcal{D}_\alpha|)^2  )+ K_5\\
	&= A_1 + B_2 + K_5
	\end{split}
	\end{equation}
	with  $B_2:=i \lambda \alpha_0 Uf_\alpha(|\mathcal{D}_\alpha|)(\alpha \cdot \nu) (I_N +4f_\alpha(|\mathcal{D}_\alpha|)^2  )$, $A_1$ and $K_1$ are defined as in the proof under assumption~(i), and $K_5 = K_1 + K_6$ with $K_6$ being a compact operator resulting from a product of $\mathcal{K}_z$ from Lemma~\ref{selfadjoint_lemma_2} and bounded operators. 
	Note that $B_2$ is invertible by Remark~\ref{remark_K_alpha} with $B_2^{-1} = -i \lambda^{-1}  (I_N +4f_\alpha(|\mathcal{D}_\alpha|)^2  )^{-1} (\alpha \cdot \nu) f_\alpha(|\mathcal{D}_\alpha|)^{-1} U^* \alpha_0$.
	Hence, Lemma~\ref{selfadjoint_lemma_2} and Remark~\ref{remark_K_alpha} imply that
	\begin{equation} \label{equation_norms2}
	\begin{split}
	\|A_1\|_{L^2(\Sigma;\dC^N)\to L^2(\Sigma;\dC^N)}  &\leq \max_{\omega \in [\omega_{\min},\omega_{\max}]}{{|1- d \omega^2|}},\\
	\|B_2^{-1}\|_{L^2(\Sigma;\dC^N)\to L^2(\Sigma;\dC^N)} &\leq \frac{1}{|\lambda| \omega_{\min}(1+4 \omega_\min^2)}.
	\end{split}
	\end{equation}
	Thus, assumption~(ii) yields $\|B_2^{-1}A_1\|_{L^2(\Sigma;\dC^N)\to L^2(\Sigma;\dC^N)} <1$. Therefore, in this case $F$ is a Fredholm operator with index zero.
	
	
	(iii) We choose $L = (I_N -  \mathcal{C}_z P_{0,0,\lambda} )(I_N - \mathcal{C}_z P_{\eta,-\tau,-\lambda})$. Due to the Lemmas~\ref{selfadjoint_lemma_2} and~\ref{lemma_K_alpha}~(ii) and \eqref{C_inv} one finds with~\eqref{equation_F_ii} that the expression~\eqref{def_F} reads as
	\begin{equation*}
	\begin{split}
	F&=	 (I_N -  \mathcal{C}_z P_{0,0,\lambda} )\Big((I_N +  i\lambda \alpha_0 \mathcal{C}_z(\alpha\cdot \nu))(I_N+4\mathcal{C}_z^2) - (d+4)\mathcal{C}_z^2 +K_1\Big)\\
	&=\frac{4-\lambda^2}{4}(I_N+4\mathcal{C}_z^2)  -(d+4)(I_N -  \mathcal{C}_z P_{0,0,\lambda} )\mathcal{C}_z^2  + K_7\\
	&= A_2  - B_3 + K_{8}
	\end{split}
	\end{equation*}
	with $A_2 := \frac{4-\lambda^2}{4}(I_N+4f_\alpha(|\mathcal{D}_\alpha|)^2)$, $B_3:=(d+4)(I_N -  Uf_\alpha(|\mathcal{D}_\alpha|) P_{0,0,\lambda} )f_\alpha(|\mathcal{D}_\alpha|)^2 $ and the compact operator $K_1$ from above. Moreover, $K_7$ and $K_{8}$ are the compact maps given by  $K_7 := (I_N -  \mathcal{C}_z P_{0,0,\lambda} ) K_1  + i \lambda (\alpha_0 \mathcal{C}_z + \mathcal{C}_z\alpha_0)(\alpha \cdot \nu)(I_N+4\mathcal{C}_z^2)$  and $K_{8}:= K_7 + K_9$,  where $K_9$ is a compact operator resulting from a product of $\mathcal{K}_z$ in Lemma~\ref{selfadjoint_lemma_2} and bounded operators.
	It follows again from Lemma~\ref{selfadjoint_lemma_2} and Remark~\ref{remark_K_alpha} that
	\begin{equation} 
	\begin{split}
	\|A_2^{-1}\|_{L^2(\Sigma;\dC^N)\to L^2(\Sigma;\dC^N)}  &\leq \frac{4}{|4-\lambda^2|(1+4\omega_{\min}^2)}, \\
	\|B_3\|_{L^2(\Sigma;\dC^N)\to L^2(\Sigma;\dC^N)} &\leq |d+4|(1+\omega_\max|\lambda|)\omega_\max^2.
	\end{split}
	\end{equation}
	Therefore, the assumption (iii) implies $\|A_2^{-1}B_3\|_{L^2(\Sigma;\dC^N)\to L^2(\Sigma;\dC^N)}<1$. Hence, $F$ is a Fredholm operator with index zero.
	
	(iv) 	Choose $L=I_N- P_{0,0,\lambda}\mathcal{C}_z$ in~\eqref{def_F} and note that $L$ is injective by Lemma~\ref{selfadjoint_lemma_1}. Then
	\begin{equation*}
	\begin{split}
	F&=(I_N- P_{0,0,\lambda}\mathcal{C}_z )(I_N+ P_{0,0,\lambda}\mathcal{C}_z)  \\
	&= I_N + \lambda^2(\alpha \cdot \nu)\alpha_0\mathcal{C}_z(\alpha \cdot \nu)\alpha_0\mathcal{C}_z = \frac{4-\lambda^2}{4}I_N + K_{10}, 
	\end{split}
	\end{equation*}
	where  $K_{10} := \lambda^2(\alpha \cdot \nu)\alpha_0\mathcal{C}_z(\alpha \cdot \nu)(\alpha_0\mathcal{C}_z + \mathcal{C}_z \alpha_0)$ is a compact operator in $L^2(\Sigma;\mathbb{C}^N)$. Thus, $F$ is again a Fredholm operator with index zero.
	
	(v) Choose $L = I_N$ as the identity. Then, with Lemma~\ref{selfadjoint_lemma_2} we get 
	\begin{equation*}
	  F = I_N + P_{\eta, \tau, \lambda} \mathcal{C}_z = I_N + P_{\eta, \tau, \lambda} U f_\alpha (|\mathcal{D}_\alpha|) + K_{11},
	\end{equation*}
	where $K_{11} = P_{\eta, \tau, \lambda} \mathcal{K}_z$ is compact by Lemma~\ref{selfadjoint_lemma_2}. Since 
	\begin{equation*}
	  \| P_{\eta, \tau, \lambda} U f_\alpha (|\mathcal{D}_\alpha|) \|_{L^2(\Sigma; \mathbb{C}^N) \to L^2(\Sigma; \mathbb{C}^N)} \leq \big( |\eta| + |\tau| + |\lambda| \big) \omega_{\max} < 1,
	\end{equation*}
	we conclude with the Neumann formula that $F$ is a Fredholm operator with index zero.
	
	(vi) Since $P_{\eta,\tau,\lambda} P_{\eta,-\tau,-\lambda} = d I_N $, our assumptions imply that $P_{\eta,\tau,\lambda}$ is bijective in $L^2(\Sigma;\dC^N)$ and its inverse is given by $\frac{1}{d}P_{\eta,-\tau,-\lambda}$. Using~\eqref{C_inv} and \eqref{anti_commutation}, we conclude that
	\begin{equation} \label{equation_large_interaction_strengths}
	I_N +  P_{\eta,\tau,\lambda}\mathcal{C}_z = P_{\eta,\tau,\lambda}(\alpha \cdot \nu)\big( I_N + P_{-\frac{4\eta}{d},-\frac{4\tau}{d},-\frac{4\lambda}{d}}\mathcal{C}_z\big)(\alpha\cdot\nu)\mathcal{C}_z.
	\end{equation}
	By the assumptions in (vi) and the result from (v) the operator $ I_N + P_{-\frac{4\eta}{d},-\frac{4\tau}{d},-\frac{4\lambda}{d}}\mathcal{C}_z$ is bijective. Moreover, as $P_{\eta,\tau,\lambda}$, $\alpha\cdot\nu$, and $\mathcal{C}_z$ are bijective, we conclude that  $I_N +  P_{\eta,\tau,\lambda}\mathcal{C}_z$ is bijective as well.
\end{proof}

In the following theorem we state the main result of this section on the self-adjointness and the spectral properties of the operator $A_{\eta, \tau, \lambda}$ defined in~\eqref{def_A_eta,tau,mu}. Recall that $A_0$ is the free Dirac operator given by~\eqref{def_A_0}, that $P_{\eta, \tau, \lambda}$ is introduced in~\eqref{def_P_eta_tau_lambda}, and that $\Phi_z$ and $\mathcal{C}_z$, $z \in \rho(A_0)$, are the operators defined in~\eqref{def_Phi_lambda} and~\eqref{def_C_lambda}, respectively.

\begin{theorem}\label{selfadjoint_theorem}
	Let $\Sigma$ be the boundary of a bounded Lipschitz domain and let $\eta,\tau,\lambda \in \dR$ such that one of the assumptions \textup{(i)--(vi)} in Proposition~\ref{selfadjoint_prop_1} is fulfilled. Then, $A_{\eta,\tau,\lambda}$ is self-adjoint in $L^2(\mathbb{R}^q; \mathbb{C}^N)$ and for all $z \in \rho(A_0) \cap \rho(A_{\eta,\tau,\lambda})$ the resolvent is given by
	\begin{equation} \label{equation_Krein}
	(A_{\eta,\tau,\lambda} - z)^{-1} = (A_0 - z)^{-1} - \Phi_z(I_N +  P_{\eta,\tau,\lambda}\mathcal{C}_z)^{-1}  P_{\eta,\tau,\lambda}\Phi_{\overline{z}}^*.
	\end{equation}
	Moreover, the following holds:
	\begin{itemize}
		\item[(i)] $\sigma_{\mathrm{ess}}(A_{\eta,\tau,\lambda}) =\sigma_{\mathrm{ess}}(A_0) =  \left(-\infty,-m\right]\cup \left[m,\infty\right)$.
		\item[(ii)] $\sigma_{\mathrm{disc}}(A_{\eta,\tau,\lambda}) $ is finite. 
		\item[(iii)] For $z \in (-m,m)$ one has that $z \in \sigma_\mathrm{p}(A_{\eta,\tau,\lambda})$ if and only if $-1 \in \sigma_\mathrm{p}( P_{\eta,\tau,\lambda} \mathcal{C}_z)$.
	\end{itemize}
\end{theorem}
\begin{proof}
	By Theorem \ref{theorem_Krein_abstract_A_B} (ii), Lemma~\ref{selfadjoint_lemma_1}, and Proposition \ref{selfadjoint_prop_1} we have
	\[\ran(A_{\eta,\tau,\lambda}-z) =L^2(\dR^q;\dC^N) \quad \forall z \in \dC \setminus \dR.\]
	 Thus, the symmetric operator $A_{\eta,\tau,\lambda}$ is self-adjoint. The resolvent formula  and assertion (iii) are direct consequences of  Theorems \ref{theorem_Krein_abstract_A_B} and~\ref{theorem_QBT}.
	
	In order to show item (i) we notice that $\Phi_{\overline{z}}^* = \gamma^{(1/2)}(\overline{z})^*$ is compact as a mapping from $L^2(\dR^q;\dC^N) $ to $ L^2(\Sigma;\dC^N)$, cf. Theorem \ref{theorem_QBT} (ii). Therefore, by~\eqref{equation_Krein} the difference $(A_{\eta,\tau,\lambda} - z)^{-1} - (A_0 - z)^{-1}$  is compact, implying 
	\begin{equation*}
	\sigma_{\mathrm{ess}}(A_{\eta,\tau,\lambda}) = \sigma_{\mathrm{ess}}(A_0) = \sigma(A_0) = \left(-\infty,-m\right]\cup \left[m,\infty\right).
	\end{equation*}
	
	Finally, to see that item~(ii) holds, one can follow line by line the proofs of \cite[Proposition~3.8]{BHOP20} for $q=2$ and \cite[Theorem~5.4~(ii)]{BHM20} for $q=3$; this is possible, as $\dom A_{\eta, \tau, \lambda} \subset H^{1/2}(\mathbb{R}^q \setminus \Sigma; \mathbb{C}^N)$.
\end{proof}

\begin{remark}
  Assume that $\lambda = 0$ and that the number $d = \eta^2 - \tau^2$ satisfies
  \begin{equation} \label{Benhellal1}
    d < \frac{1}{\| \mathcal{C}_z \|_{L^2(\Sigma;\dC^N)\to L^2(\Sigma;\dC^N)}^2}
  \end{equation}
  for a $z \in  \rho(A_0)$. As $\omega_\max \leq \| \mathcal{C}_z \|_{L^2(\Sigma;\dC^N)\to L^2(\Sigma;\dC^N)}$, this implies $\omega_\max^2 d < 1$ and thus, the condition in Proposition~\ref{selfadjoint_prop_1}~(i) is fulfilled.  
  Similarly,  if
  \begin{equation} \label{Benhellal2}
    d> 16\| \mathcal{C}_z \|_{L^2(\Sigma;\dC^N)\to L^2(\Sigma;\dC^N)}^2
  \end{equation}
  holds for some $z \in  \rho(A_0)$, then  
  $\omega_\min \geq \frac{1}{4\| \mathcal{C}_z \|_{L^2(\Sigma;\dC^N)\to L^2(\Sigma;\dC^N)}}$, which follows from Lemma~\ref{lemma_K_alpha}~(iv), implies $d \omega_\min^2 > 1$ and again the condition in Proposition~\ref{selfadjoint_prop_1}~(i) is satisfied.  
  In particular, if \eqref{Benhellal1} or~\eqref{Benhellal2} hold, then all assertions in Theorem~\ref{selfadjoint_theorem} are true. We remark that \eqref{Benhellal1} and~\eqref{Benhellal2} coincide with the conditions (b) and (a) in \cite[Remark~4.4]{Ben22}, respectively, and hence, these results are contained in the analysis in the present paper as well.  
\end{remark}

In the following corollary we present some special cases for which Theorem~\ref{selfadjoint_theorem} can be applied and which are of particular interest: purely Lorentz scalar interactions (i.e. $\eta = \lambda = 0$), non-critical purely anomalous magnetic interactions (i.e. $\eta = \tau = 0$ and $\lambda^2 \neq 4$), and the non-critical confinement case (i.e. $d=-4$ and $\lambda^2 \neq 4$). The confinement case is discussed in greater detail in Section~\ref{section_confinement}, while the result on the purely Lorentz scalar case is also contained in \cite{BHSS21}. The following corollary is an immediate consequence of Proposition~\ref{selfadjoint_prop_1}~(i), (iii), and (iv) and Theorem~\ref{selfadjoint_theorem}.

\begin{corollary}\label{selfadjoint_cor_1}
	Assume that one of the following conditions is satisfied:
	\begin{itemize}
		\item[(i)] $\eta = \lambda = 0$.
		\item[(ii)] $\eta = \tau = 0$ and $\lambda^2 \neq 4$.
		\item[(iii)]$ d= -4$ and $\lambda^2 \neq  4$.
	\end{itemize}
	Then,  all assertions in Theorem~\ref{selfadjoint_theorem} hold.
\end{corollary}

In the following lemma, we consider the special case when $(\alpha \cdot \nu) \mathcal{C}_z + \mathcal{C}_z (\alpha \cdot \nu)  $  is compact in $L^2(\Sigma;\dC^N)$. It turns out that then the quantities $\omega_\min, \omega_\max$ in Lemma~\ref{selfadjoint_lemma_2} are both equal to $\frac{1}{2}$ and hence,  the conditions in Proposition~\ref{selfadjoint_prop_1} reduce to
	\begin{equation}\label{cond_3}
	\left(1-\frac{d}{4}\right)^2 - \lambda^2 \neq 0,
	\end{equation}
	i.e. that the interaction strengths are non-critical.  Furthermore, we prove that for $C^1$-smooth boundaries  $(\alpha \cdot \nu) \mathcal{C}_z + \mathcal{C}_z (\alpha \cdot \nu)$ is indeed compact in $L^2(\Sigma;\dC^N)$. 

\begin{lemma}\label{selfadjoint_cor_2} 
	Assume that $(\alpha \cdot \nu) \mathcal{C}_z + \mathcal{C}_z (\alpha \cdot \nu)  $  is compact in $L^2(\Sigma;\dC^N)$ and that~\eqref{cond_3} holds. Then $I_N + P_{\eta,\tau,\lambda}\mathcal{C}_z$ is isomorphic in $L^2(\Sigma;\dC^N)$ for $z \in \dC \setminus \dR$ and all statements in Theorem~\ref{selfadjoint_theorem} hold.	
	In particular, $(\alpha \cdot \nu) \mathcal{C}_z + \mathcal{C}_z (\alpha \cdot \nu)  $  is compact, if $\Sigma$ is the boundary of a $C^1$ domain.
\end{lemma}
\begin{proof}
	If   $(\alpha \cdot \nu) \mathcal{C}_z + \mathcal{C}_z (\alpha \cdot \nu)  $  is compact,
	then by~\eqref{C_inv}
	\begin{equation*}
	I_N-4\mathcal{C}_z^2 = -4\mathcal{C}_z\big((\alpha \cdot \nu) \mathcal{C}_z + \mathcal{C}_z (\alpha \cdot \nu) \big) (\alpha \cdot \nu)
	\end{equation*} 
	and thus, $I_N-4\mathcal{C}_z^2$ is compact as well.
	Hence, by Lemma~\ref{lemma_K_alpha} also $I_N-4 \mathcal{D}_\alpha^2$ is compact. Since  $I_N+2|\mathcal{D}_\alpha|$ is self-adjoint and bounded from below by $1$, the operator  $I_N+2|\mathcal{D}_\alpha|$ is bijective and its inverse is bounded in $L^2(\Sigma;\dC^N)$. Therefore, 
	\begin{equation*}
	I_N-2|\mathcal{D}_\alpha| =(I_N+2|\mathcal{D}_\alpha|)^{-1} (I_N-4\mathcal{D}_\alpha^2)
	\end{equation*}
	is compact. Thus, the quantities $\omega_{\min}, \omega_{\max}$ in Lemma~\ref{selfadjoint_lemma_2} satisfy $\omega_{\min} = \omega_{\max} = \frac{1}{2}$.
	Consequently, the conditions in Proposition~\ref{selfadjoint_prop_1}~(i)--(ii) yield for $\lambda \neq 0$
	\begin{equation*}
	\bigg|\frac{1-  \frac{d}{4}}{\lambda} \bigg| \neq 1,
	\end{equation*}
	If $\lambda =0$, then the condition in Proposition~\ref{selfadjoint_prop_1}~(i) reads
	\begin{equation*}
	\bigg|1-  \frac{d}{4} \bigg| \neq 0,
	\end{equation*}
	which is equivalent to \eqref{cond_3} with $\lambda =0$. Hence, Proposition~\ref{selfadjoint_prop_1} yields that $I_N + P_{\eta,\tau,\lambda}\mathcal{C}_z$ is bijective and thus, Theorem~\ref{selfadjoint_theorem} can be applied.
	
	Eventually, it can be shown in the same way as in \cite[Theorem~1.2~(c)]{FJR78} that $(\alpha \cdot \nu)\mathcal{C}_z + \mathcal{C}_z (\alpha \cdot \nu) $ is compact in $L^2(\Sigma;\dC^N)$, if $\Sigma$ is the boundary of a compact $C^1$ domain, see also \cite[Remark~3.6]{AMV14} for a similar discussion for $q=3$.
	In  dimension three one  can also apply \cite[Lemma~4.1]{Ben22}, as (H1) in \cite[IV.~A]{Ben22} is fulfilled for the boundary of a compact $C^1$ domain. 
\end{proof}

We finish this section by presenting a statement about the discrete spectrum of $A_{\eta,\tau,\lambda}$, when $q=3$.

\begin{proposition}\label{selfadjoint_prop_3}
	Let $q = 3$ and define the number 
	\begin{equation*}
	M_+ := \sup_{z \in (-m,m)} \|\mathcal{C}_z\|_{L^2(\Sigma;\dC^4) \to L^2(\Sigma;\dC^4)}.
	\end{equation*}
	Then, $M_+$ is finite. Moreover, assume that one of the following holds:
	\begin{itemize}
		\item[(i)]  $|\eta|+ |\tau| + |\lambda| < \frac{1}{M_+}$.
		\item[(ii)]  $|\eta| + |\tau| + |\lambda| \neq 0$ and $\frac{|d|}{|\eta| + |\tau| + |\lambda|} > 4 M_+$.
	\end{itemize}
	Then, $A_{\eta, \tau, \lambda}$ satisfies all claims in Theorem~\ref{selfadjoint_theorem} and $\sigma_{\mathrm{disc}}(A_{\eta,\tau,\lambda}) = \emptyset$.
\end{proposition}
\begin{proof}
	The finiteness of $M_+$ is shown in \cite[Lemma~3.2 and Remark~2.4]{AMV15}. Moreover, as 
	\begin{equation*}
	  \omega_{\max} \leq \|\mathcal{C}_z\|_{L^2(\Sigma;\dC^4) \to L^2(\Sigma;\dC^4)}
	\end{equation*}
	holds for any $z \in \mathbb{C} \setminus ((-\infty, -m] \cup [m, \infty))$, cf. Lemma~\ref{lemma_K_alpha}~(iv), we get $\omega_{\max} \leq M_+$. Hence, assumption~(i) implies the assumption in Proposition~\ref{selfadjoint_prop_1}~(v), while~(ii) yields the assumption in Proposition~\ref{selfadjoint_prop_1}~(vi). In particular, Theorem~\ref{selfadjoint_theorem} can be applied.
	
	To show the claim about the discrete spectrum, which is, by Theorem~\ref{selfadjoint_theorem}, contained in $(-m, m)$, we fix $z \in (-m,m)$ in the following.
	Let us start with (i). The definition of $P_{\eta, \tau, \lambda}$ in~\eqref{def_P_eta_tau_lambda} and the triangle inequality imply
	\begin{equation*}
	\| P_{\eta,\tau,\lambda}\mathcal{C}_z\|_{L^2(\Sigma;\dC^4)\to L^2(\Sigma;\dC^4)} \leq \big( |\eta|+|\tau|+|\lambda| \big) \|\mathcal{C}_z\|_{L^2(\Sigma;\dC^4)\to L^2(\Sigma;\dC^4)} < 1.
	\end{equation*}
	Hence,  $-1 \notin \sigma_{\mathrm{p}}(P_{\eta,\tau,\lambda}\mathcal{C}_z)$, which yields with  Theorem~\ref{selfadjoint_theorem}~(iii) that $ z \notin \sigma_{\mathrm{disc}}(A_{\eta,\tau,\lambda})$.
	
	Next, we turn to (ii). Due to our assumptions, we can write
	\begin{equation*}
	I_4 +  P_{\eta,\tau,\lambda}\mathcal{C}_z = P_{\eta,\tau,\lambda}(\alpha \cdot \nu)\big( I_4 + P_{-\frac{4\eta}{d},-\frac{4\tau}{d},-\frac{4\lambda}{d}}\mathcal{C}_z\big)(\alpha\cdot\nu)\mathcal{C}_z;
	\end{equation*}
	cf.~\eqref{equation_large_interaction_strengths}.
	Now, the assumptions in (ii) and the result from (i) show that $I_4 + P_{-\frac{4\eta}{d},-\frac{4\tau}{d},-\frac{4\lambda}{d}}\mathcal{C}_z$ is injective. Moreover, the other terms in the last displayed formula are injective. Hence, $0 \notin \sigma_\textup{p}(I_4 +  P_{\eta,\tau,\lambda}\mathcal{C}_z)$, which implies with  Theorem~\ref{selfadjoint_theorem}~(iii) that $z \notin \sigma_\textup{disc}(A_{\eta, \tau, \lambda})$.
\end{proof}

\subsection{Confinement} \label{section_confinement}

In this section, we show that for some special combinations of interaction strengths the operator $A_{\eta,\tau,\lambda}$ introduced in~\eqref{def_A_eta,tau,mu} decomposes into two Dirac operators acting in $L^2(\Omega_\pm; \mathbb{C}^N)$ with boundary conditions. From a physical point of view, this means that a particle that is located in one of the domains $\Omega_\pm$ will remain in $\Omega_\pm$ for all times, i.e. $\Sigma$ is impenetrable for the particle. Moreover, in all situations in the confinement case we  describe self-adjoint realizations of $A_{\eta,\tau,\lambda}$, which is a novelty when all three parameters $\eta, \tau, \lambda$ are present, and we can do that in the situation that $\Sigma$ is only Lipschitz smooth. 
The phenomenon of confinement was already observed in the case that $\Sigma \subset \mathbb{R}^3$ is a sphere in \cite{DES89} and then for $C^2$-smooth surfaces $\Sigma \subset \mathbb{R}^3$ in various situations in \cite{AMV15, BEHL19, BHM20, Ben21, Ben22} and for $C^\infty$-smooth $\Sigma \subset \mathbb{R}^2$ in \cite{BHOP20, CLMT21}. 
Recall that $P_{\eta, \tau, \lambda}$ is defined by~\eqref{def_P_eta_tau_lambda}.
The starting point is the following observation:

\begin{lemma} \label{lemma_confinement}
  Let $\eta, \tau, \lambda \in \mathbb{R}$, set $d = \eta^2-\tau^2-\lambda^2$, and let $A_{\eta, \tau, \lambda}$ be defined as in~\eqref{def_A_eta,tau,mu}. Then, the following holds:
  \begin{itemize}
    \item[(i)] If $d \neq -4$, then there exists a matrix-valued function $Q_{\eta, \tau, \lambda}: \Sigma \rightarrow \mathbb{C}^{N \times N}$ depending on the parameters $\eta,\tau,\lambda$ such that $Q_{\eta, \tau, \lambda}$ is pointwise invertible and $f \in \dom A_{\eta,\tau,\lambda}$ if and only if $f \in H^{1/2}_\alpha(\mathbb{R}^q \setminus \Sigma; \mathbb{C}^N)$ and
    \begin{equation} \label{transmission_condition}
      \gamma_D^+ f_+ = Q_{\eta, \tau, \lambda} \gamma_D^- f_-.
    \end{equation}
    \item[(ii)] If $d=-4$, then $A_{\eta, \tau, \lambda} = A_{\eta,\tau,\lambda}^+ \oplus A_{\eta,\tau,\lambda}^-$, where $A_{\eta,\tau,\lambda}^\pm$ are the operators acting in $L^2(\Omega_\pm; \mathbb{C}^N)$ given by
    \begin{equation} \label{def_op_confinement}
      \begin{split}
        A_{\eta,\tau,\lambda}^\pm f &= \big(-i (\alpha \cdot \nabla) + m \alpha_0\big) f, \\
        \dom A_{\eta,\tau,\lambda}^\pm &= \big\{ f \in H_\alpha^{1/2}(\Omega_\pm; \mathbb{C}^N): \big( 2 I_N \mp i (\alpha \cdot \nu) P_{\eta, \tau, \lambda} \big) \gamma_D^\pm f = 0 \big\}.
      \end{split}
    \end{equation}
  \end{itemize}
\end{lemma}
\begin{proof}
  Define the function $R: \Sigma \rightarrow \mathbb{C}^{N \times N}$ by
  \begin{equation*}
    R := \frac{i}{2} (\alpha \cdot \nu) P_{\eta, \tau, \lambda} = \frac{i}{2} (\alpha \cdot \nu) \big( \eta I_N + \tau \alpha_0 + \lambda i (\alpha \cdot \nu) \alpha_0 \big) .
  \end{equation*}
  Then, the transmission condition for $f \in \dom A_{\eta, \tau, \lambda}$ is equivalent to
  \begin{equation} \label{transmission_condition1}
    (I_N - R) \gamma_D^+ f_+ = (I_N + R) \gamma_D^- f_-.
  \end{equation}
  With the anti-commutation relations~\eqref{anti_commutation} it is not difficult to see that $R^2 = -\frac{d}{4} I_N$. Hence, if $d \neq -4$, then $I_N \pm R$ is invertible and~\eqref{transmission_condition1} is equivalent to
  \begin{equation*}
    \gamma_D^+ f_+ = (I_N - R)^{-1} (I_N + R) \gamma_D^- f_- = \frac{4}{d+4} (I_N + R)^2 \gamma_D^- f_-,
  \end{equation*}
  which is~\eqref{transmission_condition} with $Q_{\eta, \tau, \lambda} = \frac{4}{d+4} (I_N + R)^2$ and shows item~(i).
  
  If $d=-4$, then $(I_N \mp R) (I_N \pm R) = 0$. Hence, if $f \in \dom A_{\eta, \tau, \lambda}$, then 
  \begin{equation*} 
    (I_N \mp R)^2 \gamma_D^\pm f_\pm = 2 \left( I_N \mp R \right) \gamma_D^\pm f_\pm = 0.
  \end{equation*}
  On the other hand, if $f_\pm \in \dom A_{\eta, \tau, \lambda}^\pm$, then clearly~\eqref{transmission_condition1} holds, as the left and the right hand side in~\eqref{transmission_condition1} are zero, and thus, $f_+ \oplus f_- \in \dom A_{\eta, \tau, \lambda}$. This finishes the proof of item~(ii).
\end{proof}

\begin{remark}
  Assume that $\eta, \tau, \lambda \in \mathbb{R}$ are such that $d = \eta^2-\tau^2-\lambda^2=-4$. Let $\nu = (\nu_1, \dots, \nu_q)$, define the function $\mathcal{N}: \Sigma \rightarrow \mathbb{C}^{N/2 \times N/2}$ by
  \begin{equation*}
    \mathcal{N} = \begin{cases} \nu_1 + i \nu_2, & q=2, \\ \sigma \cdot \nu, &q=3, \end{cases}
  \end{equation*}
  and note that
  \begin{equation*}
    \alpha \cdot \nu = \begin{pmatrix} 0 & \mathcal{N}^* \\ \mathcal{N} & 0 \end{pmatrix}
  \end{equation*}
  holds.
  Then, the boundary conditions for $f \in \dom A_{\eta, \tau, \lambda}^\pm$ can be rewritten as
  \begin{equation} \label{boundary_conditions_alternative1}
    0 = \big( 2 I_N \pm \lambda \alpha_0 \mp i (\alpha \cdot \nu) (\eta I_N + \tau \alpha_0) \big) \gamma_D^\pm f = \begin{pmatrix} (2 \pm \lambda) I_{N/2} & \mp i(\eta-\tau) \mathcal{N}^* \\ \mp i (\eta + \tau) \mathcal{N} & (2 \mp \lambda) I_{N/2} \end{pmatrix} \gamma_D^\pm f,
  \end{equation}
  i.e. $f = (f_1, f_2) \in \dom A_{\eta, \tau, \lambda}^\pm$ with $f_1, f_2 \in L^2(\Omega_\pm; \mathbb{C}^{N/2})$ if and only if $(f_1, f_2) \in H^{1/2}_\alpha(\Omega_\pm; \mathbb{C}^{N})$ and
  \begin{equation} \label{boundary_conditions_alternative}
    \begin{split}
      (2 \pm \lambda) \gamma_D^\pm f_1 \mp i (\eta - \tau) \mathcal{N}^* \gamma_D^\pm f_2 &= 0, \\
      \mp i (\eta+\tau) \mathcal{N} \gamma_D^\pm f_1 + (2 \mp \lambda) \gamma_D^\pm f_2 &= 0.
    \end{split}
  \end{equation}
  If $\lambda^2 \neq 4$, then $d=-4$ implies  $\eta^2 \neq \tau^2$ and hence $d=-4$ and $\mathcal{N} \mathcal{N}^*=I_{N/2}$ yield that the two equations in~\eqref{boundary_conditions_alternative} are equivalent to each other. Thus, in the latter case, one of the equations in~\eqref{boundary_conditions_alternative} is sufficient to check if $f \in H^{1/2}_\alpha(\Omega_\pm; \mathbb{C}^N)$ belongs to $\dom A_{\eta, \tau, \lambda}^\pm$.
\end{remark}

In the following two propositions we discuss the (essential) self-adjointness of $A_{\eta, \tau, \lambda}$ in the confinement case. First, we discuss the non-critical case, i.e. when  $(\frac{d}{4}-1)^2 - \lambda^2 \neq 0$. Using $d=-4$, we see that the condition for being non-critical is, in the confinement case, equivalent to $\lambda^2 \neq 4$. We note that the Dirac operators with the same boundary conditions as $A_{\eta,\tau,\lambda}^\pm$ were studied in the non-critical case in space dimension $q=2$ in \cite{BFSB17_1, PvdB20} and in dimension $q=3$ in \cite{BHM20}.

\begin{proposition} \label{proposition_confinement_non_critical}
  Let $\eta, \tau, \lambda \in \mathbb{R}$ such that $d=-4$ and $\lambda^2 \neq 4$ and let $A_{\eta, \tau, \lambda}^\pm$ be defined by~\eqref{def_op_confinement}. Then $A_{\eta, \tau, \lambda}^\pm$ is self-adjoint in $L^2(\Omega_\pm; \mathbb{C}^N)$ and the following is true:
  \begin{itemize}
    \item[(i)] $\sigma(A_{\eta, \tau, \lambda}^+)$ is purely discrete.
    \item[(ii)] $\sigma_\textup{ess}(A_{\eta, \tau, \lambda}^-)=(-\infty,-m] \cup [m,\infty)$ and $\sigma_\textup{disc}(A_{\eta, \tau, \lambda}^-)$ is finite.
  \end{itemize}
\end{proposition}
\begin{proof}
  By Proposition~\ref{selfadjoint_prop_1}~(iii) and Theorem~\ref{selfadjoint_theorem} the operator $A_{\eta, \tau, \lambda}$ is self-adjoint in $L^2(\mathbb{R}^q; \mathbb{C}^N)$, $\sigma_\textup{ess}(A_{\eta, \tau, \lambda})=(-\infty,-m] \cup [m,\infty)$, and $\sigma_\textup{disc}(A_{\eta, \tau, \lambda})$ is finite. Taking Lemma~\ref{lemma_confinement} into account, this can only be true if $A_{\eta,\tau,\lambda}^\pm$ is self-adjoint in $L^2(\Omega_\pm; \mathbb{C}^N)$ and, since $\dom A_{\eta,\tau,\lambda}^+ \subset H^{1/2}(\Omega_+; \mathbb{C}^N)$ is compactly embedded in $L^2(\Omega_+;\mathbb{C}^N)$, if items~(i) and~(ii) are true.
\end{proof}

In the following proposition we discuss the case of confinement $d=-4$, when the interaction strengths $\eta, \tau, \lambda \in \mathbb{R}$ are critical, i.e. when $\lambda^2=4$. Note that the latter two conditions imply $\eta^2 = \tau^2$.

\begin{proposition} \label{proposition_confinement_critical}
  Let $\eta, \tau, \lambda \in \mathbb{R}$ such that $d=-4$ and $\lambda^2 = 4$ and let $A_{\eta, \tau, \lambda}^\pm$ be defined by~\eqref{def_op_confinement}. Then the following holds:
  \begin{itemize}
    \item[(i)] For $\eta=\tau=0$ and $\lambda^2=4$ the operators $A_{0,0,\lambda}^\pm$ are essentially 
    self-adjoint in $L^2(\Omega_\pm; \mathbb{C}^N)$ and the following is true:
    \begin{itemize}
      \item[(a)] $-\sgn(\lambda) m$ is an eigenvalue of $\overline{A_{0,0,\lambda}^{+}}$ with infinite multiplicity and $\sigma(\overline{A_{0,0,\lambda}^{+}}) \setminus \{ -\sgn(\lambda) m \}$ is purely discrete.
      \item[(b)] $\sigma(\overline{A_{0,0,\lambda}^{-}})=(-\infty,-m] \cup [m,\infty)$.
    \end{itemize}
    \item[(ii)] For $\eta^2=\tau^2\neq0$ and $\lambda^2=4$ one has $A_{\eta,\tau,\lambda}^{\sgn(\eta \tau \lambda)} = A_{0,0,\lambda}^{\sgn(\eta \tau \lambda)}$ and there exist $\hat{\eta}, \hat{\tau}, \hat{\lambda} \in \mathbb{R}$ such that $\hat{d} := \hat{\eta}^2-\hat{\tau}^2-\hat{\lambda}^2 = -4$, $(\frac{\hat{d}}{4}-1)^2 \neq \hat{\lambda}^2$, and $A_{\eta,\tau,\lambda}^{-\sgn(\eta \tau \lambda)}  = A_{\hat{\eta},\hat{\tau},\hat{\lambda}}^{-\sgn(\eta \tau \lambda)}$.
  \end{itemize}
\end{proposition}

We remark that the operators $\overline{A_{0,0,2}^\pm}$ and $\overline{A_{0,0,-2}^\pm}$ are Dirac operators with \textit{zigzag} type boundary conditions and their spectra without $\pm m$ are closely related to the spectra of the Dirichlet Laplacians in $\Omega_\pm$. We refer to \cite{EH22, H21, S95} for studies on Dirac operators with zigzag type boundary conditions. 
It is remarkable that in the situation $\lambda^2=4$ and $\eta^2=\lambda^2\neq0$ one of the operators $\overline{A_{\eta,\tau,\lambda}^\pm}$ is a Dirac operator with zigzag type boundary conditions, while the other one is a Dirac operator with boundary conditions, as they are treated in Proposition~\ref{proposition_confinement_non_critical}.
In particular, the spectral properties of $\overline{A_{\eta,\tau,\lambda}^\pm}$ in the critical confinement case can be understood with the help of the Dirac operator with zigzag boundary conditions in Proposition~\ref{proposition_confinement_critical}~(i) and Proposition~\ref{proposition_confinement_non_critical} and in all cases $A_{\eta,\tau,\lambda}^\pm$ is essentially self-adjoint.

\begin{proof}[Proof of Proposition~\ref{proposition_confinement_critical}]
  Throughout the proof, we often use for $f \in L^2(\Omega_\pm; \mathbb{C}^N)$ the notation $f=(f_1,f_2)$ with $f_1,f_2 \in L^2(\Omega_\pm; \mathbb{C}^{N/2})$.
  We sketch the proof of item~(i) for $\lambda=2$, for details see also \cite{H21, S95} or \cite[Theorem~2.4]{CLMT21}; the proof for $\lambda=-2$ follows the same lines. 
  We make use of the differential expressions that are given by
  \begin{equation} \label{def_D}
    \mathcal{D} := \begin{cases} -i(\partial_1 + i \partial_2), &q=2, \\ -i (\sigma \cdot \nabla), &q=3, \end{cases}
    \qquad \mathcal{D}^* := \begin{cases} -i(\partial_1 - i \partial_2), &q=2, \\ -i (\sigma \cdot \nabla), &q=3. \end{cases}
  \end{equation}
  With the help of these expressions we can write
  \begin{equation} \label{Dirac_block_structure}
    -i (\alpha \cdot \nabla) + m \alpha_0 = \begin{pmatrix} m I_{N/2} & \mathcal{D}^* \\ \mathcal{D} & -m I_{N/2} \end{pmatrix}.
  \end{equation}
  By~\eqref{boundary_conditions_alternative} a function $f=(f_1,f_2) \in H^{1/2}_\alpha(\Omega_+;\mathbb{C}^N)$ belongs to $\dom A_{0,0,2}^+$ if and only if
  \begin{equation} \label{bc_lambda_2}
    \gamma_D^+ f_1=0,
  \end{equation}
  while there are no boundary conditions for $f_2$. 
  Let $f = (f_1, f_2) \in \dom A_{0,0,2}^+$. Then it follows from $f \in H^{1/2}_\alpha(\Omega_+; \mathbb{C}^N)$ and~\eqref{Dirac_block_structure} that $f_1, f_2 \in H^{1/2}(\Omega_+; \mathbb{C}^{N/2})$ and $\mathcal{D} f_1, \mathcal{D}^* f_2 \in L^2(\Omega_+; \mathbb{C}^{N/2})$. We claim that one even has $f_1 \in H^1_0(\Omega_+; \mathbb{C}^{N/2})$. 
  To see this, define the function $g_+:=(f_1,0)$. Since $\mathcal{D} f_1 \in L^2(\Omega_+; \mathbb{C}^{N/2})$, it follows from~\eqref{Dirac_block_structure} that $g_+  \in H^{1/2}_\alpha(\Omega_+; \mathbb{C}^N)$ and hence, $g := g_+ \oplus 0 \in H^{1/2}_\alpha(\mathbb{R}^q \setminus \Sigma; \mathbb{C}^N)$. Since $f \in \dom A_{0,0,2}^+$, we have $\gamma_D^+ f_1=0$ and thus, $g \in \ker \Gamma_0^{(1/2)} = \dom A_0 = H^1(\mathbb{R}^q; \mathbb{C}^N)$, cf. Theorem~\ref{theorem_QBT}. Hence, $f_1 \in H^1_0(\Omega_+; \mathbb{C}^{N/2})$ and we conclude that
  \begin{equation} \label{dom_A_confinement_critical}
    \dom A_{0,0,2}^+ = H^1_0(\Omega_+; \mathbb{C}^{N/2}) \oplus \{ f \in H^{1/2}(\Omega_+; \mathbb{C}^{N/2}): \mathcal{D}^* f \in L^2(\Omega_+; \mathbb{C}^{N/2})\}.
  \end{equation}

  Recall that $\mathcal{D}, \mathcal{D}^*$ are given by~\eqref{def_D} and define the operator
  \begin{equation*}
    \begin{split}
      \mathcal{T} f &= \mathcal{D} f, \quad \dom \mathcal{T} = H^1_0(\Omega_+; \mathbb{C}^{N/2}).
    \end{split}
  \end{equation*}
  It is not difficult to see that $\mathcal{T}$ is closed and that its adjoint is given by
  \begin{equation*}
    \begin{split}
      \mathcal{T}^* f &= \mathcal{D}^* f, \quad \dom \mathcal{T}^* = \{ f \in L^2(\Omega_+; \mathbb{C}^{N/2}): \mathcal{D}^* f \in L^2(\Omega_+; \mathbb{C}^{N/2}) \}.
    \end{split}
  \end{equation*}
  Let $P$ be the projection in $L^2(\Omega_+; \mathbb{C}^N)$ defined for $f=(f_1,f_2) \in L^2(\Omega_+; \mathbb{C}^N)$ by $P (f_1,f_2)=(0,f_2)$. Then $\dom \mathcal{T}^* \simeq P H^0_\alpha(\Omega_+; \mathbb{C}^N)$
  and thus, as $C_0^\infty(\overline{\Omega_+}; \mathbb{C}^{N})$ is dense in $H_\alpha^0(\Omega_+; \mathbb{C}^N)$, also $C_0^\infty(\overline{\Omega_+}; \mathbb{C}^{N/2})$ is dense in $\dom \mathcal{T}^*$ with respect to the graph norm. Therefore, we conclude from~\eqref{Dirac_block_structure} and~\eqref{dom_A_confinement_critical} that
  \begin{equation*}
    \overline{A_{0,0,2}^+} = \begin{pmatrix} m I_{N/2} & \mathcal{T}^* \\ \mathcal{T} & -m I_{N/2} \end{pmatrix}.
  \end{equation*}
  In particular, $\overline{A_{0,0,2}^+}$ has a supersymmetric structure, cf. \cite[Appendix~A]{H21}, and thus $\overline{A_{0,0,2}^+}$ is self-adjoint. Next, choose $c_1, c_2 \in \mathbb{R}$ such that $x_1 - i x_2 - (c_1 - i c_2) \neq 0$ for all $x = (x_1, \dots, x_q) \in \Omega_+$, and define for $n>2$ the functions
  \begin{equation} \label{kernel_element}
    f_n(x) := (x_1-ix_2-c_1 + i c_2)^{-n} e_N, \quad x=(x_1, \dots, x_q) \in \Omega_+,
  \end{equation}
  where $e_N$ is the $N$-th canonical basis vector in $\mathbb{C}^N$. Then, a direct calculation shows that $f_n \in \ker(\overline{A_{0,0,2}^+}+m)$, i.e. $-m$ is an eigenvalue of $\overline{A_{0,0,2}^+}$ with infinite multiplicity. Finally, one shows in a similar way as in \cite[Theorem~3.4]{H21} with the help of \cite[Proposition~A.2]{H21} that $\sigma(\overline{A_{0,0,2}^+}) = \{ -m \} \cup \{ \pm \sqrt{m^2+z}: z \in \sigma(-\Delta_D^{\Omega_+}) \}$, where $-\Delta_D^{\Omega_+}$ is the Dirichlet Laplacian in $\Omega_+$, and hence, as $\Omega_+$ is a bounded Lipschitz domain,  $\sigma(\overline{A_{0,0,2}^+}) \setminus \{ -m \}$ is purely discrete; cf. \cite[Theorem~2.1 and Remark~2.3]{EH22} as well. Hence, item~(i)~(a) is shown.
  
  To see the claims of assertion~(i)~(b), note that the boundary conditions for $A_{0,0,2}^-$ are $2(I_N-\alpha_0) \gamma_D^- f = 0$. Therefore, one can conclude all claims with similar arguments as above; cf. \cite[Theorem~3.4]{H21} for details in dimension three.  In particular, as $\Omega_-$ is an exterior domain it follows from \cite[Theorem~2.1]{EH22} that $\sigma(\overline{A_{0,0,2}^-})= \{ m \} \cup \{ \pm \sqrt{m^2+z}: z \in \sigma(-\Delta_D^{\Omega_-}) \} = (-\infty,-m] \cup [m,\infty)$.
  
  Let us prove item~(ii). In this case, we have $\lambda^2=4$ and $\eta^2 = \tau^2$ and thus, the boundary conditions in~\eqref{boundary_conditions_alternative} are equivalent to
  \begin{equation} \label{boundary_conditions_alternative_crit}
    \begin{split}
      2 (1 \pm \textup{sgn}\, \lambda) \gamma_D^\pm f_1 \mp i \eta(1 - \textup{sgn}\,(\eta \tau)) \mathcal{N}^* \gamma_D^\pm f_2 &= 0, \\
      \mp i \eta (1 + \textup{sgn}\,(\eta \tau)) \mathcal{N} \gamma_D^\pm f_1 + 2 (1 \mp \textup{sgn}\, \lambda) \gamma_D^\pm f_2 &= 0.
    \end{split}
  \end{equation}
  In the following, we consider $\eta=\tau \neq 0$ and $\lambda=2$; in the other cases the proof is similar. In this situation  
  the boundary conditions for $A_{\eta,\eta,2}^+$ are $\gamma_D^+ f_1=0$, while there are no boundary conditions for $f_2$. Hence, we conclude $A_{\eta,\eta,2}^+ = A_{0,0,2}^+$, cf. \eqref{bc_lambda_2}.
  
  Moreover, by~\eqref{boundary_conditions_alternative_crit} the boundary conditions for $f = (f_1,f_2) \in \dom A_{\eta,\eta,2}^-$  are 
  \begin{equation} \label{bc_crit}
    2 i \mathcal{N} \eta \gamma_D^- f_1 + 4 \gamma_D^- f_2 = 0.
  \end{equation}
  Choose $\hat{\eta}, \hat{\tau}, \hat{\lambda} \in \mathbb{R}$ such that $\eta = 2 \frac{\hat{\eta} + \hat{\tau}}{2+\hat{\lambda}}$, $\hat{d} := \hat{\eta}^2-\hat{\tau}^2-\hat{\lambda}^2 = -4$, and $\hat{\lambda}^2 \neq 4$. Then, with the help of~\eqref{boundary_conditions_alternative} it is not difficult to see that~\eqref{bc_crit} is equivalent to the boundary conditions for $A_{\hat{\eta},\hat{\tau},\hat{\lambda}}^-$. This finishes the proof of this proposition.
\end{proof}

Finally, we remark that in dimension two one has that $m$ is an eigenvalue of $\overline{A_{0,0,2}^-}$ of infinite multiplicity. Indeed, let $x_0 = (x_{0,1}, x_{0,2}) \in \Omega_+$. Then, in a similar way as in~\eqref{kernel_element} one defines the functions
\begin{equation*} 
  g_n(x) := (x_1+ix_2-x_{0,1} - i x_{0,2})^{-n} e_1, \quad x=(x_1, x_2) \in \Omega_-,
\end{equation*}
where $e_1$ is the first canonical basis vector in $\mathbb{R}^2$, and shows via a direct calculation that $g_n \in \ker(\overline{A_{0,0,2}^-}-m)$, cf. \cite[Theorem~2.4]{CLMT21}. Similarly, one gets that $-m$ is an eigenvalue of $\overline{A_{0,0,-2}^-}$ of infinite multiplicity. However, this argument fails in dimension $q=3$.

\section{Dirac operators with singular interactions supported on smooth curves and surfaces}\label{critical-section}

In this section, we study the self-adjointness of Dirac operators with singular $\delta$-shell interactions also in the case of critical interaction strengths. However, since the proofs of some necessary auxiliary results are very technical and long, we only state them here and refer to \cite{H22} for details. In \cite{H22}, the operator $\mathcal{C}_z$ in~\eqref{def_C_lambda} is analyzed in a more detailed way, the main ingredients there are periodic pseudodifferential operators in dimension $q=2$, cf. \cite{SV}, and the theory of pseudo-homogeneous kernels for $q=3$, see \cite{N01}.  
Parts of the results presented below are contained in \cite{BH20, OV16} for $q=3$ with purely electrostatic critical interaction strengths, in \cite{BHOP20} for $q=2$ with combinations of electrostatic and Lorentz scalar interactions, and in \cite{Ben21, Ben22} for $q=3$ with combinations of electrostatic and Lorentz scalar interactions, see also \cite{BP22} for a recent contribution to the study of the essential spectrum under the latter assumptions. The proofs of the results in this section make heavily use of Schur complement techniques -- a strategy which was first used in this context in \cite{BHOP20}.

Throughout this section let $\Omega_+ \subset \mathbb{R}^q$ be a bounded and simply connected $C^\infty$-domain and let, as usual, $\Omega_- := \mathbb{R}^q \setminus \overline{\Omega_+}$ and $\Sigma := \partial \Omega_+$, i.e.
$\Sigma$ is a $C^\infty$-smooth bounded and closed curve in $\mathbb{R}^2$ or a $C^\infty$-smooth closed and compact surface in $\mathbb{R}^3$. Note that in this situation the spaces $H^s_\alpha(\Sigma; \mathbb{C}^N)$, $s \in [-\frac{1}{2}, \frac{1}{2}]$, defined in \eqref{def_H_U} and~\eqref{def_H_U_dual} with $\mathbb{U} = \alpha \cdot \nu$ coincide with $H^s(\Sigma; \mathbb{C}^N)$. Moreover, we assume in this section that $m>0$, as the bounded perturbation $m \alpha_0$ does not influence the self-adjointness and interesting spectral properties only appear for $m > 0$.

In a similar way as in~\eqref{def_A_eta,tau,mu} we define for $\eta, \tau, \lambda \in \mathbb{R}$ the operator
\begin{equation} \label{def_A_max}
  \begin{split}
    B_{\eta, \tau, \lambda} f &= \big(-i (\alpha \cdot \nabla) + m \alpha_0\big) f_+ \oplus \big(-i (\alpha \cdot \nabla) + m \alpha_0\big) f_-, \\
   	\dom B_{\eta,\tau,\lambda} &= \left\{ f \in H^{0}_\alpha(\mathbb{R}^q \setminus \Sigma; \mathbb{C}^N): i (\alpha \cdot \nu) (\gamma_D^+ f_+ - \gamma_D^- f_-) \right.\\
    &\hspace{130 pt }\left.+ P_{\eta,\tau,\lambda}\frac{1}{2} (\gamma_D^+ f_+ + \gamma_D^- f_-) = 0 \right\},
  \end{split}
\end{equation}
where $\gamma_D^\pm$ are the Dirichlet trace operators defined in Corollary~\ref{corollary_trace_extension} and $P_{\eta,\tau,\lambda}$ is given by~\eqref{def_P_eta_tau_lambda}. Note that the operator $B_{\eta, \tau, \lambda}$ is defined on a larger set as $A_{\eta, \tau, \lambda}$ in~\eqref{def_A_eta,tau,mu}. However, it will turn out that in the non-critical case $\lambda^2 \neq (\frac{d}{4} - 1)^2$, which includes all combinations of interaction strengths treated in Section~\ref{section_Lipschitz}, one always has $\dom B_{\eta, \tau, \lambda} \subset H^1(\mathbb{R}^q \setminus \Sigma; \mathbb{C}^N)$ and hence, $A_{\eta, \tau, \lambda} = B_{\eta, \tau, \lambda}$; cf. Theorem~\ref{theorem_self_adjoint_critical}.

In order to show the self-adjointness and to study the spectral properties of $B_{\eta,\tau,\lambda}$ for all values of $\eta, \tau, \lambda \in \mathbb{R}$ we make use of the ordinary boundary triple $\{ L^2(\Sigma; \mathbb{C}^N), \Upsilon_0, \Upsilon_1 \}$ from Theorem~\ref{theorem_OBT} and the associated notations. 
For the operators $\Lambda$ and $V$ appearing in the boundary triple  in Theorem~\ref{theorem_OBT} we make specific choices. We will introduce $\Lambda$ as an operator acting on scalar-valued functions, but we will identify it with $\Lambda I_M$ which acts on vector-valued functions for various values of $M \in \mathbb{N}$.
In space dimension $q=2$, we write $\mathbb{T} := \mathbb{R} / \mathbb{Z} \simeq [0,1)$, and denote by $\mathcal{D}(\mathbb{T})$ and $\mathcal{D}'(\mathbb{T})$ the sets of all $1$-periodic test functions and distributions in $\mathbb{R}$, respectively, by $\ell$ the length of $\Sigma$, by $\gamma: [0, \ell] \rightarrow \mathbb{R}^2$ an arc length parametrization of $\Sigma$, and we introduce the mappings  $e_n(t)=e^{2\pi i n t}$, $n \in \mathbb{Z}$, and
$U: \mathcal{D}'(\Sigma) \rightarrow \mathcal{D}'(\mathbb{T})$ by
\begin{equation*} 
  \langle U f, \varphi \rangle_{\mathcal{D}'(\mathbb{T}) \times \mathcal{D}(\mathbb{T})} := \big\langle f,  \ell^{-1} \varphi (\ell^{-1} \gamma^{-1} (\cdot )) \big\rangle_{\mathcal{D}'(\Sigma) \times \mathcal{D}(\Sigma)}, \quad \varphi \in \mathcal{D}(\mathbb{T}),
\end{equation*}
where $\langle \cdot, \cdot \rangle_{X' \times X}$ denotes the sesquilinear duality product in $X' \times X$ for a Hilbert space $X$.
Then we define for a constant $c_\Lambda > 0$ and $t \in \mathbb{R}$ the map 
\begin{equation} \label{def_Lambda_2d}
  \Lambda^t u :=  U^{-1} \left(\frac{4 \pi}{\ell} \right)^{t/2} \sum_{n \in \mathbb{Z}} (c_\Lambda + |n|)^{t/2} \langle U u, e_{n} \rangle_{\mathcal{D}'(\mathbb{T}) \times \mathcal{D}(\mathbb{T})} e_n, \qquad u \in \mathcal{D}'(\Sigma).
\end{equation}
If $q=3$, we denote by  $\mathcal{S}$  the single layer boundary integral operator  associated with $-\Delta + 1$ that is acting on $\varphi \in C^\infty(\Sigma; \mathbb{C})$ by
\begin{equation*}
  \mathcal{S} \varphi(x) = \int_\Sigma \frac{e^{-|x-y|}}{4 \pi |x-y|} \varphi(y) d \sigma(y), \quad x \in \Sigma.
\end{equation*}
It is known that for any $s \in \mathbb{R}$ this map can be extended to a bounded operator $\mathcal{S}: H^{s}(\Sigma; \mathbb{C}) \rightarrow H^{s+1}(\Sigma; \mathbb{C})$; cf. \cite[Theorem~7.2]{M00}. Define now with $c_\Lambda>0$
\begin{equation} \label{def_Lambda_3d}
  \Lambda := (\mathcal{S}^{-1} + c_\Lambda)^{1/2}.
\end{equation}
In both cases $q \in \{ 2,3 \}$ the operator $\Lambda$ gives rise to a bijective operator from $H^s(\Sigma; \mathbb{C})$ to $H^{s-1/2}(\Sigma; \mathbb{C})$ for all $s \in \mathbb{R}$, see \cite{H22} for details. Moreover, the realization of $\Lambda$ in $L^2(\Sigma; \mathbb{C})$ is self-adjoint with $\dom \Lambda = H^{1/2}(\Sigma; \mathbb{C})$, one has $\Lambda \geq c_\Lambda^{1/2}$, and, since $\Lambda^{-1}$ is compact by Rellich's embedding theorem, this realization in $L^2(\Sigma; \mathbb{C})$ satisfies $\sigma_\textup{ess}(\Lambda)=\emptyset$.

Next, the map $V$ appearing in the definition of the boundary triple in Theorem~\ref{theorem_OBT} is defined, depending on the space dimension, by
\begin{equation} \label{def_V}
  V = \begin{pmatrix} 1 & 0 \\ 0 & -i (\nu_1 - i \nu_2) \end{pmatrix} \quad \text{for } q=2 \text{ and} \quad V = \begin{pmatrix} I_2 & 0 \\ 0 & -i (\sigma \cdot \nu) \end{pmatrix} \quad \text{for } q=3.
\end{equation}
Moreover, we introduce the matrix $B \in \mathbb{C}^{N \times N}$ by
\begin{equation*}
  B := \begin{pmatrix} (\eta + \tau) I_{N/2} & \lambda I_{N/2} \\ \lambda I_{N/2} & (\eta - \tau) I_{N/2} \end{pmatrix}.
\end{equation*}
Then, we have $P_{\eta,\tau,\lambda} = V^* B V$.
Define in $L^2(\Sigma; \mathbb{C}^N)$ the operator $\Theta$ by
\begin{equation} \label{def_Theta}
  \begin{split}
    \Theta \varphi &= \Lambda \big( B - 4 V (\alpha \cdot \nu) \widetilde{\mathcal{C}}_0 (\alpha \cdot \nu) V^* \big) \Lambda \varphi, \\
      \dom \Theta &= \big\{ \varphi \in L^2(\Sigma; \mathbb{C}^N): \big( B - 4 V (\alpha \cdot \nu) \widetilde{\mathcal{C}}_0 (\alpha \cdot \nu) V^* \big) \Lambda \varphi \in H^{1/2}(\Sigma; \mathbb{C}^N) \big\},
  \end{split}
\end{equation} 
where $\widetilde{\mathcal{C}}_0$ is defined as in Proposition~\ref{proposition_extensions_Phi_C}.
Using the relations $P_{\eta,\tau,\lambda} = V^* B V$ and
\begin{equation*}
  \begin{split}
    \frac{1}{2} (\gamma_D^+ f_+ + \gamma_D^- f_-)  &= V^* \Lambda \Upsilon_0 f, \\
    i (\alpha \cdot \nu) (\gamma_D^+ f_+ - \gamma_D^- f_-) &= -V^* \Lambda^{-1} \Upsilon_1 f - 4  (\alpha \cdot \nu) \widetilde{\mathcal{C}}_0 (\alpha \cdot \nu) V^* \Lambda \Upsilon_0 f,
  \end{split}
\end{equation*}
it is not difficult to see that the transmission conditions in the definition of $B_{\eta, \tau, \lambda}$ are equivalent to $\Upsilon_1 f - \Theta \Upsilon_0 f = 0$, i.e. we have
\begin{equation} \label{equation_B_BT}
  \dom B_{\eta, \tau, \lambda} = \dom T^{(0)} \upharpoonright \ker (\Upsilon_1 - \Theta \Upsilon_0).
\end{equation}
Therefore, to understand the properties of $B_{\eta,\tau,\lambda}$, in view of Theorem~\ref{theorem_Krein_abstract}, it suffices to study the associated parameter $\Theta$. 
In order to show the self-adjointness of $\Theta$, a detailed analysis of the involved integral operators is required. Define the operator $\mathcal{R}$ by
\begin{equation} \label{def_R}
  \mathcal{R} \varphi (x) = \lim_{\varepsilon \searrow 0} \int_{\Sigma \setminus B(x,\varepsilon)} r(x,y) \varphi(y) d \sigma(y), \qquad \varphi \in C^\infty(\Sigma; \mathbb{C}^{N/2}),
\end{equation}
where the integral kernel $r$ is given by
\begin{equation*}
  r(x,y) = \begin{cases} -\frac{2}{\pi}\frac{\nu_1(y) + i \nu_2(y)}{x_1 + i x_2-(y_1 +iy_2)}, &\text{ if }q=2, \\
   - \frac{\sigma \cdot (x-y)}{\pi|x-y|^3} (\sigma \cdot \nu(y)), &\text{ if } q=3. \end{cases}
\end{equation*}
Note that for $q=2$ the operator $\mathcal{R}$ is a multiple of the Cauchy transform, while for $q=3$ the map $\mathcal{R}$ is closely related to the Riesz transform. We define also the formal adjoint of $\mathcal{R}$ w.r.t. the inner product in $L^2(\Sigma; \mathbb{C}^{N/2})$ by
\begin{equation*} 
  \mathcal{R}^* \varphi (x) = \lim_{\varepsilon \searrow 0} \int_{\Sigma \setminus B(x,\varepsilon)} r(y,x)^* \varphi(y) d \sigma(y).
\end{equation*}
For any $s \in [-\frac{1}{2}, \frac{1}{2}]$ one can show that $\mathcal{R}$ and $\mathcal{R}^*$ give rise to  bounded operators in $H^s(\Sigma; \mathbb{C}^{N/2})$ that satisfy for all $\varphi \in H^s(\Sigma; \mathbb{C}^{N/2})$ and $\psi \in H^{-s}(\Sigma; \mathbb{C}^{N/2})$ the relation
\begin{equation*}
  \langle \mathcal{R} \varphi, \psi \rangle_{H^s(\Sigma; \mathbb{C}^{N/2}) \times H^{-s}(\Sigma; \mathbb{C}^{N/2})}
  = \langle \varphi, \mathcal{R}^* \psi \rangle_{H^s(\Sigma; \mathbb{C}^{N/2}) \times H^{-s}(\Sigma; \mathbb{C}^{N/2})},
\end{equation*}
see \cite{H22}, and we denote their realizations in $H^{-1/2}(\Sigma; \mathbb{C}^{N/2})$ by $\widetilde{\mathcal{R}}$ and $\widetilde{\mathcal{R}}^*$, respectively. Then we have the following result \cite{H22}.

\begin{proposition} \label{proposition_Riesz_transform}
  Let $\mathcal{R}$ and $\mathcal{R}^*$ be defined as above and let $\mathcal{M}(z)$, $z \in (-m,m)$, be the Weyl function associated with the boundary triple in Theorem~\ref{theorem_OBT}. Then the following holds:
  \begin{itemize}
    \item[(i)] There exists a self-adjoint and compact operator $\mathcal{K}_1$ in $L^2(\Sigma; \mathbb{C}^N)$ such that
    \begin{equation*}
      \Theta-\mathcal{M}(z) = \Lambda \begin{pmatrix} (\eta + \tau)I_{N/2} & \lambda I_{N/2} + \widetilde{\mathcal{R}}^* \\ \lambda I_{N/2} + \widetilde{\mathcal{R}} & (\eta - \tau)I_{N/2} \end{pmatrix} \Lambda - 4 \begin{pmatrix} (z - m)I_{N/2} & 0 \\ 0 & (z+m)I_{N/2} \end{pmatrix} + \mathcal{K}_1.
    \end{equation*} 
        \item[(ii)] The operator $\widetilde{\mathcal{R}} - \widetilde{\mathcal{R}}^*: H^{-1/2}(\Sigma; \mathbb{C}^{N/2}) \rightarrow H^{1/2}(\Sigma; \mathbb{C}^{N/2})$ is bounded.
    \item[(iii)] One has $\mathcal{R}^2  = 4 I_{N/2}$. In particular, this implies that 
    \begin{equation*}
      \widetilde{\mathcal{R}}  \widetilde{\mathcal{R}}^* - 4 I_{N/2} = \widetilde{\mathcal{R}}(\widetilde{\mathcal{R}}^* - \widetilde{\mathcal{R}}): H^{-1/2}(\Sigma; \mathbb{C}^{N/2}) \rightarrow H^{1/2}(\Sigma; \mathbb{C}^{N/2})
    \end{equation*}
    is bounded.
    \item[(iv)] There exist a compact operator $\mathcal{K}_2: H^{-1/2}(\Sigma; \mathbb{C}^{N/2}) \rightarrow H^{1/2}(\Sigma; \mathbb{C}^{N/2})$ and closed subspaces $\mathcal{H}_\pm$ of $H^{-1/2}(\Sigma; \mathbb{C}^{N/2})$ satisfying $\dim \mathcal{H}_+ = \dim \mathcal{H}_-$, $H^{-1/2}(\Sigma; \mathbb{C}^{N/2}) = \mathcal{H}_+ \dot{+} \mathcal{H}_-$, and $\mathcal{H}_\pm \not\subset H^s(\Sigma; \mathbb{C}^{N/2})$ for any $s > -\frac{1}{2}$ such that the realization of $\widetilde{\mathcal{R}} + \widetilde{\mathcal{R}}^*$ in the space $H^{-1/2}(\Sigma; \mathbb{C}^{N/2})$ can be written as
    \begin{equation*}
      \widetilde{\mathcal{R}} + \widetilde{\mathcal{R}}^* = 4 P_+ - 4 P_- + \mathcal{K}_2,
    \end{equation*}
    where $P_\pm$ is the projection onto $\mathcal{H}_\pm$. Moreover, for any $\varphi \in H^{1/2}(\Sigma; \mathbb{C}^{N/2})$ one has $P_\pm \varphi \in H^{1/2}(\Sigma; \mathbb{C}^{N/2})$.
  \end{itemize}
\end{proposition}

Now, we are prepared to prove the self-adjointness of $B_{\eta, \tau, \lambda}$. We remark that the result in the non-critical case $(\frac{d}{4} - 1)^2 - \lambda^2 \neq 0$ has been obtained in space dimension $q=2$ with a different proof in \cite{CLMT21}.

\begin{theorem} \label{theorem_self_adjoint_critical}
  The operator $B_{\eta, \tau, \lambda}$ is self-adjoint in $L^2(\mathbb{R}^q; \mathbb{C}^N)$ and the following holds:
  \begin{itemize}
    \item[(i)] If $(\frac{d}{4} - 1)^2 - \lambda^2 \neq 0$, then $\dom B_{\eta, \tau, \lambda} \subset H^1(\mathbb{R}^q \setminus \Sigma; \mathbb{C}^N)$.
    \item[(ii)] If $(\frac{d}{4} - 1)^2 - \lambda^2 = 0$, then $\dom B_{\eta, \tau, \lambda} \not\subset H^s(\mathbb{R}^q \setminus \Sigma; \mathbb{C}^N)$ for all $s>0$.
  \end{itemize}
\end{theorem}
\begin{proof}
   In view of Theorem~\ref{theorem_Krein_abstract} and~\eqref{equation_B_BT} the operator $B_{\eta, \tau, \lambda}$ is self-adjoint in $L^2(\mathbb{R}^q; \mathbb{C}^N)$ if and only if $\Theta$ is self-adjoint in $L^2(\Sigma; \mathbb{C}^N)$. Using Proposition~\ref{proposition_Riesz_transform} and $\mathcal{M}(0)=0$ we see that this is the case if and only if the operator
   \begin{equation*}
     \begin{split}
       \Xi \varphi &= \Lambda \begin{pmatrix} (\eta + \tau)I_{N/2} & \lambda I_{N/2} + \widetilde{\mathcal{R}}^* \\ \lambda I_{N/2} + \widetilde{\mathcal{R}} & (\eta - \tau) I_{N/2} \end{pmatrix} \Lambda \varphi, \\
       \dom \Xi &= \left\{ \varphi \in L^2(\Sigma; \mathbb{C}^N): \Lambda \begin{pmatrix} (\eta + \tau) I_{N/2} & \lambda I_{N/2} + \widetilde{\mathcal{R}}^* \\ \lambda I_{N/2} + \widetilde{\mathcal{R}} & (\eta - \tau) I_{N/2} \end{pmatrix} \Lambda \varphi \in L^2(\Sigma; \mathbb{C}^N) \right\},
     \end{split}
   \end{equation*}
   is self-adjoint in $L^2(\Sigma; \mathbb{C}^N)$, as $\Theta - \Xi$ is a bounded and self-adjoint operator in $L^2(\Sigma; \mathbb{C}^N)$. We are going to prove the self-adjointness of $\Xi$ for $\eta + \tau \neq 0$, the other cases will be commented at the end of the proof. We will make use of the restriction
  \begin{equation} \label{def_Theta_1}
    \Xi_1 := \Xi \upharpoonright H^1(\Sigma; \mathbb{C}^N).
  \end{equation}
  Using standard arguments as, e.g., in Step~1 and Step~3 in the proof of \cite[Lemma~5.4]{BH20}, it is not difficult to show that $\Xi_1$ is symmetric, $\Xi_1^* \subset \Xi$, and that $\Xi$ is closed. So in order to show the self-adjointness of $\Xi$, it suffices to prove $\Xi \subset \overline{\Xi_1}$. For that, we consider the Schur complements associated with $\Xi$ and $\Xi_1$, respectively. Recall that we assume $\eta + \tau \neq 0$. We claim that the representation
  \begin{equation} \label{Schur_complement_decomposition}
    \begin{split}  
      \Xi = \begin{pmatrix} I_{N/2} & 0 \\ (\eta + \tau)^{-1} \Lambda (\lambda I_{N/2} + \widetilde{\mathcal{R}}) \Lambda^{-1} & I_{N/2} \end{pmatrix}  \begin{pmatrix} (\eta+\tau) \Lambda^2 & 0 \\ 0 & \mathfrak{s} \end{pmatrix}\qquad \qquad \qquad & \\
      \begin{pmatrix} I_{N/2} & (\eta+\tau)^{-1} \Lambda^{-1} (\lambda I_{N/2} + \widetilde{\mathcal{R}}^*) \Lambda \\ 0 & I_{N/2} \end{pmatrix}&
    \end{split}
  \end{equation}
  with the maximal realization of the Schur complement
  \begin{equation*}
    \begin{split}
      \mathfrak{s} \varphi &= (\eta - \tau) \Lambda^2 - (\eta+\tau)^{-1} \Lambda (\lambda I_{N/2} + \widetilde{\mathcal{R}}) (\lambda I_{N/2} + \widetilde{\mathcal{R}}^*) \Lambda \varphi, \\
      \dom \mathfrak{s} &= \big\{ \varphi \in L^2(\Sigma; \mathbb{C}^{N/2}): \\
      & \qquad \big( (\eta - \tau) \Lambda^2 - (\eta+\tau)^{-1} \Lambda (\lambda I_{N/2} + \widetilde{\mathcal{R}}) (\lambda I_{N/2} + \widetilde{\mathcal{R}}^*) \Lambda \big) \varphi \in L^2(\Sigma; \mathbb{C}^{N/2}) \big\},
    \end{split}
  \end{equation*}
  holds and that the domain of $\Xi$ can be characterized as
  \begin{equation} \label{dom_Xi}
    \begin{split}
      \dom \Xi = \big\{ \varphi = &(\varphi_1,\varphi_2) \in L^2(\Sigma; \mathbb{C}^N): \varphi_2\in \dom \mathfrak{s}, \\
      & \varphi_1 + (\eta+\tau)^{-1} \Lambda^{-1} (\lambda I_{N/2} + \widetilde{\mathcal{R}}^*) \Lambda \varphi_2 \in H^1(\Sigma; \mathbb{C}^{N/2}) \big\}.
    \end{split}
  \end{equation}
  Indeed, the right hand side in~\eqref{dom_Xi} is contained in $\dom \Xi$, as the application of the product on the right hand side in~\eqref{Schur_complement_decomposition} on  elements in this set gives an element in $L^2(\Sigma; \mathbb{C}^N)$ and coincides with the application of $\Xi$. On the other hand, as the first factor on the right hand side in~\eqref{Schur_complement_decomposition} is a bijective mapping in $L^2(\Sigma; \mathbb{C}^N)$ and $\dom \Lambda^2 = H^1(\Sigma; \mathbb{C}^{N/2})$, also the second inclusion in~\eqref{dom_Xi} holds.
  Using Proposition~\ref{proposition_Riesz_transform} we see that
  \begin{equation} \label{Schur_complement}
    \begin{split}
      \mathfrak{s} &= (\eta + \tau)^{-1} \Lambda \big((\eta^2 - \tau^2 - \lambda^2) I_{N/2} - \widetilde{\mathcal{R}} \widetilde{\mathcal{R}}^* - \lambda(\widetilde{\mathcal{R}} + \widetilde{\mathcal{R}}^*)\big) \Lambda \\
      &=(\eta + \tau)^{-1} \Lambda \big((\eta^2 - \tau^2 - \lambda^2 - 4) I_{N/2} - 4 \lambda (P_+ - P_-)\big) \Lambda+ \mathcal{A} \\
      &=(\eta + \tau)^{-1} \Lambda \big((d - 4 - 4 \lambda )P_+  + (d - 4 + 4 \lambda ) P_-\big) \Lambda+ \mathcal{A},
    \end{split}
  \end{equation}
  where $\mathcal{A}$ is a bounded operator in $L^2(\Sigma; \mathbb{C}^{N/2})$ and $P_\pm$ are the projections onto the spaces $\mathcal{H}_\pm$ in Proposition~\ref{proposition_Riesz_transform}~(iv). Under the assumption in~(i) the last equation implies for $\varphi = (\varphi_1, \varphi_2) \in \dom \Xi$ that 
  \begin{equation*}
    (d - 4 - 4 \lambda )P_+ \Lambda \varphi_2 + (d - 4 + 4 \lambda ) P_- \Lambda \varphi_2 \in H^{1/2}(\Sigma; \mathbb{C}^{N/2}).
  \end{equation*}
  Applying $P_\pm$, which is by Proposition~\ref{proposition_Riesz_transform} a map in $H^{1/2}(\Sigma; \mathbb{C}^{N/2})$, to this equation yields $P_\pm \Lambda \varphi_2 \in H^{1/2}(\Sigma; \mathbb{C}^{N/2})$, which gives finally $\varphi_2 \in H^1(\Sigma; \mathbb{C}^{N/2})$. Taking the mapping properties of $\Lambda$ and $\widetilde{\mathcal{R}}^*$ into account, the second condition in~\eqref{dom_Xi} shows for $\varphi = (\varphi_1, \varphi_2) \in \dom \Xi$ that $\varphi_1 \in H^1(\Sigma; \mathbb{C}^{N/2})$. Hence, we obtain $\dom \Xi = H^1(\Sigma; \mathbb{C}^{N})$ and thus, $\Xi = \Xi_1$.

  On the other hand, under the assumption in~(ii) at least one of the terms with $P_\pm$ in the last line in~\eqref{Schur_complement} is zero. For instance, if $\lambda = \frac{d}{4}-1$, then there is no condition on $P_+ \Lambda \varphi$ in $\dom \mathfrak{s}$, i.e. any $\varphi \in L^2(\Sigma; \mathbb{C}^{N/2})$ with $P_- \Lambda \varphi \in H^{1/2}(\Sigma; \mathbb{C}^{N/2})$ belongs to $\dom \mathfrak{s}$.  Since $\mathcal{H}_+ \not\subset H^s(\Sigma; \mathbb{C}^{N/2})$ for any $s > -\frac{1}{2}$ by Proposition~\ref{proposition_Riesz_transform}~(iv), this implies that $\dom \mathfrak{s} \not\subset H^s(\Sigma; \mathbb{C}^{N/2})$ for any $s > 0$ and thus, also $\dom \Xi \not\subset H^s(\Sigma; \mathbb{C}^N)$ for any $s > 0$. Clearly, a similar statement holds, if $\lambda = -\frac{d}{4} + 1$.
  
  To see now that $\Xi \subset \overline{\Xi_1}$ also under the assumption in~(ii), we use \cite[Theorem~2.2.14]{Tretter}, which implies that 
  \begin{equation*} 
    \begin{split}
      \dom \overline{\Xi_1} = \big\{ \varphi = (\varphi_1,\varphi_2) &\in L^2(\Sigma; \mathbb{C}^N): \varphi_2\in \dom \overline{\mathfrak{s}_1}, \\
      & \varphi_1 + (\eta+\tau)^{-1} \Lambda^{-1} (\lambda I_{N/2} + \widetilde{\mathcal{R}}^*) \Lambda \varphi_2 \in H^1(\Sigma; \mathbb{C}^{N/2}) \big\},
    \end{split}
  \end{equation*}
  where $\mathfrak{s}_1$ is the Schur complement associated with $\Xi_1$ and given by $\mathfrak{s}_1 = \mathfrak{s} \upharpoonright H^1(\Sigma; \mathbb{C}^{N/2})$. 
  Suppose that $\lambda = \frac{d}{4}-1$; the arguments for $\lambda = -\frac{d}{4}+1$ follow the same lines.
  Let $\varphi \in \dom \mathfrak{s}$ and choose $\varphi_{n} \subset H^1(\Sigma; \mathbb{C}^{N/2})$ such that $P_+ \Lambda \varphi_{n}$ converge to $P_+ \Lambda \varphi$ in $H^{-1/2}(\Sigma; \mathbb{C}^{N/2})$ and $P_- \Lambda \varphi_{n}=P_- \Lambda \varphi$. Then, $\varphi_{n} \rightarrow \varphi$ in $L^2(\Sigma; \mathbb{C}^{N/2})$ and by \eqref{Schur_complement} also $\mathfrak{s}_1 \varphi_{n} \rightarrow \mathfrak{s} \varphi$. Hence, $\varphi \in \dom \overline{\mathfrak{s}_1}$ and $\overline{\mathfrak{s}_1} \varphi = \mathfrak{s} \varphi$ which shows that also in this case $\mathfrak{s} \subset \overline{\mathfrak{s}_1}$ and then with~\eqref{dom_Xi} that $\Xi \subset \overline{\Xi_1}$. Hence, in all cases $\Xi$ and thus also $\Theta$ is self-adjoint in $L^2(\Sigma; \mathbb{C}^N)$, which implies with Theorem~\ref{theorem_Krein_abstract} the self-adjointness of $B_{\eta,\tau,\lambda}$.
  
  It remains to see the statements on the Sobolev regularity in $\dom B_{\eta,\tau,\lambda}$. For this purpose, we use that $\beta(z) \varphi \in H^s_\alpha(\mathbb{R}^q \setminus \Sigma; \mathbb{C}^N)$ if and only if $\varphi \in H^s(\Sigma; \mathbb{C}^N)$ holds for any $s \in [0,1]$, where $\beta(z)$ denotes the $\gamma$-field associated with the boundary triple $\{ L^2(\Sigma; \mathbb{C}^N), \Upsilon_0, \Upsilon_1 \}$, cf. Theorem~\ref{theorem_OBT}. Indeed, if $\varphi \in H^s(\Sigma; \mathbb{C}^N)$, then the mapping properties of $\Phi_z$ and $\mathcal{C}_z$ in Theorem~\ref{theorem_QBT}, Proposition~\ref{proposition_extensions_Phi_C}, and Corollary~\ref{corollary_extensions_Phi_C}, and $\Lambda$ imply that 
  \begin{equation*}
    \beta(z) \varphi = -4 \widetilde{\Phi}_z (\alpha \cdot \nu) \widetilde{\mathcal{C}}_z (\alpha \cdot \nu) V^* \Lambda  \varphi \in H^s_\alpha(\mathbb{R}^q \setminus \Sigma; \mathbb{C}^N).
  \end{equation*}
  Conversely, if $\beta(z) \varphi \in H^s_\alpha(\mathbb{R}^q \setminus \Sigma; \mathbb{C}^N)$, then the mapping properties of the trace operator in Lemma~\ref{lemma_trace_theorem} and Corollary~\ref{corollary_trace_extension} and Proposition~\ref{proposition_extensions_Phi_C}~(iii) show that
  \begin{equation*}
    \gamma_D^+ (\beta(z) \varphi)_+ + \gamma_D^- (\beta(z) \varphi)_-
    = -8 (\widetilde{\mathcal{C}}_z (\alpha \cdot \nu))^2 V^* \Lambda  \varphi = 2 V^* \Lambda  \varphi
    \in H^{s-1/2}(\Sigma; \mathbb{C}^N).
  \end{equation*}
  This implies that $\Lambda  \varphi \in H^{s-1/2}(\Sigma; \mathbb{C}^N)$, i.e. $\varphi \in H^{s}(\Sigma; \mathbb{C}^N)$. 
  
  Recall that for $z \in \mathbb{C} \setminus \mathbb{R}$ the formula
  \begin{equation} \label{equation_krein_application}
    (B_{\eta, \tau, \lambda} - z)^{-1} = (A_\infty - z)^{-1} + \beta(z) \big( \Theta - \mathcal{M}(z) \big)^{-1} \beta(\overline{z})^*
  \end{equation}
  holds, see Theorem~\ref{theorem_Krein_abstract}. Hence, under the assumption in item~(i) we conclude that $\dom A_\infty \subset H^1(\mathbb{R}^q \setminus \Sigma; \mathbb{C}^N)$ and $\dom \Theta = H^1(\Sigma; \mathbb{C}^N)$ imply 
  \begin{equation*}
    \dom B_{\eta, \tau, \lambda}  \subset \dom A_\infty + \beta(z) \big( H^1(\Sigma; \mathbb{C}^N) \big) \subset H^1(\mathbb{R}^q \setminus \Sigma; \mathbb{C}^N).
  \end{equation*}
  With similar arguments we see that, under the assumption of statement~(ii), the relation $\dom \Theta \not\subset H^s(\Sigma; \mathbb{C}^N)$ implies $\dom B_{\eta,\tau,\lambda} \not\subset H^s(\mathbb{R}^q \setminus \Sigma; \mathbb{C}^N)$ for any $s>0$. This finishes the proof of the theorem in the case $\eta+\tau\neq 0$.

  If $\eta = -\tau \neq 0$, then a similar strategy as in the case $\eta+\tau\neq0$ leads to the conclusion of the claimed results. Indeed, in this case one can make a similar construction as in~\eqref{Schur_complement_decomposition}, but with the Schur complement being in the  upper left corner in the operator matrix in the middle of the product in~\eqref{Schur_complement_decomposition} instead of the lower right corner. 
  
  Finally, assume $\eta=\tau=0$ and $\lambda \neq 0$. Using Proposition~\ref{proposition_Riesz_transform}~(iv) one sees that $\Xi$ can be written as
    \begin{equation*}
    \begin{split}
      \Xi &= \Lambda \begin{pmatrix} 0 & \lambda I_{N/2} + \widetilde{\mathcal{R}}^* \\ \lambda I_{N/2} + \widetilde{\mathcal{R}} & 0 \end{pmatrix} \Lambda \\
      &= \Lambda \begin{pmatrix} 0 & \lambda I_{N/2} + \frac{1}{2} (\widetilde{\mathcal{R}} + \widetilde{\mathcal{R}}^*) + \frac{1}{2}(\widetilde{\mathcal{R}}^* - \widetilde{\mathcal{R}}) \\ \lambda I_{N/2}+ \frac{1}{2} (\widetilde{\mathcal{R}} + \widetilde{\mathcal{R}}^*) + \frac{1}{2}(\widetilde{\mathcal{R}} - \widetilde{\mathcal{R}}^*) & 0 \end{pmatrix} \Lambda \\
      &= \Lambda \begin{pmatrix} 0 & (\lambda + 2)P_+ + (\lambda - 2) P_- \\ (\lambda +2) P_+ + (\lambda-2)P_- & 0 \end{pmatrix} \Lambda +\mathcal{B},
    \end{split}
  \end{equation*}
  where $\mathcal{B}$ is a bounded operator in $L^2(\Sigma; \mathbb{C}^N)$.
  The representation in the third line of the last equation and Proposition~\ref{proposition_Riesz_transform} show as above that for $\lambda \neq \pm 2$ one has $\dom \Xi = H^1(\Sigma; \mathbb{C}^N)$ and thus, $\Xi = \Xi_1$ is self-adjoint.  
  On the other hand, for $\lambda = \pm 2$ one has no restrictions for one of the spaces $\mathcal{H}_\pm$, which implies with a similar argument as before in the case $\eta + \tau \neq 0$ that $\overline{\Xi_1} = \Xi$ is again self-adjoint, but $\dom \Xi \not\subset H^s(\Sigma; \mathbb{C}^N)$ for any $s>0$.
  Moreover, one finds in a similar way as above with the help of Krein's resolvent formula in~\eqref{equation_krein_application} the claimed results about the Sobolev regularity in $\dom B_{0,0,\lambda}$. This finishes the proof of this theorem.
\end{proof}

\begin{remark}
  Let $(\frac{d}{4} - 1)^2 - \lambda^2 = 0$. Then, the representation of $\dom \Theta = \dom \Xi$ in~\eqref{dom_Xi} and~\eqref{Schur_complement} together with Krein's resolvent formula in~\eqref{equation_krein_application} gives a way how the singular part in $\dom B_{\eta, \tau, \lambda}$ can be characterized with the help of the projections $P_\pm$ from Proposition~\ref{proposition_Riesz_transform}.
\end{remark}

\begin{remark}
  Note that for $(\frac{d}{4} - 1)^2 - \lambda^2 = 0$ and $\eta, \lambda \neq 0$ the assumptions in \cite[Definition~3.1]{Ben22} are not fulfilled due to the appearance of the projections $P_\pm$. Hence, Theorem~\ref{theorem_self_adjoint_critical} is not contained in the results in \cite{Ben22}.
\end{remark}

In the non-critical case, i.e. if $(\frac{d}{4} - 1)^2 - \lambda^2 \neq 0$, we already have from Theorem~\ref{selfadjoint_theorem} and Corollary~\ref{selfadjoint_cor_2} a complete picture of the qualitative spectral properties of $B_{\eta, \tau, \lambda}$. In particular, a Krein-type resolvent formula and a Birman-Schwinger principle are established there. In the following proposition similar formulas are provided also in the critical case. Recall that $A_0$ is the free Dirac operator defined in~\eqref{def_A_0}, $P_{\eta, \tau, \lambda}$ is introduced in~\eqref{def_P_eta_tau_lambda}, $\widetilde{\mathcal{C}}_z$ and $\widetilde{\Phi}_z$ are the extensions of $\mathcal{C}_z$ and $\Phi_z$ from Proposition~\ref{proposition_extensions_Phi_C}, $\mathcal{M}(z)$ is given as in Theorem~\ref{theorem_OBT}, and $\Theta$ is constituted by~\eqref{def_Theta}. For similar results when $\lambda=0$ we refer to \cite[Theorem~5.7]{BH20}, \cite[Theorems~4.11 and 4.13]{BHOP20}, and \cite[Proposition~4.1 and Lemma~4.1]{Ben21}.

\begin{proposition} \label{proposition_Krein_Birman_Schwinger_critical}
  Let $\eta, \tau, \lambda \in \mathbb{R}$ such that $(\frac{d}{4} - 1)^2 - \lambda^2 = 0$. Then, the following holds:
  \begin{itemize}
    \item[(i)] For all $z \in \rho(A_0) \cap \rho(B_{\eta,\tau,\lambda})$ the map $I_N +  P_{\eta,\tau,\lambda} \widetilde{\mathcal{C}}_z$ admits a bounded inverse from $H^{1/2}(\Sigma; \mathbb{C}^{N})$ to $H^{-1/2}(\Sigma; \mathbb{C}^N)$ and 
	\begin{equation*} 
	(B_{\eta,\tau,\lambda} - z)^{-1} = (A_0 - z)^{-1} - \widetilde{\Phi}_z(I_N +  P_{\eta,\tau,\lambda} \widetilde{\mathcal{C}}_z)^{-1}  P_{\eta,\tau,\lambda} \Phi_{\overline{z}}^*.
	\end{equation*}
	\item[(ii)] For $z \in (-m,m)$ one has that $z \in \sigma_\mathrm{p}(B_{\eta,\tau,\lambda})$ if and only if $-1 \in \sigma_\mathrm{p}( P_{\eta,\tau,\lambda} \widetilde{\mathcal{C}}_z)$.
	\item[(iii)] For $z \in (-m,m)$ one has that $z \in \sigma_\mathrm{ess}(B_{\eta,\tau,\lambda})$ if and only if 
	\begin{equation*}
	  0 \in \sigma_\textup{ess}(\Theta - \mathcal{M}(z)) = \sigma_\textup{ess}\big(\Lambda(B- 4 V (\alpha \cdot \nu) \widetilde{\mathcal{C}}_z (\alpha \cdot \nu) V^*)\Lambda \big).
    \end{equation*}
  \end{itemize}
\end{proposition}

\begin{remark}
  The statement in Proposition~\ref{proposition_Krein_Birman_Schwinger_critical}~(iii) can be formulated for any choice of $V$ satisfying the assumptions in Section~\ref{section_OBT}. That means for $z \in (-m,m)$ one has that $z \in \sigma_\mathrm{ess}(B_{\eta,\tau,\lambda})$ if and only if 
  \begin{equation} \label{essential_spectrum}
    0 \in \sigma_\textup{ess}(\Theta - \mathcal{M}(z)) = \sigma_\textup{ess}\big(\Lambda V(P_{\eta, \tau, \lambda} - 4 (\alpha \cdot \nu) \widetilde{\mathcal{C}}_z (\alpha \cdot \nu) ) V^* \Lambda \big).
  \end{equation}
  Indeed, the self-adjointness of $B_{\eta, \tau, \lambda}$ was shown in Theorem~\ref{theorem_self_adjoint_critical}. Hence, the abstract results in Theorem~\ref{theorem_Krein_abstract} are true for the boundary triple in Theorem~\ref{theorem_OBT} constructed with any general $V$ and the parameter $\Theta$ is given for this triple by
  \begin{equation*}
    \Theta = \Lambda V(P_{\eta, \tau, \lambda} - 4 (\alpha \cdot \nu) \widetilde{\mathcal{C}}_0 (\alpha \cdot \nu) ) V^* \Lambda 
  \end{equation*}
  defined on its maximal domain in $L^2(\Sigma; \mathbb{C}^N)$. In particular, if one sets $V = \alpha \cdot \nu$ in~\eqref{essential_spectrum}, then one recovers the result in \cite[Lemma~4.1]{Ben21} that was shown for $\lambda=0$.
\end{remark}
  
\begin{proof}[Proof of Proposition~\ref{proposition_Krein_Birman_Schwinger_critical}]
  Throughout the proof let $z \in \mathbb{C} \setminus ((-\infty, -m] \cup[m, \infty))$. First, we observe that the definitions of $\Theta$ in~\eqref{def_Theta} and $\mathcal{M}(z)$ in Theorem~\ref{theorem_OBT} yield 
  \begin{equation} \label{equation_BS}
    \begin{split}
      \Theta - \mathcal{M}(z) &= \Lambda \big( B - 4 V (\alpha \cdot \nu) \widetilde{\mathcal{C}}_z (\alpha \cdot \nu) V^* \big) \Lambda \\
      &= -4 \Lambda V \big( I_N + P_{\eta, \tau, \lambda} \widetilde{\mathcal{C}}_z \big) (\alpha \cdot \nu) \widetilde{\mathcal{C}}_z (\alpha \cdot \nu) V^* \Lambda,
    \end{split}
  \end{equation}
  where $B = V P_{\eta, \tau, \lambda} V^*$ and~\eqref{formula_C_inv_extension} were used.
  
  To show item~(i) let $z \in \rho(A_\infty) \cap \rho(B_{\eta,\tau,\lambda})$ be fixed and note that $\sigma(A_\infty)=\sigma(A_0)$ holds by Lemma~\ref{lemma_A_infty}. Then, by Theorem~\ref{theorem_Krein_abstract}~(iv) the operator $\Theta - \mathcal{M}(z)$ has a bounded inverse in $L^2(\Sigma; \mathbb{C}^N)$. Therefore, we conclude from~\eqref{equation_BS} and the mapping properties of $\Lambda, V,$ and $\alpha \cdot \nu$ that $I_N +  P_{\eta,\tau,\lambda} \widetilde{\mathcal{C}}_z$ has a bounded inverse from $H^{1/2}(\Sigma; \mathbb{C}^{N})$ to $H^{-1/2}(\Sigma; \mathbb{C}^N)$. Next,   
  one gets from Theorems~\ref{theorem_Krein_abstract}~(iii) and~\ref{theorem_OBT}  and~\eqref{equation_BS} that 
  \begin{equation} \label{resolvent1}
    \begin{split}
      (B_{\eta, \tau, \lambda} - z)^{-1} &= (A_\infty - z)^{-1} + \beta(z) (\Theta - \mathcal{M}(z))^{-1} \beta(\overline{z})^* \\
      &= (A_\infty - z)^{-1} -4 \widetilde{\Phi}_z \big( I_N + P_{\eta, \tau, \lambda} \widetilde{\mathcal{C}}_z \big)^{-1} (\alpha \cdot \nu) \mathcal{C}_{\overline{z}}^* (\alpha \cdot \nu) \Phi_{\overline{z}}^*.
    \end{split}
  \end{equation}
  To proceed, we note that the operator $A_\infty$ defined in~\eqref{def_A_infty} can be described with the quasi boundary triple $\{ L^2(\Sigma; \mathbb{C}^N), \Gamma_0^{(1)}, \Gamma_1^{(1)} \}$ from Theorem~\ref{theorem_QBT} with $s=1$ as $A_\infty = T^{(1)} \upharpoonright \ker \Gamma_1^{(1)}$. Hence, we can apply Theorem~\ref{theorem_Krein_abstract}~(iii) and get 
  \begin{equation*}
    (A_\infty - z)^{-1} = (A_0 - z)^{-1} - \Phi_z (\mathcal{C}_z)^{-1} \Phi_{\overline{z}}^* = (A_0 - z)^{-1} +4 \Phi_z (\alpha \cdot \nu) \mathcal{C}_z (\alpha \cdot \nu) \Phi_{\overline{z}}^*.
  \end{equation*}
  Combining this with~\eqref{resolvent1} and using the resolvent identity, $\mathcal{C}_{\overline{z}}^* = \mathcal{C}_z$, which holds by~\eqref{equation_diff_m} and Theorem~\ref{theorem_QBT}~(iii), and~\eqref{formula_C_inv_extension} we get 
  \begin{equation*}
    \begin{split}
      (B_{\eta, \tau, \lambda} - z)^{-1} &= (A_0 - z)^{-1} +4 \Phi_z (\alpha \cdot \nu) \mathcal{C}_z (\alpha \cdot \nu) \Phi_{\overline{z}}^* \\
      &\qquad -4 \widetilde{\Phi}_z \big( I_N + P_{\eta, \tau, \lambda} \widetilde{\mathcal{C}}_z \big)^{-1} (\alpha \cdot \nu) \mathcal{C}_z (\alpha \cdot \nu) \Phi_{\overline{z}}^* \\
      &= (A_0 - z)^{-1} - \widetilde{\Phi}_z(I_N +  P_{\eta,\tau,\lambda} \widetilde{\mathcal{C}}_z)^{-1}  P_{\eta,\tau,\lambda} \Phi_{\overline{z}}^*,
    \end{split}
  \end{equation*}
  which is the claimed formula.

  Item~(ii) is a consequence of the Birman-Schwinger principle in Theorem~\ref{theorem_Krein_abstract}~(i) and~\eqref{equation_BS}, as the maps $\Lambda, V, (\alpha \cdot \nu),$ and $\widetilde{\mathcal{C}}_z$ are bijective.  
  Finally, assertion~(iii) is a direct consequence of Theorem~\ref{theorem_Krein_abstract}~(v) and the representations of $\Theta$ in~\eqref{def_Theta} and $\mathcal{M}(z)$ in Theorem~\ref{theorem_OBT}.
\end{proof}

In the following we will study the essential spectrum of $B_{\eta, \tau, \lambda}$ in the critical case. 
In the critical confinement case, i.e. when $d=-4$ and $\lambda^2=4$, the spectral properties are fully described in Proposition~\ref{proposition_confinement_critical}. Hence, in the following we will focus on the case  $(\frac{d}{4}-1)^2-\lambda^2=0$ and $d\neq-4$.
In this situation, we will see that there may be an additional point in the essential spectrum of $B_{\eta,\tau,\lambda}$. In order to show the corresponding result, additional knowledge about the operators $\widetilde{\mathcal{R}}$ and $\widetilde{\mathcal{R}}^*$ defined in~\eqref{def_R} is necessary. Recall that $\mathcal{H}_\pm$ are the spaces introduced in Proposition~\ref{proposition_Riesz_transform}~(iv).


\begin{proposition} \label{proposition_essential_spectrum_preparation}
  For the operators $\widetilde{\mathcal{R}}, \widetilde{\mathcal{R}}^*$ and the choice of $\Lambda$ in~\eqref{def_Lambda_2d}--\eqref{def_Lambda_3d} the following holds:
  \begin{itemize}
    \item[(i)] For $q=2$, the operator $\Lambda (\widetilde{\mathcal{R}} \widetilde{\mathcal{R}}^* - 4 I_1) \Lambda$ is compact in $L^2(\Sigma; \mathbb{C})$, for any $n \in \mathbb{N}$ one has $\Lambda^n P_\pm = P_\pm \Lambda^n$, and $P_\pm$ are orthogonal projections in $H^s(\Sigma; \mathbb{C})$ for all $s \in \mathbb{R}$.
    \item[(ii)] For $q=3$, assume that there exists an open set $\Sigma_0 \subset \Sigma$ such that $\Sigma_0$ is contained in a plane. Then, there exists a sequence $(\varphi_n) \subset L^2(\Sigma; \mathbb{C}^2)$ such that $\| \varphi_n \|_{L^2(\Sigma; \mathbb{C}^2)} = 1$, $\varphi_n$ converges weakly to zero, and 
    $\Lambda (\widetilde{\mathcal{R}} \widetilde{\mathcal{R}}^* - 4 I_2) \Lambda \varphi_n \rightarrow 0$ in $L^2(\Sigma; \mathbb{C}^2)$. It is possible to choose this sequence such that $(\Lambda \varphi_n) \subset \mathcal{H}_\pm$.
    \item[(iii)] The operator $\Lambda^{-2} \widetilde{\mathcal{R}} - \widetilde{\mathcal{R}}  \Lambda^{-2}: H^{-1/2}(\Sigma; \mathbb{C}^{N/2}) \rightarrow H^{1/2}(\Sigma; \mathbb{C}^{N/2})$ is compact.
  \end{itemize}
\end{proposition}

Now, we are prepared to state and prove the following result about the essential spectrum of $B_{\eta, \tau, \lambda}$:

\begin{theorem} \label{theorem_essential_spectrum}
  Let $\eta, \tau, \lambda \in \mathbb{R}$ be such that $(\frac{d}{4}-1)^2-\lambda^2=0$ and $d \neq -4$. Then, the following holds:
  \begin{itemize}
    \item[\textup{(i)}] If $q=2$, then $\sigma_\textup{ess}(B_{\eta, \tau, \lambda}) = (-\infty,-m] \cup \big\{ -\frac{\tau}{\eta} m \big\} \cup [m,\infty)$.
    \item[\textup{(ii)}] If $q=3$, assume that  there exists an open set $\Sigma_0 \subset \Sigma$ such that $\Sigma_0$ is contained in a plane. Then  $(-\infty,-m] \cup \big\{ -\frac{\tau}{\eta} m \big\} \cup [m,\infty) \subset \sigma_\textup{ess}(B_{\eta, \tau, \lambda})$.
  \end{itemize}
\end{theorem}

We also refer to \cite{BP22} for a very recent result in a similar spirit, where the three-dimensional Dirac operator with a combination of electrostatic and Lorentz-scalar $\delta$-shell interactions supported on the boundary of a bounded $C^\infty$-domain was considered in the critical case (i.e. when $\eta, \tau \in \mathbb{R}$ with $\eta^2 - \tau^2 = 4$ and $\lambda = 0$). With the help of pseudodifferential techniques, it was shown there that a whole new interval of essential spectrum centered at $-\frac{\tau}{\eta} m$ appears in this situation. In particular, in this special case the inclusion in Theorem~\ref{theorem_essential_spectrum}~(ii) is in general not an equality. 

\begin{proof}[Proof of Theorem~\ref{theorem_essential_spectrum}]
  We note first that with the same singular sequences as in the proof of \cite[Theorem~5.7~(i)]{BH20} one shows that $(-\infty,-m] \cup [m,\infty) \subset \sigma_\textup{ess}(B_{\eta,\tau,\lambda})$. Hence, it remains to study $\sigma_\textup{ess}(B_{\eta, \tau, \lambda}) \cap (-m,m)$.  
  
  We use that $z \in \sigma_{\textup{ess}}(B_{\eta, \tau, \lambda}) \cap (-m,m)$ if and only if $0 \in \sigma_\textup{ess}(\Theta - \mathcal{M}(z))$, cf. Proposition~\ref{proposition_Krein_Birman_Schwinger_critical}~(iii). By Proposition~\ref{proposition_Riesz_transform} the latter is true if and only if $0$ belongs to the essential spectrum of 
  \begin{equation} \label{def_op_essential_spectrum}
    \hat{\Xi} = \Lambda \begin{pmatrix} (\eta + \tau) I_{N/2} & \lambda I_{N/2} + \widetilde{\mathcal{R}}^* \\ \lambda I_{N/2} + \widetilde{\mathcal{R}} & (\eta - \tau) I_{N/2} \end{pmatrix} \Lambda - 4 \begin{pmatrix} (z - m) I_{N/2} & 0 \\ 0 & (z+m)I_{N/2} \end{pmatrix}
  \end{equation}
  defined on its maximal domain $\dom \hat{\Xi} = \dom \Theta$, as $\Theta - \hat{\Xi}$ is a compact and self-adjoint operator that does not influence the essential spectrum. We note that the assumptions $(\frac{d}{4}-1)^2-\lambda^2=0$ and $d \neq -4$ imply that $|\eta| \neq |\tau|$, hence, one can choose the constant $c_\Lambda$ in~\eqref{def_Lambda_2d}--\eqref{def_Lambda_3d} sufficiently large such that $0 \in \rho((\eta+\tau)\Lambda^2 -4 (z-m)I_{N/2})$. Then, it follows from \cite[Theorem~2.4.6]{Tretter}   that $0 \in \sigma_\textup{ess}(\hat{\Xi})$ if and only if $0 \in \sigma_{\textup{ess}}(\hat{\mathfrak{s}})$, where $\hat{\mathfrak{s}}$ is the Schur complement associated with $\hat{\Xi}$ given by
  \begin{equation*}
   \begin{split}
    \hat{\mathfrak{s}} &= (\eta-\tau) \Lambda^2 - 4 (z+m) I_{N/2} \\
    & \qquad - \Lambda (\lambda I_{N/2} + \widetilde{\mathcal{R}}) \Lambda \big( (\eta+\tau) \Lambda^2 - 4 (z-m) I_{N/2} \big)^{-1} \Lambda (\lambda I_{N/2} + \widetilde{\mathcal{R}}^*) \Lambda, \\
    \dom \widehat{\mathfrak{s}} &= \big\{ \varphi \in L^2(\Sigma; \mathbb{C}^{N/2}): (\eta-\tau) \Lambda^2 \varphi - 4 (z+m) \varphi \\
    & \qquad - \Lambda (\lambda I_{N/2} + \widetilde{\mathcal{R}}) \Lambda \big( (\eta+\tau) \Lambda^2 - 4 (z-m)  I_{N/2} \big)^{-1} \Lambda (\lambda I_{N/2} + \widetilde{\mathcal{R}}^*) \Lambda \varphi \\
    & \qquad \qquad \in L^2(\Sigma; \mathbb{C}^{N/2}) \big\}.
   \end{split}
  \end{equation*}
  By iterating the resolvent identity twice, one finds for $w \in \rho (\Lambda^2)$ that
  \begin{equation*} 
    (\Lambda^2 - w)^{-1} = \Lambda^{-2} + w \Lambda^{-2} (\Lambda^2 - w)^{-1} = \Lambda^{-2} + w \Lambda^{-4} + w^2 \Lambda^{-3} (\Lambda^2 - w)^{-1} \Lambda^{-1}
  \end{equation*}
  holds. Applying this formula with $w=4 (z-m)/(\eta+\tau)$, we find that
  \begin{equation*}
    \begin{split}
      \hat{\mathfrak{s}} = (\eta-\tau) \Lambda^2 - 4 (z+m) I_{N/2} - \frac{1}{\eta+\tau} \Lambda (\lambda I_{N/2} + \widetilde{\mathcal{R}}) (\lambda I_{N/2} + \widetilde{\mathcal{R}}^*) \Lambda &\\
       -4 \frac{z-m}{(\eta+\tau)^2} \Lambda (\lambda I_{N/2} + \widetilde{\mathcal{R}}) \Lambda^{-2} (\lambda I_{N/2} + \widetilde{\mathcal{R}}^*) \Lambda + \mathcal{K}_1&,
    \end{split}
  \end{equation*}
  where 
  \begin{equation*}
    \mathcal{K}_1 = - 16\frac{(z-m)^2}{(\eta+\tau)^3} \Lambda (\lambda I_{N/2} + \widetilde{\mathcal{R}}) \Lambda^{-2} \big(  \Lambda^2 - 4 (z-m)(\eta+\tau)^{-1} \big)^{-1} (\lambda I_{N/2} + \widetilde{\mathcal{R}}^*) \Lambda
  \end{equation*}
  is due to the mapping properties of $\Lambda$ specified at the beginning of this section bounded from $L^2(\Sigma; \mathbb{C}^{N/2})$ to $H^1(\Sigma; \mathbb{C}^{N/2})$ and hence, by Rellich's embedding theorem compact in $L^2(\Sigma; \mathbb{C}^{N/2})$. Taking Propositions~\ref{proposition_Riesz_transform}~(iv) and~\ref{proposition_essential_spectrum_preparation}~(iii) into account and making a similar calculation as in~\eqref{Schur_complement}, there exist compact operators $\mathcal{K}_2, \mathcal{K}_3$ in $L^2(\Sigma; \mathbb{C}^{N/2})$ such that we can rewrite
  \begin{equation*}
    \begin{split}
      \hat{\mathfrak{s}} &= \frac{1}{\eta+\tau} \Lambda ((d - 4) I_{N/2} - 4\lambda (P_+-P_-)) \Lambda + \frac{1}{\eta+\tau} \Lambda (4 I_{N/2} - \widetilde{\mathcal{R}}\widetilde{\mathcal{R}}^*) \Lambda    \\
      & \qquad - 4 (z+m) I_{N/2}-4 \frac{z-m}{(\eta+\tau)^2} \Lambda^{-1} (\lambda I_{N/2} + \widetilde{\mathcal{R}}) (\lambda I_{N/2} + \widetilde{\mathcal{R}}^*) \Lambda + \mathcal{K}_2\\
      &= \frac{1}{\eta+\tau} \Lambda \big((d - 4) I_{N/2} - 4\lambda (P_+-P_-)\big) \Lambda + \frac{1}{\eta+\tau} \Lambda (4 I_{N/2} - \widetilde{\mathcal{R}}\widetilde{\mathcal{R}}^*) \Lambda    \\
      & \qquad - 4 (z+m) I_{N/2} -4 \frac{z-m}{(\eta+\tau)^2} \Lambda^{-1} \big((\lambda^2+4) I_{N/2} + 4\lambda (P_+-P_-) \big) \Lambda + \mathcal{K}_3.
    \end{split}
  \end{equation*}
  In the following, we consider the case when $4\lambda=d-4$, the case $4\lambda = 4-d$ can be handled similarly. For $4\lambda=d-4$, the expression for $\hat{\mathfrak{s}}$ simplifies to
  \begin{equation} \label{repr_s}
    \begin{split}
      \hat{\mathfrak{s}} &= \frac{2}{\eta+\tau} \Lambda (d - 4) P_- \Lambda - 8 \frac{\eta}{\eta+\tau}\Lambda^{-1} \left( z + \frac{\tau}{\eta}m \right) P_+ \Lambda    + \frac{1}{\eta+\tau} \Lambda (4 I_{N/2} - \widetilde{\mathcal{R}}\widetilde{\mathcal{R}}^*) \Lambda \\
      & \qquad  - 4 \Lambda^{-1} \left( z+m + \frac{z-m}{(\eta+\tau)^2} (\lambda^2-d+8) \right) P_- \Lambda + \mathcal{K}_3.
    \end{split}
  \end{equation}
  Now, there exists a sequence $(\varphi_n)$ with $(\Lambda \varphi_n) \subset \mathcal{H}_+$ such that $\| \varphi_n \|_{L^2(\Sigma; \mathbb{C}^{N/2})} = 1$, $\varphi_n$ converges weakly to zero, and 
  $\Lambda (\widetilde{\mathcal{R}} \widetilde{\mathcal{R}}^* - 4) \Lambda \varphi_n \rightarrow 0$ in $L^2(\Sigma; \mathbb{C}^{N/2})$; for $q=3$ this is stated in Proposition~\ref{proposition_essential_spectrum_preparation}~(ii), while for $q=2$ this is true as by Proposition~\ref{proposition_essential_spectrum_preparation}~(i) $\Lambda (\widetilde{\mathcal{R}} \widetilde{\mathcal{R}}^* - 4) \Lambda$ is compact in $L^2(\Sigma; \mathbb{C})$ and compact operators turn any weakly convergent sequence into a strongly convergent one. With the help of this sequence, we see that $0 \in \sigma_\textup{ess}(\hat{\mathfrak{s}})$, if $z = -\frac{\tau}{\eta} m$, which shows that $0 \in \sigma_\textup{ess}(\Theta - \mathcal{M}(-\frac{\tau}{\eta} m))$ and hence, $-\frac{\tau}{\eta} m \in \sigma_\textup{ess}(B_{\eta,\tau,\lambda})$. In particular, for $q=3$ all claims are shown. 
    
  Let $q=2$, then it remains to show that $\sigma_\textup{ess}(B_{\eta,\tau,\lambda}) \cap (-m,m) = \{ -\frac{\tau}{\eta} m \}$. 
  Assume that this is not the case. Then following the arguments at the beginning of this proof there is $z \in (-m,m) \setminus \{ -\frac{\tau}{\eta} m \}$ such that $0 \in \sigma_\textup{ess}(\hat{\mathfrak{s}})$. Thus, there exists a singular Weyl sequence, i.e. there exists a sequence $(\varphi_n) \subset \dom \hat{\mathfrak{s}}$ such that $\| \varphi_n \|_{L^2(\Sigma; \mathbb{C})} = 1$, $\varphi_n$ converges weakly to zero, and $\hat{\mathfrak{s}} \varphi_n \rightarrow 0$.  Consider first the case $4\lambda=d-4 \neq 0$.
  Since $\Lambda (4 I_1 - \widetilde{\mathcal{R}}\widetilde{\mathcal{R}}^*) \Lambda $ is compact in $L^2(\Sigma; \mathbb{C})$, $P_\pm \Lambda = \Lambda P_\pm$, and $P_\pm$ are orthogonal projections, see Proposition~\ref{proposition_essential_spectrum_preparation}~(i),  
  the representation of $\hat{\mathfrak{s}}$ in~\eqref{repr_s} can be simplified to
  \begin{equation*} 
    \begin{split}
      \hat{\mathfrak{s}} &= - 8 \frac{\eta}{\eta+\tau} \left( z + \frac{\tau}{\eta}m \right) P_+ \\
      &\qquad +P_- \left( \frac{2d - 8}{\eta+\tau} \Lambda^2 - 4 \left( z+m + \frac{z-m}{(\eta+\tau)^2} (\lambda^2-d+8) \right) \right) P_- + \mathcal{K}_4,
    \end{split}
  \end{equation*}
  where $\mathcal{K}_4$ is a compact operator in $L^2(\Sigma; \mathbb{C})$. Since $P_\pm$ are orthogonal projections in $L^2(\Sigma; \mathbb{C})$ by Proposition~\ref{proposition_essential_spectrum_preparation}~(i), these maps are bounded in $L^2(\Sigma; \mathbb{C})$, and thus, an application of $P_+$ to last displayed formula shows that $P_+ \varphi_n \rightarrow 0$. Hence, $(P_- \varphi_n)$ is a singular sequence for 
  \begin{equation*}
    \begin{split}
      \frac{2d - 8}{\eta+\tau} \Lambda^2 - 4 \left( z+m + \frac{z-m}{(\eta+\tau)^2} (\lambda^2-d+8) \right) + \mathcal{K}_4.
    \end{split}
  \end{equation*}
  However, the definition of $\Lambda$ in~\eqref{def_Lambda_2d} implies that the essential spectrum of $\Lambda^2$ is empty, which is finally a contradiction. This finishes the proof of this theorem in the case $4\lambda=d-4 \neq 0$. The case $4\lambda=-d+4 \neq 0$ can be handled similarly. Eventually, if $4\lambda=d-4 = 0$, then the expression in~\eqref{repr_s} can be simplified to
  \begin{equation*} 
    \begin{split}
      \hat{\mathfrak{s}} &= - 8 \frac{\eta}{\eta+\tau} \left( z + \frac{\tau}{\eta}m \right)  + \mathcal{K}_4,
    \end{split}
  \end{equation*}
  and hence, $z=-\frac{\tau}{\eta} m$ is the only point with $0 \in \sigma_\textup{ess}(\widehat{\mathfrak{s}})$. Therefore, we have shown in all cases that $z=-\frac{\tau}{\eta} m$ is the only possible point such that $0 \in \sigma_\textup{ess}(\widehat{\mathfrak{s}})$, i.e. $\sigma_\textup{ess}(B_{\eta,\tau,\lambda}) \cap (-m,m) = \{ -\frac{\tau}{\eta} m \}$.
\end{proof}

\end{document}